\documentclass[a4paper,twoside,10pt]{article}

\usepackage[a4paper,left=3cm,right=3cm, top=3cm, bottom=3cm]{geometry}
\usepackage[utf8]{inputenc}
\usepackage{graphicx}
\usepackage{epstopdf}
\usepackage{amsmath}
\usepackage{amsthm}
\usepackage{amssymb}
\usepackage{cite}
\usepackage{color}
\usepackage[font=footnotesize,labelfont=bf]{caption}
\usepackage[author={Lorenzo}]{pdfcomment}
\usepackage{soul}
\usepackage{algorithm}
\usepackage{comment}

\restylefloat{table}
\theoremstyle{plain}
\newtheorem{theorem}{Theorem}[section]
\newtheorem{corollary}[theorem]{Corollary}
\newtheorem{lemma}[theorem]{Lemma}
\newtheorem{proposition}[theorem]{Proposition}
\theoremstyle{remark}
\newtheorem{remark}{Remark}

\newcommand{\lm}[1]{{\color{red}#1}}

\newcommand{\Norm}[1]{{\left\|{#1} \right\|}}
\newcommand{\SemiNorm}[1]{{\left|{#1} \right|}}
\newcommand{\jump}[1]{\left[\!\left[#1\right]\!\right]}

\newcommand{\p}{p}
\newcommand{\h}{h}

\newcommand{\Pcalpbftt}{\mathcal P_{\pbft}}
\newcommand{\Pcaltildepbftt}{\widetilde{\mathcal P}_{\pbft}}
\newcommand{\Pcaltildepbftmo}{\widetilde{\mathcal P}_{\pbft-1}}

\newcommand{\Nbb}{\mathbb N}
\newcommand{\Pbb}{\mathbb P}
\newcommand{\Rbb}{\mathbb R}

\newcommand{\xbf}{\mathbf x}
\newcommand{\qpnt}{q_{\pnt}}

\newcommand{\qpkt}{q_{\pkt}}
\newcommand{\qpx}{q_{\px}}
\newcommand{\qpntmo}{q_{\pnt-1}}
\newcommand{\qpntmth}{q_{\pnt-3}}
\newcommand{\Ihat}{\widehat I}

\newcommand{\Pizpbft}{\Pi^0_{\pbft}}
\newcommand{\utildeh}{\widetilde U_{\h}}

\newcommand{\Deltax}{\Delta_{\xbf}}
\newcommand{\nablax}{\nabla_{\xbf}}
\newcommand{\QT}{Q_T}
\newcommand{\uzero}{u_0}
\newcommand{\uone}{u_1}
\newcommand{\uzeroh}{u_{0,\h}}
\newcommand{\uoneh}{u_{1,\h}}
\newcommand{\Tcalh}{\mathcal T_\h}
\newcommand{\taun}{\tau_n}
\newcommand{\taum}{\tau_m}
\newcommand{\tauk}{\tau_k}
\newcommand{\Xh}{X_{\h,\boldsymbol{\tau}}}
\newcommand{\uh}{U_{\h}}
\newcommand{\uhpr}{U_{\h}'}
\newcommand{\vh}{v_\h}
\newcommand{\Wh}{W_{\h}}
\newcommand{\uhplus}{U_{\h}^+}
\newcommand{\Whplus}{W_{\h}^+}
\newcommand{\uhminus}{U_{\h}^-}
\newcommand{\Vh}{V_{\h}}

\newcommand{\tnmo}{t_{n-1}}
\newcommand{\tmmo}{t_{m-1}}
\newcommand{\tmmotilde}{t_{\mtilde-1}}
\newcommand{\tn}{t_n}
\newcommand{\tm}{t_m}
\newcommand{\tnminus}{t_n^-}

\newcommand{\taubold}{\boldsymbol \tau}

\newcommand{\Pizpbftmo}{\Pi^0_{\pbft-1}}
\newcommand{\Pizpbftmth}{\Pi^0_{\pbft-3}}

\newcommand{\In}{I_n}
\newcommand{\Inmo}{I_{n-1}}
\newcommand{\lambdan}{\lambda_n}
\newcommand{\lambdam}{\lambda_m}
\newcommand{\mun}{\mu_n}
\newcommand{\mum}{\mu_m}

\newcommand{\eh}{e_{\h}}
\newcommand{\xih}{\xi_{\h}}
\newcommand{\PihEcal}{\Pi_\h^{\mathcal E}}
\newcommand{\fxih}{f_{\xih}}
\newcommand{\Pionepbftmo}{\Pi^{1}_{\pbft-1}}
\newcommand{\Pionepbft}{\Pi^{1}_{\pbft}}
\newcommand{\tauo}{\tau_1}
\newcommand{\px}{\p_{\xbf}}
\newcommand{\pbft}{\mathbf \p^t}
\newcommand{\pnt}{\p_n^t}
\newcommand{\pmt}{\p_m^t}
\newcommand{\pkt}{\p_k^t}
\newcommand{\pmtildet}{\p_{\widetilde m}^t}
\newcommand{\pot}{\p_1^t}

\newcommand{\U}{U}

\newcommand{\V}{V}
\newcommand{\Vpr}{V'}
\newcommand{\vB}{v_B}
\newcommand{\Uhat}{\widehat U}
\newcommand{\Vhat}{\widehat V}
\newcommand{\Vplus}{V^+}
\newcommand{\Xtau}{X_{\boldsymbol\tau}}
\newcommand{\HozOmega}{H^1_0(\Omega)}

\newcommand{\cX}{c_X}
\newcommand{\ctildeX}{\widetilde c_X}

\renewcommand{\Im}{I_m}
\newcommand{\Ik}{I_k}
\newcommand{\Imtilde}{I_{\mtilde}}
\newcommand{\mtilde}{\widetilde m}

\newcommand{\etao}{\eta_1}
\newcommand{\etatw}{\eta_2}
\newcommand{\oscf}{{\rm osc}(f)}
\newcommand{\oscnf}{{\rm osc}_n(f)}

\DeclareMathOperator{\card}{card}

\title{\normalsize{A priori and a posteriori error estimates
of a $\mathcal C^0$-in-time method for the wave equation in second order formulation}}
\author{\normalsize{Z. Dong\thanks{Inria, 48 rue Barrault, 75013 Paris, France;
CERMICS, ENPC, Institut Polytechnique de Paris, 77455 Marne-la-Vallee Cedex 2, France
{\tt zhaonan.dong@inria.fr, zuodong.wang@inria.fr}},
L. Mascotto\thanks{Dipartimento di Matematica e Applicazioni, Universit\`a degli Studi di Milano-Bicocca, 20125 Milan, Italy;
IMATI-CNR, 27100, Pavia, Italy;
Fakult\"at f\"ur Mathematik, Universit\"at Wien, 1090 Vienna, Austria
{\tt lorenzo.mascotto@unimib.it}},
Z. Wang\footnotemark[1]}}
\date{}

\begin{document}

\maketitle

\begin{abstract}
\noindent We establish fully-discrete a priori
and semi-discrete in time a posteriori
error estimates for a discontinuous-continuous Galerkin
discretization of the wave equation in second order formulation;
the resulting method is a Petrov-Galerkin scheme based on
piecewise polynomial test functions
and continuous piecewise polynomial trial functions in time, respectively.
Crucial tools in the a priori analysis for the fully-discrete formulation
are the design of suitable projection and interpolation operators
extending those used in the parabolic setting,
and stability estimates based on a nonstandard choice of the test function;
a priori estimates are shown,
which are measured in $L^\infty$-type norms in time.
For the semi-discrete in time formulation,
we exhibit reliable a posteriori error estimates
for the error measured in the $L^\infty(L^2)$ norm
with fully explicit constants;
to this aim, we design a reconstruction operator
into $\mathcal C^1$ piecewise polynomials
over the time grid with optimal approximation properties
in terms of the polynomial degree
distribution and the time steps.
Numerical examples illustrate the theoretical findings.

\medskip\noindent
\textbf{AMS subject classification}: 65M50; 65M60;  65J10.

\medskip\noindent
\textbf{Keywords}: wave equation;
discontinuous Galerkin method; adaptive algorithm; $hp$-analysis;
a posteriori error analysis.
\end{abstract}

\section{Introduction} \label{section:introduction}

We establish fully-discrete a priori
and semi-discrete in time a posteriori error estimates
for a $\mathcal C^0$-in-time method, see, e.g., \cite{Walkington:2014},
approximating solutions to the wave equation in second order formulation,
which are explicit in the spatial mesh size, the time steps,
and the polynomial degrees.

\paragraph*{Formulation, a priori estimates, and minimal literature.}
The $\mathcal C^0$-in-time method we are interested in
is based on a second order formulation of the wave equation.
Compared to several references where first order systems are considered
\cite{French-Peterson:1996, Hulbert-Hughes:1990, French:1993, Johnson:1993},
the dimension of the corresponding discrete spaces is smaller
for fixed time steps.

The method lies in between a fully DG and a
$\mathcal C^1$ schemes in time:
it employs piecewise polynomial test functions
and continuous piecewise polynomial trial functions in time, respectively;
upwind terms involving first time derivatives
in time of the trial functions
are included in the formulation.
The polynomial degree in time of the trial functions
is larger by~1 than that for the test functions,
leading to square systems for each time interval,
which are solved sequentially
as a time marching scheme.

A key tool in deriving stability and a priori error estimates
is the choice of an appropriate test function.
For instance, in \cite{Walkington:2014}, a higher order
fully-discrete version of the test function
in~\cite{Baker:1976} is used, leading to stability estimates
in the $L^{\infty}$-type norms in time.
Nonlinear problems are also discretized with this approach~\cite{Gomez-Nikolic:2025}.

\paragraph*{A posteriori estimates and minimal literature.}
A posteriori error estimates are well established for elliptic problems
and a considerable amount of work is available for parabolic problems as well;
on the other hand, hyperbolic (and in particular wave)
problems are less explored.

A posteriori error estimates for wave problems
in second order formulation are studied in~\cite{Adjerid:2002}
and later rigorously analysed in~\cite{Bernardi-Suli:2005};
first order systems are instead the topic of
\cite{Makridakis-Nochetto:2006, Johnson:1993, Suli:1999};
a posteriori error estimates measured in the $L^\infty(L^2)$ norm
are investigated for different time stepping schemes
in \cite{Georgoulis-Lakkis-Makridakis:2013, Georgoulis-Lakkis-Makridakis-Virtanen:2016, Gorynina-Lozinski-Picasso:2019}.
A posteriori error estimates
that are reliable and efficient in the same norm
are instead investigated in
\cite{ChaumontFrelet-Ern:2025, ChaumontFrelet:2023}.

\paragraph*{Features of the $\mathcal C^0$-in-time method.}
Compared to fully DG schemes in time,
the $\mathcal C^0$-in-time method involves fewer unknowns;
in view of the a posteriori error analysis
for the semi-discrete in time version,
one can employ tools from the parabolic setting
and deduce a posteriori bounds that are fully explicit with
respect to the polynomial degree.
On the other hand, a modification of the scheme
seems to be suited for dynamic mesh changes in space~\cite{Karakashian-Makridakis:2005, Dong-Georgoulis-Mascotto-Wang:2025}.

\paragraph*{First main contribution of the manuscript.}
We modify the analysis in~\cite{Walkington:2014}
so as to have explicit dependence on the polynomial degree
for the a priori analysis of the fully-discrete scheme.
We consider uniform polynomial degree in the spatial discretization
and possibly nonuniform polynomial degree in the time discretization.
Static, locally quasi-uniform meshes in space are considered,
allowing for small elements
that in explicit schemes would impact
on the CFL condition~\cite{Grote-Michel-Sauter:2021}.

Our analysis hinges upon deriving stability estimates for the scheme,
which are explicit in the polynomial degrees;
see Section~\ref{subsection:stability-p-explicit}.
Based on such stability estimates
and the properties of an integrated Thom\'ee operator
discussed in Section~\ref{subsection:integrated-Thomee},
a priori error estimates
are proven in Section~\ref{subsection:abstract-analysis}.
Error estimates
are obtained in Section~\ref{subsection:h-p-convergence},
which are explicit in the spatial mesh size, the time steps,
and the polynomial degrees.
Optimal rates for the error
measured with respect to $L^\infty$-type norms in time
is shown for sufficiently regular solutions and data
for fixed polynomial degrees in time and space.

\paragraph*{Second main contribution of the manuscript.}
For the semi-discrete in time method,
we design an error estimator satisfying
a posteriori error estimates,
with explicit dependence on the polynomial degree distribution in time.
The crucial tool in a posteriori error estimates for time-dependent
problems is the derivation of a reconstruction operator
into smoother spaces.
The original idea in the context of parabolic problems
traces back to Makridakis and Nochetto \cite{Makridakis-Nochetto:2006};
the corresponding $\p$-version analysis is detailed
in~\cite{Schoetzau-Wihler:2010}
and later in~\cite{Holm-Wihler:2018}.
We design a related operator in the wave equation setting
and derive $\p$-optimal approximation estimates
in several norms in Section~\ref{subsection:reconstruction-operator};
such an operator is instrumental
for designing an error estimator that is reliable for the error measured
in the $L^{\infty}(L^2)$ norm;
see Section~\ref{subsection:error-estimator-estimates};
in the recent work~\cite{Dong-Georgoulis-Mascotto-Wang:2025},
a posteriori error estimates in that norm are derived,
which are valid  also for dynamic mesh changes.
The upper bound is explicit in the polynomial degree
distribution and the time steps,
without unknown constants.
Since the test and trial test functions have different
polynomial degrees in time,
the a posteriori error bounds involve extra oscillation terms
compared to the parabolic setting.

While most references focus on explicit time-stepping,
we focus on implicit approaches;
they are more appropriate for the analysis
of coupled wave and parabolic problems,
such as fluid-structure interactions,
which we plan to study in future work,
and allow for elements of different sizes
for different regimes without
influencing the local time-steps.
To the best of our knowledge,
we provide for the first time in the literature fully explicit,
reliable a posteriori error estimates for a semi-discrete in time method
for the approximation of solutions to the wave equation
in second order formulation, which are explicit
in the polynomial degree distribution in time
and optimal in the time steps.
The proposed analysis does not hinge upon any CFL condition,
which is greatly advantageous for adaptivity
whilst compared to methods based on explicit time stepping;
there, for each mesh refinement,
one has to check whether the resulting spatial mesh
size is sufficiently small compared to the corresponding
time step and the polynomial degree~\cite{Grote-Michel-Sauter:2021}.
This improvement is even more effective for wave problems
involving higher order spatial elliptic operators
\cite{Nataraj-RuizBaierYousuf:2025},
where the CFL condition poses even stricter constraints
on the spatial mesh size.

\paragraph*{List of the main results and advances.}
For the reader's convenience, we list here
the main results of the manuscript
([{\textbf{APRI}}] = a priori analysis;
[{\textbf{APOS}}] = a posteriori analysis):
\begin{itemize}
    \item{} [{\textbf{APRI}}]
    Theorem~\ref{theorem:stability} is concerned with
    fully explicit stability estimates
    for the $\mathcal C^0$-in-time formulation,
    which are explicit with respect to the polynomial degrees
    in time and space employed throughout;
    \item{} [{\textbf{APRI}}]
    Proposition~\ref{proposition:Pcaltildept-L2-approx}
    analyzes the approximation
    properties of a novel integrated Thom\'ee operator;
    \item{} [{\textbf{APRI}}]
    Theorem~\ref{theorem:h-p-convergence}
    discusses a priori estimates that are explicit with respect
    to the spatial mesh size, the time step distribution,
    the spatial polynomial degree,
    and the polynomial degree distribution in time;
    \item{} [{\textbf{APOS}}] we define an
    error estimator~$\eta$ in~\eqref{error-estimator}
    for the semi-discrete in time formulation;
    \item{} [{\textbf{APOS}}] corresponding reliability estimates with respect to the
    $L^\infty(L^2)$ norm of the error
    are displayed in
    Proposition~\ref{proposition:reliability-estimate}.
\end{itemize}

\paragraph*{Functional setting.}
Standard notation is used throughout for Sobolev and Bochner spaces.
Let~$D$ be a Lipschitz domain in~$\Rbb^d$, $d=1$, $2$, and $3$, with boundary~$\partial D$.
The space of Lebesgue measurable and square integrable functions over~$D$ is~$L^2(D)$.
The Sobolev space of positive integer order~$s$ is~$H^s(D)$.
We endow~$H^s(D)$ with the inner product, seminorm, and norm
\[
(\cdot,\cdot)_{s,D},
\qquad\qquad\qquad
\SemiNorm{\cdot}_{s,D},
\qquad\qquad\qquad
\Norm{\cdot}_{s,D}.
\]
Interpolation theory is used to construct Sobolev spaces of positive noninteger order;
duality is used to define negative order Sobolev spaces.
We shall be particularly using the space~$H^{-1}(D)$,
which is the dual of~$H^1_0(D)$;
the duality pairing between the two spaces
is $\langle \cdot, \cdot \rangle$.
The space of polynomials of nonnegative degree~$\p$ over~$D$ is~$\Pbb_\p(D)$.

Given~$\mathcal X$ a real Banach space with norm~$\Norm{\cdot}_{\mathcal X}$,
an interval $I$, and~$t$ larger than or equal to~$1$,
we define $L^t(I; \mathcal X)$ as the Bochner space of measurable functions~$v$ from $I$ to~$\mathcal X$
such that the following quantity is finite:
\[
\Norm{v}_{L^t(I; \mathcal X)} :=
\begin{cases}
\left( \int_I \Norm{v}_{\mathcal X}^t \right)^{\frac1t} dt & \text{for } 1\le t < \infty \\
\text{ess\ sup}_{t \in I} \Norm{v}_{\mathcal X}                       & \text{for } t=\infty.
\end{cases}
\]
For~$s$ in~$\Nbb$, the space~$H^s(I;\mathcal X)$ is the space of measurable functions~$v$
whose derivatives in time up to order~$s$ belong to $L^2(I;\mathcal X)$.
For any real number~$s$ larger than or equal to~$0$,
the space~$H^s(I;\mathcal X)$ is constructed using interpolation
of integer order Bochner spaces.
Bochner inner products are denoted
by $(\cdot,\cdot)_{L^t(I;\mathcal X)}$ and $(\cdot,\cdot)_{H^s(I;\mathcal X)}$.

To avoid confusion, the seminorm symbol $\SemiNorm{\cdot}$
is only used to denote spatial seminorms.
Seminorms in time are rather displayed as $L^2$ norms
of a suitable time derivative.
The first and second derivative symbols
are $\cdot'$ and~$\cdot''$;
time derivatives of order $s$ larger than~$2$ 
are displayed as $\cdot^{(s)}$.

\paragraph*{The continuous problem.}
Let~$\Omega$ be a polytopic, Lipschitz domain in~$\Rbb^d$, $d=1,2,3$;
$T$ a positive final time;
$\QT:=(0,T] \times \Omega$ the space--time cylinder;
$\uzero$ in~$H_0^1(\Omega)$;
$\uone$ in~$L^2(\Omega)$;
$f$ in~$L^2(0,T; L^2(\Omega))$.

Given~$\Deltax \cdot$ and~$\nablax \cdot$
the spatial Laplace and gradient operator,
we consider the following problem:
find~$u:\QT \to \Rbb$ such that
\begin{equation} \label{strong-problem}
\begin{cases}
    u'' - \Deltax u = f      & \text{in } \QT \\
    u(t,\cdot) = 0      & \text{on } (0,T] \times \partial\Omega\\
    u(0,\cdot)     = \uzero(\cdot) & \text{on } \Omega\\
    u'(0,\cdot)    = \uone(\cdot)  & \text{on } \Omega.\\
\end{cases}
\end{equation}
Introduce the spaces
\[
X:= H^2(0,T; H^{-1}(\Omega))
\cap L^2(0,T; \HozOmega)
\cap H^1(0,T; L^2(\Omega)),
\qquad\qquad
Y:= L^2(0,T; \HozOmega),
\]
and the bilinear form on~$H^1_0(\Omega)\times H^1_0(\Omega)$
\[
a(u,v):= (\nablax u, \nablax v)_{0,\Omega} .
\]
We consider the following weak formulation of problem~\eqref{strong-problem}:
\begin{equation} \label{weak-problem}
\begin{cases}
\text{find } u \in X \text{ such that }\\
\int_{0}^{T} [\langle u'', v \rangle
+  a(u,v)]dt
= \int_{0}^{T} (f, v)_{0,\Omega} dt \qquad\qquad \forall v \in Y \\
u(0,\cdot)    = \uzero(\cdot) \text{ in } H_0^1(\Omega),
\qquad
u'(0,\cdot)    = \uone(\cdot) \text{ in } L^2(\Omega) .
\end{cases}
\end{equation}
Problem~\eqref{strong-problem} is well posed;
see, e.g., \cite[Chapter~8]{Raviart-Thomas:1983}.

In~\cite{Walkington:2014}, inhomogeneous Dirichlet
and inhomogeneous Neumann boundary conditions are considered;
This results in further complication in the analysis below,
cf. \cite[Example 3.2]{Walkington:2014}.
For this reason, we prefer to stick to the setting in~\eqref{weak-problem}.

\paragraph*{Spatial meshes, time grids, and polynomial degree distributions.}
We consider either a simplicial or tensor-product
conforming mesh~$\Tcalh$ of~$\Omega$
and a corresponding $H^1$-conforming Lagrangian finite element space~$V_h$
of uniform order~$\px$.
We assume the existence of a constant~$\gamma$ in $(0,1)$
such that each element of~$\Tcalh$ is star-shaped with respect to a ball
of radius larger than or equal to the diameter of that element;
moreover, we assume local quasi-uniformity of the mesh, i.e.,
given~$\h_1$ and~$\h_2$ the diameters
of two arbitrary \lm{neighbouring} elements with $\h_1\le \h_2$,
one has $\h_2 \le \gamma \h_1$.
Throughout, $\h$ denotes the maximum of all diameters
of the elements in~$\Tcalh$.

We further consider a decomposition~$0=t_0 < t_1 < \dots < t_N = T$
of~$[0,T]$
and introduce~$\taun:= t_n-t_{n-1}$ for all $n=1,\dots,N$.
With each time interval~$\In:=(\tnmo,\tn]$,
we associate a local polynomial degree~$\pnt$;
we collect such polynomial orders (in time) in the
vector~$\pbft$ in~$\Nbb_+^{N}$, and set $\pbft_{n}:=\pnt$.
For~$k$ in~$\mathbb Z$,
$\pbft+k$ is the vector of entries $\pnt+k$.

For all $n=1,\dots,N$,
we set $(v')^+(\tnmo,\cdot):=v'_{|\In}(\tnmo,\cdot)$
and $(v')^-(\tnmo,\cdot):=v'_{|\Inmo}(\tnmo,\cdot)$;
for all~$v$ piecewise continuous in time,
we set $v^+(\tnmo,\cdot):=v_{|\In}(\tnmo,\cdot)$
and $v^-(\tnmo,\cdot):=v_{|\Inmo}(\tnmo,\cdot)$.

We define the tensor product space
\[
\Pbb_{\pnt}(\In; \Vh)
:= \{ \Wh \in L^2(\In;\Vh) \mid \Wh = a(x) b(t),
    \; a\in\Vh,\ b\in \Pbb_{\pnt}(\In) \}.
\]
Throughout we assume that
\[
\pnt\ge 2
\qquad\qquad\qquad
\forall n=1,\dots,N.
\]

\paragraph*{The fully-discrete $\mathcal C^0$-in-time method.}
Let~$\uzeroh$ and~$\uoneh$ be approximations
of~$\uzero$ and~$\uone$ in~$\Vh$.
Throughout, we pick~$\uzeroh$
as the elliptic projection of~$\uzero$
defined in display~\eqref{elliptic-projector} below;
$\uoneh$ as the $L^2$-orthogonal projection of~$\uone$ onto~$\Vh$.
Other variants are possible, but are omitted here;
if these projections are not considered,
we may get estimates with nonsharp constants on occasions, e.g.,
in Theorem~\ref{theorem:stability} below.

We define
\[
\Xh :=  \{ \uh \in \mathcal C^0 (0,T; \Vh) \mid
                \uh(0,\cdot) = \uzeroh , \;
                \uh{}_{|\In} \in \Pbb_{\pnt}(\In; \Vh)
                \; \forall n =1,\dots,N\}
\]
and the upwind jump operator for the time derivative as
\[
\jump{\uh'}(\tnmo,\cdot)=
\begin{cases}
\uh'{}_{|I_1}(0,\cdot) - \uoneh(\cdot)
        & \text{if } n=1 \\
\uh'{}_{|\In}(\tnmo,\cdot) -\uh'{}_{|\Inmo}(\tnmo,\cdot)
        & \text{if } n=2,\dots,N .
\end{cases}
\]
The $\mathcal C^0$-in-time method,
see, e.g., \cite{Walkington:2014}, reads as follows:
find~$\uh$ in~$\Xh$ such that
\begin{equation} \label{Walkington:method}
\begin{split}
& \int_{\In} [(\uh'',\Wh)_{0,\Omega} + a(\uh,\Wh) ] dt
 + (\jump{\uh'}(\tnmo,\cdot),\Whplus(\tnmo,\cdot))_{0,\Omega}\\
& \qquad   = (f,\Wh)_{L^2(\In;L^2(\Omega))}
 \qquad\qquad\qquad \forall \Wh \in \Pbb_{\pnt-1}(\In; \Vh),
\quad \forall n=1,\dots,N .
\end{split}
\end{equation}
The initial condition~$\uzeroh$ is imposed strongly in~$\Xh$;
the initial condition~$\uoneh$ is imposed weakly
through the upwind term at the initial time.

Method~\eqref{Walkington:method} is solved time-slab by time-slab as a time marching scheme.
Initial conditions on each time slab are assigned taking the values
of the solution at the final time of the previous slab
and upwinding the first time derivative.

The existence and uniqueness of a solution and the data of
method~\eqref{Walkington:method}
follow, e.g., assuming sufficient smoothness of the
solution to problem~\eqref{weak-problem},
showing stability estimates
as those in Theorem~\ref{theorem:stability} below
(which imply uniqueness),
and using the fact that on each time slab
the linear system to solve is square
(which entails that existence is equivalent to uniqueness).

\paragraph*{The semi-discrete in time $\mathcal C^0$-in-time method.}
Define the space
\[
\Xtau:=
\{ \U \in \mathcal C^0 (0,T; \HozOmega) \mid
                \U(0,\cdot) = \uzero , \;
                \U{}_{|\In} \in \Pbb_{\pnt}(\In; \HozOmega)\}.
\]
In Section~\ref{section:apos} below,
we prove fully explicit, reliable a posteriori error estimates
for the time semi-discrete in time version of~\eqref{Walkington:method}.
More precisely, we look for $\U$ in~$\Xtau$ such that
\begin{equation} \label{Walkington:method-time-semi-discrete}
\begin{split}
\int_{\In} [(\U'',\V)_{0,\Omega} + a(\U,\V) ] dt
& + (\jump{\U'}(\tnmo,\cdot),\Vplus(\tnmo,\cdot))_{0,\Omega}
   = (f,\V)_{L^2(0,T;L^2(\Omega))}\\
& \forall \V \in \Pbb_{\pnt-1}(\In; \HozOmega),
\quad \forall n=1,\dots,N,
\end{split}
\end{equation}
with~$\uzero$ imposed strongly in~$\Xtau$
and~$\uone$ imposed weakly through upwinding.

Let~$\Pizpbftmo$ denote the piecewise~$L^2$ projector onto~$\Pbb_{\pnt-1}(\In; L^2(\Omega))$
for all $n=1,\dots,N$.
We can replace the right-hand side of~\eqref{Walkington:method-time-semi-discrete} with the following expression:
\begin{equation} \label{rhs:semidiscrete}
(f,\V)_{L^2(\In;L^2(\Omega))}
= (\Pizpbftmo f,\V)_{L^2(\In;L^2(\Omega))}
\qquad\qquad \forall n=1,\dots,N .
\end{equation}
The well-posedness of~\eqref{Walkington:method-time-semi-discrete}
follows from standard arguments of DG time-stepping schemes;
cf. \cite[Ch.~12]{Thomee:2007}.

\paragraph*{Structure of the remainder of the paper.}
We discuss stability and error estimates
of the fully-discrete method~\eqref{Walkington:method}
in Section~\ref{section:a-priori}, which are explicit
in the spatial mesh size, the time steps,
and the polynomial degrees.
A posteriori error estimates for the semi-discrete
in time method~\eqref{Walkington:method-time-semi-discrete}
are derived in Section~\ref{section:apos}.
We assess the numerical findings with
numerical experiments in Section~\ref{section:numerical-experiments},
and draw some conclusions in Section~\ref{section:conclusions}.

\section{A priori error analysis} \label{section:a-priori}
This section is concerned with proving stability
and a priori estimates for method~\eqref{Walkington:method}:
in Section~\ref{subsection:stability-p-explicit},
we show stability estimates
following the analysis in~\cite{Walkington:2014}
by tracking the explicit dependence on the polynomial
degree distributions;
in Section~\ref{subsection:integrated-Thomee},
we introduce and show the approximation
properties of an integrated Thom\'ee-type operator
in terms of the spatial mesh size, the time steps, and the polynomial degrees,
which are instrumental in deriving abstract error estimates
in Section~\ref{subsection:abstract-analysis};
standard polynomial approximation results
yield error estimates in Section~\ref{subsection:h-p-convergence}.

\subsection{Stability estimates} \label{subsection:stability-p-explicit}
The stability of method~\eqref{Walkington:method} in certain norms
is investigated in~\cite[Theorem~4.5]{Walkington:2014}.
The main idea behind the derivation of stability estimates
is to take a suitable test function,
namely the $L^2$ projection onto the correct test space
of the time derivative of the discrete solution
times a weight mimicking an exponential function;
this idea traces back to \cite{French:1993};
cf. \cite{ChaumontFrelet:2023, ChaumontFrelet-Ern:2025}
for more recent similar results.
The reason for this is that testing only with
the time derivative of the discrete solution
would yield to stability estimates
at the time grid point only,
i.e., no global stability estimates would be available;
see \cite[eq. (4.1)]{Walkington:2014}.

The constants in the stability estimates in~\cite{Walkington:2014},
depend implicitly on the distribution~$\pbft$ of polynomial degrees in time;
for this reason, we revisit that proof so as
to carry out an explicit analysis in terms of the polynomial
distribution in time.

To this aim,
given a generic element~$g$ in the space~$\Xh$,
let~$m=m(g)$ so that~$\Im$ is the interval where
\small{
\begin{equation} \label{selecting-m}
\Norm{g'}_{L^\infty(\Im;L^2(\Omega))}^2
    + \SemiNorm{g}_{L^\infty(\Im;H^1(\Omega))}^2
= \max_{n=1}^N \Big( \Norm{g'}_{L^\infty(\In;L^2(\Omega))}^2
 + \SemiNorm{g}_{L^\infty(\In;H^1(\Omega))}^2 \Big),
\end{equation}}\normalsize
The index~$m$ depends on the choice of~$g$
and can be expected on most occasions
to be attained at the final interval.
We also define
\begin{equation} \label{mun}
\mun := \frac{1}{1024 (\pnt)^2 (2 \pnt+1)}
\qquad\qquad\qquad \forall n=1,\dots,N.
\end{equation}

\begin{theorem} \label{theorem:stability}
Let~$\uh$ be the solution to~\eqref{Walkington:method}
and~$f$ be the source term in~\eqref{strong-problem}.
Let~$\uzeroh$ be the elliptic projection of~$\uzero$
and $\uoneh$ be the $L^2$-orthogonal projection of~$\uone$ onto~$\Vh$.
The following stability estimate holds true\footnote{The norms of $\uhpr$
increase cubically in $p$ with respect to to norm of~$f$,
and with rate $p^\frac32$ with respect to the norm of the initial conditions.}:
\begin{equation} \label{discrete-stability-estimate}
\begin{split}
& \mum \left(
    \Norm{\uhpr}_{L^\infty(\Im; L^2(\Omega))}^2
    + \SemiNorm{\uh}_{L^\infty(\Im; H^1(\Omega))}^2 \right)
    + \frac14 \sum_{n=1}^m
 \Norm{\jump{\uhpr}(\tnmo,\cdot)}_{0,\Omega}^2 \\
&    \le \frac12 \left(\SemiNorm{\uzero}_{1,\Omega}^2
        + \Norm{\uone}_{0,\Omega}^2 \right)
        + \frac{\tm}{\mum}
            \Norm{f}_{L^2(0,\tm;L^2(\Omega))}^2.
\end{split}
\end{equation}
\end{theorem}
\begin{proof}
Define
\[
\lambdan := \frac{1}{4(2\pnt+1)\taun}
\qquad\qquad\qquad \forall n=1,\dots,N.
\]
Restrict the (piecewise in time) $L^2$ projector $\Pizpbftmo$
to~$\Pbb_{\pnt-1}(\In;\Vh)$
and consider the following test function already used in~\cite[Theorem~4.5]{Walkington:2014}:
\[
\Wh{}_{|\In}:= \Pizpbftmo ([1-\lambdan(t-\tnmo)] \uhpr)
\qquad\qquad \forall n=1,\dots,N.
\]
Picking~$\Wh$ as above in \eqref{Walkington:method}
and proceeding as in the proof of \cite[Theorem~4.5]{Walkington:2014},
more precisely see~\cite[eq. (4.2)]{Walkington:2014},
yield, for all $n=1,\dots,N$,
\small{\[
\begin{split}
& (1-\lambdan\taun) \frac12 
        \left( \SemiNorm{\uh(\tn,\cdot)}_{1,\Omega}^2
        + \Norm{(\uhminus)'(\tn,\cdot)}_{0,\Omega}^2\right)
    + \frac{\lambdan}{2} \left( \Norm{\uhpr}_{L^2(\In;L^2(\Omega))}^2
    + \SemiNorm{\uh}_{L^2(\In;H^1(\Omega))}^2 \right)\\
& + \frac12 \Norm{\jump{\uhpr}(\tnmo,\cdot)}_{0,\Omega}^2
  \le  \frac12 \left( \SemiNorm{\uh(\tnmo,\cdot)}_{1,\Omega}^2
    + \Norm{(\uhminus)'(\tnmo,\cdot)}_{0,\Omega}^2 \right)\\
& + (\jump{\uhpr}(\tnmo,\cdot), (\uhplus)'(\tnmo,\cdot) - \Whplus(\tnmo,\cdot))_{0,\Omega}
        + \int_{\In} [1-\lambdan(\cdot-\tnmo)] (\Pizpbftmo f, \uhpr)_{0,\Omega}dt.
\end{split}
\]}\normalsize
We estimate the last two terms on the right-hand side separately:
one involving the jump of the first derivative at~$\tnmo$;
the other involving the source term~$f$.

As for the ``jump'' term, we invoke \cite[Corollary 4.4]{Walkington:2014}
and the definition of~$\lambdan$,
use standard manipulations, and get
\[
\begin{split}
& (\jump{\uhpr}(\tnmo,\cdot), (\uhplus)'(\tnmo,\cdot) - \Whplus(\tnmo,\cdot))_{0,\Omega}\\
& \le \Norm{\jump{\uhpr}(\tnmo,\cdot)}_{0,\Omega}
        \Norm{(\uhplus)'(\tnmo,\cdot) - \Whplus(\tnmo,\cdot)}_{0,\Omega}\\
& \le \Norm{\jump{\uhpr}(\tnmo,\cdot)}_{0,\Omega}
        \lambdan \sqrt{(2\pnt+1)\taun} \Norm{\uhpr}_{L^2(\In;L^2(\Omega))}\\
& \le \sqrt{(2\pnt+1)\taun\lambdan} \Big( \frac12 \Norm{\jump{\uhpr}(\tnmo,\cdot)}_{0,\Omega}^2
                              + \frac{\lambdan}{2} \Norm{\uhpr}_{L^2(\In;L^2(\Omega))}^2  \Big) \\
& = \frac12 \Big( \frac12 \Norm{\jump{\uhpr}(\tnmo,\cdot)}_{0,\Omega}^2
    + \frac{\lambdan}{2} \Norm{\uhpr}_{L^2(\In;L^2(\Omega))}^2  \Big)
\qquad \forall n=1,\dots,N.
\end{split}
\]
As for the ``source'' term, we write
\[
\begin{split}
&\int_{\In} [1-\lambdan(\cdot-\tnmo)] (\Pizpbftmo f, \uhpr)_{0,\Omega} dt
  \le   \Norm{\Pizpbftmo f}_{L^1(\In;L^2(\Omega))} \Norm{\uhpr}_{L^\infty(\In;L^2(\Omega))}.
\end{split}
\]
We combine the three displays above:
for all $n=1,\dots,N$,
\small{\[
\begin{split}
& (1-\lambdan\taun) \frac12 \left( \SemiNorm{\uh(\tn,\cdot)}_{1,\Omega}^2
   + \Norm{(\uhminus)'(\tn,\cdot)}_{0,\Omega}^2\right) \\
& \qquad    + \frac{\lambdan}{4} \left( \Norm{\uhpr}_{L^2(\In;L^2(\Omega))}^2
    + \SemiNorm{\uh}_{L^2(\In;H^1(\Omega))}^2 \right)
 + \frac14 \Norm{\jump{\uhpr}(\tnmo,\cdot)}_{0,\Omega}^2 \\
& \le \frac12 \left( \SemiNorm{\uh(\tnmo,\cdot)}_{1,\Omega}^2
    + \Norm{(\uhminus)'(\tnmo,\cdot)}_{0,\Omega}^2\right)
    + \Norm{\Pizpbftmo f}_{L^1(\In;L^2(\Omega))} 
        \Norm{\uhpr}_{L^\infty(\In;L^2(\Omega))}.
\end{split}
\]}\normalsize
Using that $\lambdan\taun < 1/4$, and $1-\lambdan\taun  \geq0$, we simplify this inequality:
\small{\begin{equation} \label{inequality-step-n}
\begin{split}
& \frac{\lambdan}{4} \left( \Norm{\uhpr}_{L^2(\In;L^2(\Omega))}^2
    + \SemiNorm{\uh}_{L^2(\In;H^1(\Omega))}^2 \right)
 + \frac14 \Norm{\jump{\uhpr}(\tnmo,\cdot)}_{0,\Omega}^2 \\
&    \le \frac12 \left( \SemiNorm{\uh(\tnmo,\cdot)}_{1,\Omega}^2
        + \Norm{(\uhminus)'(\tnmo,\cdot)}_{0,\Omega}^2 \right)
        +    \Norm{\Pizpbftmo f}_{L^1(\In;L^2(\Omega))} \Norm{\uhpr}_{L^\infty(\In;L^2(\Omega))}.
\end{split}
\end{equation}}\normalsize
In~\cite[eq. (4.1)]{Walkington:2014}, a stability estimate
is proven taking $\uhpr$ as a test function,
which gives, for all $n=2,\dots,N$,
\footnotesize\[
\begin{split}
& \frac12 \Big( \SemiNorm{\uh(\tnmo,\cdot)}_{1,\Omega}^2
                  + \Norm{(\uhminus)'(\tnmo,\cdot)}_{0,\Omega}^2 \Big)
    + \frac14 \sum_{\ell=1}^{n-1}
        \Norm{\jump{\uhpr}(t_{\ell-1},\cdot)}_{0,\Omega}^2
\le \frac12 \Big( \SemiNorm{\uh(0,\cdot)}_{1,\Omega}^2
    + \Norm{(\uhminus)'(0,\cdot)}_{0,\Omega}^2 \Big) \\
& \qquad +  \Norm{\Pizpbftmo f}_{L^1(0,\tnmo;L^2(\Omega))}
    \Big( \frac12 \Norm{\uhpr}_{L^\infty(0,\tnmo;L^2(\Omega))}^2
    + \frac12 \SemiNorm{\uh}_{L^\infty(0,\tnmo;H^1(\Omega))}^2\Big)^{\frac12}.
\end{split}
\]\normalsize
Inserting the bound
\[
\Norm{\uhpr}_{L^\infty(\In;L^2(\Omega))}
\leq \Big( \Norm{\uhpr}_{L^\infty(\In;L^2(\Omega))}^2
+ \SemiNorm{\uh}_{L^\infty(\In;H^1(\Omega))}^2 \Big)^{\frac12}
\qquad\qquad \forall n=1,\dots,N
\]
in~\eqref{inequality-step-n},
selecting~$n$ to be~$m=m(\uh)$ as in~\eqref{selecting-m},
and
combining the three displays above give
\small{\begin{equation*}
\begin{split}
& \frac{\lambdam}{4}
    \left( \Norm{\uhpr}_{L^2(\Im; L^2(\Omega))}^2
+ \SemiNorm{\uh}_{L^2(\Im;H^1(\Omega))}^2 \right)
 + \frac14 \sum_{n=1}^m
        \Norm{\jump{\uhpr}(\tnmo,\cdot)}_{0,\Omega}^2 \\
&    \le \frac12 \left( \SemiNorm{\uh(0,\cdot)}_{1,\Omega}^2
    + \Norm{(\uhminus)'(0,\cdot)}_{0,\Omega}^2 \right) \\
&   \qquad  + 2 \Big( \sum_{n=1}^m
        \Norm{\Pizpbftmo f}_{L^1(\In;L^2(\Omega))} \Big)
        \Big( \Norm{\uhpr}_{L^\infty(\Im;L^2(\Omega))}^2
    + \SemiNorm{\uh}_{L^\infty(\Im;H^1(\Omega))}^2
        \Big)^{\frac12}.
\end{split}
\end{equation*}}\normalsize
We recall the one dimensional $L^\infty$ to $L^2$-norm
polynomial inverse inequality in \cite[eq. (3.6.4)]{Schwab:1998}:
\[
 \SemiNorm{\uh}_{L^\infty(\Im;H^1(\Omega))}^2
 \leq  \frac{32(\pmt)^2}{\taun}
    \SemiNorm{\uh}_{L^2(\Im;H^1(\Omega))}^2.
\]
Using that $2\mum = 1/(512 (\pmt)^2 (2 \pmt+1))$,
we deduce
\small{\begin{equation*}
\begin{split}
& \frac{1}{512 (\pmt)^2 (2 \pmt+1)} \left( \Norm{\uhpr}_{L^\infty(\Im;L^2(\Omega))}^2
        + \SemiNorm{\uh}_{L^\infty(\Im;H^1(\Omega))}^2\right)
 + \frac14 \sum_{n=1}^m
    \Norm{\jump{\uhpr}(\tnmo,\cdot)}_{0,\Omega}^2 \\
& \le \frac12 \left( \SemiNorm{\uh(0,\cdot)}_{1,\Omega}^2
    + \Norm{(\uhminus)'(0,\cdot)}_{0,\Omega}^2 \right) \\
& \qquad + 2 \Big( \sum_{n=1}^m
        \Norm{\Pizpbftmo f}_{L^1(\In;L^2(\Omega))} \Big)
        \Big( \Norm{\uhpr}_{L^\infty(\Im;L^2(\Omega))}^2
        + \SemiNorm{\uh}_{L^\infty(\Im;H^1(\Omega))}^2
        \Big)^{\frac12}.
\end{split}
\end{equation*}}\normalsize
Using Young's inequality
$2ab \leq \varepsilon a^2
            + \frac{b^2}{\varepsilon}$,
with $\varepsilon = 1024 (\pmt)^2 (2 \pmt+1)
=:\mum^{-1}$
for the last term in the above relation,
we infer
\small{\begin{equation*}
\begin{split}
& \mum
    \left( \Norm{\uhpr}_{L^\infty(\Im; L^2(\Omega))}^2
    + \SemiNorm{\uh}_{L^\infty(\Im; H^1(\Omega))}^2 \right)
 + \frac14 \sum_{n=1}^m \Norm{\jump{\uhpr}(\tnmo,\cdot)}_{0,\Omega}^2 \\
&    \le \frac12 \left( \SemiNorm{\uh(0,\cdot)}_{1,\Omega}^2
        + \Norm{(\uhminus)'(0,\cdot)}_{0,\Omega}^2 \right)
        + \frac{1}{\mum}  \Big( \sum_{n=1}^m
        \Norm{\Pizpbftmo f}_{L^1(\In;L^2(\Omega))} \Big)^2.
\end{split}
\end{equation*}}\normalsize
We use the H\"{o}lder inequality,
the stability of the $L^2$ projector
in the $L^2(\In)$ norm,
and the $\ell^2$ Cauchy–Schwarz inequality
in the last term on the right-hand side above:
\begin{equation} \label{reason-difference-Walkington}
\sum_{n=1}^m \Norm{\Pizpbftmo f}_{L^1(\In;L^2(\Omega))}
\leq \sum_{n=1}^m  \taun^{\frac12}
        \Norm{f}_{L^2(\In;L^2(\Omega))}
\le \tm^{\frac12} \Norm{f}_{L^2(0,\tm;L^2(\Omega))}.
\end{equation}
The assertion follows by using the stability
of the discrete initial conditions~$\uzeroh$
and~$\uoneh$ with respect to the $H^1_0(\Omega)$ and $L^2(\Omega)$ norms.
\end{proof}

\begin{remark} \label{remark:norms-stability}
The norms appearing on the left-hand side of~\eqref{discrete-stability-estimate}
are of the same sort of those in~\cite[Theorem 4.5]{Walkington:2014}.
Instead, that on the right-hand side
involving the source term~$f$ differs a bit:
here, we employ an $L^2(L^2)$-type norm weighted
with $\tm^{\frac12} / \mum^\frac12$,
which scales exactly as the $L^1(L^2)$ norm used by Walkington.
The reason for this change
resides in inequality~\eqref{reason-difference-Walkington};
by proceeding as in~\cite{Walkington:2014},
we would end up with further suboptimality by two orders
in the polynomial degree due to the $L^1$ stability of the $L^2$ projector.
\end{remark}

\subsection{An integrated Thom\'ee-type operator} \label{subsection:integrated-Thomee}
Given a Hilbert space~$\mathcal X$
with inner product $(\cdot,\cdot)_{\mathcal X}$
and induced norm~$\Norm{\cdot}_{\mathcal X}$,
we introduce an operator~$\Pcalpbftt$
mapping $\mathcal C^1(0,T; \mathcal X)$
into the space of piecewise polynomials with polynomial
distribution~$\pbft$ over the time grid.
This operator is the integrated version of an operator in~\cite{Thomee:2007}
in the framework of parabolic problems, see also~\cite{Aziz-Monk:1989},
and whose $\p$-approximation properties are derived in~\cite{Schoetzau-Schwab:2000}.

The operator~$\Pcalpbftt$ is piecewise defined fixing the following conditions:
for all $n=1,\dots,N$,
\begin{equation*}
\begin{cases}
(w'-\Pcalpbftt(w)', q_{\pnt-2})_{L^2(\In;\mathcal X)} = 0   \qquad \forall n=1,\dots,N, \quad \forall q_{\pnt-2} \in \Pbb_{\pnt-2}(\In;\mathcal X); \\
\Pcalpbftt(w)'(\tnminus,\cdot) = w'(\tn,\cdot);
\qquad \Pcalpbftt(w)(\tnmo,\cdot) = w(\tnmo,\cdot)
\qquad \text{in }\mathcal X.
\end{cases}
\end{equation*}
As shown in \cite[Section~5.1]{Walkington:2014},
this operator is well defined
and the above definition is equivalent to
\small{\begin{equation} \label{equivalent-definition-Pcalpt}
\begin{cases}
(w-\Pcalpbftt(w), q_{\pnt-3})_{L^2(\In; \mathcal X)} = 0   \qquad \forall n=1,\dots,N, \quad \forall q_{\pnt-3} \in \Pbb_{\pnt-3}(\In;\mathcal X), \\
\Pcalpbftt(w)'(\tnminus,\cdot) = w'(\tn,\cdot),
\qquad \Pcalpbftt(w)(\tn,\cdot) = w(\tn,\cdot),
\qquad \Pcalpbftt(w)(\tnmo,\cdot) = w(\tnmo,\cdot)
\quad \text{in }\mathcal X,
\end{cases}
\end{equation}}\normalsize
where the orthogonality condition
in the first line of~\eqref{equivalent-definition-Pcalpt}
is not taken into account if $\pnt=2$.

The function~$\Pcalpbftt(w)$
is globally continuous
but not necessarily $\mathcal C^1$ in time
and satisfies the following property.

\begin{lemma}
For all~$u$ in $H^2(\In;\mathcal X)$ and all~$q_{\pnt-1}$ in~$\Pbb_{\pnt-1}(\In;\mathcal X)$,
the following identity holds true:
\begin{equation} \label{useful-property-Pcalpt}
( (u-\Pcalpbftt u)'', q_{\pnt-1})_{L^2(\In,\mathcal X)}
=  (\jump{(\Pcalpbftt u)'}(\tnmo,\cdot),
        q_{\pnt-1}(\tnmo,\cdot))_{\mathcal X}
\qquad \forall n=1,\dots,N.
\end{equation}
\end{lemma}
\begin{proof}
We have
\small{\[
\begin{split}
& ( (u-\Pcalpbftt u)'', q_{\pnt-1})_{L^2(\In,\mathcal X)} \\
& = - ( (u-\Pcalpbftt u)', q_{\pnt-1}')_{L^2(\In,\mathcal X)}
    + ( (u-\Pcalpbftt u)'(\tn,\cdot), q_{\pnt-1}(\tn,\cdot))_{\mathcal X}\\
& \qquad\qquad    - ( (u-\Pcalpbftt u)'(\tnmo,\cdot), q_{\pnt-1}(\tnmo,\cdot))_{\mathcal X} \\
& \overset{\eqref{equivalent-definition-Pcalpt}}{=}
    - ( (u-\Pcalpbftt u)'(\tnmo,\cdot), q_{\pnt-1}(\tnmo,\cdot))_{\mathcal X}
  \overset{\eqref{equivalent-definition-Pcalpt}}{=}
    (\jump{(\Pcalpbftt u)'}(\tnmo,\cdot), q_{\pnt-1}(\tnmo,\cdot))_{\mathcal X},
\end{split}
\]} \normalsize
which is the assertion.
\end{proof}

The scope of this section is showing the approximation properties of~$\Pcalpbftt$.
We recall the properties of the original Thom\'ee operator~\cite{Thomee:2007}.
Introduce~$\Pcaltildepbftt$
mapping $\mathcal C^0(0,T;\mathcal X)$
into the space of piecewise polynomials with degree distribution~$\pbft$
over the time grid (in particular, discontinuous functions in time).
The operator~$\Pcaltildepbftt$ is defined fixing the following conditions:
\begin{equation} \label{operator:SS}
\begin{cases}
(w-\Pcaltildepbftt(w), q_{\pnt-1})_{L^2(\In;\mathcal X)} = 0
\qquad \forall n=1,\dots,N, \quad \forall q_{\pnt-1} \in \Pbb_{\pnt-1}(\In;\mathcal X); \\
\Pcaltildepbftt(w)(\tnminus,\cdot) = w(\tn,\cdot)
\qquad \text{in }\mathcal X.
\end{cases}
\end{equation}
The following approximation results can be found
in \cite[Lemmas $3.6$--$3.8$, Theorem 3.10]{Schoetzau-Schwab:2000}
and \cite[Lemma 3.3]{Castillo-Cockburn-Schoetzau-Schwab:2002}.
\begin{lemma} \label{lemma:Schoetzau-Schwab}
Let~$\Pcaltildepbftt$ be the operator in~\eqref{operator:SS}.
Then, for all~$w$ in $H^{s+1}(\In;\mathcal X)$
with $s\geq0$,
the following inequalities hold true for all~$n=1,\dots,N$:
\small{\begin{subequations}
\begin{align}
\Norm{w-\Pcaltildepbftt w}_{L^2(\In;\mathcal X)}
&\lesssim \Norm{w-\Pizpbft w}_{L^2(\In;\mathcal X)}
        + \frac{\taun}{\pnt}
        \Norm{w'}_{L^2(\In;\mathcal X)},
        \label{Schoetzau-Schwab-a}\\
\begin{split}
\Norm{w-\Pcaltildepbftt w}_{L^2(\In;\mathcal X)}
& \lesssim \inf_{\qpnt\in\Pbb_{\pnt}(\In;\mathcal X)} \left( \Norm{w-\qpnt}_{L^2(\In;\mathcal X)}
    + \frac{\taun}{\pnt}
    \Norm{(w-\qpnt)'}_{L^2(\In;\mathcal X)} \right) \\
& \lesssim \left( \frac{\taun}{\pnt} \right)
        ^{\min(s,\pnt)+1} \Norm{w^{(s+1)}}_{L^2(\In;\mathcal X)},
\end{split}\label{Schoetzau-Schwab-b}\\
\Norm{\Pcaltildepbftt w}_{L^2(\In;\mathcal X)}
& \lesssim \Norm{w}_{L^2(\In;\mathcal X)}
        + \frac{\taun}{\pnt} \Norm{w'}_{L^2(\In;\mathcal X)},
        \label{Schoetzau-Schwab-c}\\
\begin{split}
\Norm{(w-\Pcaltildepbftt w)(\tnmo,\cdot)}_{\mathcal X}
& \lesssim \left( \frac{\taun}{\pnt} \right)^{\frac12}
\inf_{\qpnt \in \Pbb_{\pnt}(\In;\mathcal X)}
    \Norm{(w-\qpnt)'}_{L^2(\In;\mathcal X)} \\
& \lesssim \left( \frac{\taun}{\pnt} \right)
    ^{\min(s,\pnt)+\frac12} \Norm{w^{(s+1)}}_{L^2(\In;\mathcal X)}.
\end{split} \label{Schoetzau-Schwab-d}
\end{align}
\end{subequations}}
\normalsize
\end{lemma}
On the reference interval~$\Ihat=[-1,1]$\footnote{With an abuse of notation, we use the same notation
for the operators on the reference interval.},
the bounds in Lemma~\ref{lemma:Schoetzau-Schwab} are a consequence
of the properties of the Legendre polynomials, the identity
\begin{equation} \label{identity:SS}
(w - \Pcaltildepbftt w)(t,\cdot)
= (w - \Pizpbft w)(t,\cdot)
    + (w-\Pizpbft w)(1,\cdot) L_p(t)
    \qquad \text{in } \mathcal X,
\end{equation}
and sharp estimates
\cite[Lemmas~3.5 and 3.6]{Schoetzau-Schwab:2000}
for the second term on the right-hand side of~\eqref{identity:SS}.
\medskip

The operators~$\Pcalpbftt$ and~$\Pcaltildepbftt$ in~\eqref{equivalent-definition-Pcalpt} and~\eqref{operator:SS}
are related by the following identity:
for all~$t$ in~$\In$,
\begin{equation} \label{relation:SS-Ptilde}
\Pcalpbftt (w)(t,\cdot)
= \int_{\tnmo}^t (\Pcaltildepbftmo (w'))(s,\cdot) \ ds + w(\tnmo,\cdot)
\quad \text{in } \mathcal X
\qquad\qquad \forall n=1,\dots,N.
\end{equation}
Combining~\eqref{relation:SS-Ptilde} with Lemma~\ref{lemma:Schoetzau-Schwab},
and using \cite[Theorem~3.10]{Schoetzau-Schwab:2000},
we get the following result.
\begin{lemma} \label{lemma:Schoetzau-Schwab-derivatives}
Let~$\Pcalpbftt$ be the operator in~\eqref{equivalent-definition-Pcalpt}.
Then, for all $w$ in $H^{s+1}(\In;\mathcal X)$
with $s\geq1$,
the following inequalities hold true for all~$n=1,\dots,N$:
\begin{subequations}
\begin{align}
\Norm{(w-\Pcalpbftt w)'}_{L^2(\In;\mathcal X)}
& \lesssim \Norm{w'-\Pizpbftmo w'}_{L^2(\In;\mathcal X)}
    + \frac{\taun}{\pnt} \Norm{w''}_{L^2(\In;\mathcal X)} \label{Schoetzau-Schwab-derivatives-a}, \\
\begin{split}
\Norm{(w-\Pcalpbftt w)'}_{L^2(\In;\mathcal X)}
& \lesssim \inf_{\qpnt \in \Pbb_{\pnt}(\In;\mathcal X)} \left( \Norm{(w-\qpnt)'}_{L^2(\In;\mathcal X)}
    + \frac{\taun}{\pnt} \Norm{(w-\qpnt)''}_{L^2(\In;\mathcal X)} \right) \\
& \lesssim \left( \frac{\taun}{\pnt} \right)
        ^{\min(s,\pnt)} \Norm{w^{(s+1)}}_{L^2(\In;\mathcal X)},
\end{split} \label{Schoetzau-Schwab-derivatives-b} \\
\Norm{(\Pcalpbftt w)'}_{L^2(\In;\mathcal X)}
& \lesssim \Norm{w'}_{L^2(\In;\mathcal X)}
    + \frac{\taun}{\pnt} \Norm{w''}_{L^2(\In;\mathcal X)}. \label{Schoetzau-Schwab-derivatives-c}
\end{align}
\end{subequations}
\end{lemma}
\begin{proof}
Inequality~\eqref{Schoetzau-Schwab-derivatives-a} follows
from~\eqref{Schoetzau-Schwab-a} and~\eqref{relation:SS-Ptilde}.
Inequality~\eqref{Schoetzau-Schwab-derivatives-c} follows
from~\eqref{Schoetzau-Schwab-c} and~\eqref{relation:SS-Ptilde}.
Inequality~\eqref{Schoetzau-Schwab-derivatives-b} follows
from~\eqref{Schoetzau-Schwab-derivatives-c}, the triangle inequality,
and the fact that~$\Pcalpbftt$ preserves polynomials
of degree~$\pnt$ in time for all $n=1,\dots,N$.
\end{proof}

We derive bounds for the operator~$\Pcalpbftt$ in $L^2$-type norms.
To this aim, we introduce
the piecewise $H^1$ projector~$\Pionepbftmo$
onto polynomials in time as follows:
\begin{equation} \label{Pionetpmo}
\begin{cases}
    ( (w-\Pionepbftmo w)', \qpntmo'  )_{0,\In} = 0
            & \forall w \in H^1(\In;\mathcal X), \quad \qpntmo \in \Pbb_{\pnt-1}(\In;\mathcal X); \\
    (w-\Pionepbftmo w)(\tnmo,\cdot) = 0
    \quad \text{in } \mathcal X
            & \forall n=1,\dots,N.
\end{cases}
\end{equation}
On the other hand, the operator~$\Pizpbftmo$ denotes the piecewise $L^2$ projector onto polynomials in time.

Some properties of the operator~$\Pcalpbftt$
are detailed in the next result.
\begin{proposition} \label{proposition:Pcaltildept-L2-approx}
Let~$\Pcalpbftt$ and~$\Pionepbftmo$ be the operators
in~\eqref{equivalent-definition-Pcalpt} and~\eqref{Pionetpmo}.
Then, the following inequality holds true:
for all~$w$ in $H^{s+1}(\In;\mathcal X)$ with $s\geq1$
\begin{equation} \label{Pcaltildept-L2-approx}
\Norm{w-\Pcalpbftt w}_{L^2(\In;\mathcal X)}
\lesssim \Norm{w-\Pionepbft w}_{L^2(\In;\mathcal X)}
        + \frac{\taun^2}{(\pnt)^2} \Norm{w''}_{L^2(\In;\mathcal X)}
\qquad\qquad \forall n=1,\dots,N.
\end{equation}
Since~$\Pcalpbftt$ preserves polynomials in time of order~$\pnt$,
we also have
\small{\begin{equation} \label{Pcaltildept-L2-approx-bis}
\begin{split}
\Norm{w-\Pcalpbftt w}_{L^2(\In;\mathcal X)}
&\lesssim \inf_{\qpnt \in \Pbb_{\pnt}(\In;\mathcal X)} \Big( \frac{\taun}{\pnt} \Norm{(w-\qpnt)'}_{L^2(\In;\mathcal X)}
    + \frac{\taun^2}{(\pnt)^2} \Norm{(w-\qpnt)''}_{L^2(\In;\mathcal X)} \Big) \\
& \lesssim \left( \frac{\taun}{\pnt} \right)
        ^{\min(s,\pnt)+1} \Norm{w^{(s+1)}}_{L^2(\In;\mathcal X)}
\qquad\qquad \forall n=1,\dots,N.
\end{split}
\end{equation}}\normalsize
\end{proposition}
\begin{proof}
We prove the assertion on the reference interval~$\Ihat=(-1,1)$;
the general bound follows from a scaling argument.

Identity~\eqref{identity:SS} implies
\small{\[
(w - \Pcalpbftt w)'(t,\cdot)
\overset{\eqref{relation:SS-Ptilde}}{=}
(w' - \Pcaltildepbftmo (w'))(t,\cdot)
= (w'-\Pizpbftmo w')(t,\cdot)
        + (w'-\Pizpbftmo w')(1,\cdot) L_{\pnt-1}(t)
        \quad \text{in } \mathcal X.
\]}\normalsize
Since $(w-\Pcalpbftt w)(-1,\cdot)=0$, an integration by parts gives
\[
\begin{split}
(w-\Pcalpbftt w)(t,\cdot)
& = \int_{-1}^t (w-\Pcalpbftt w)'(s,\cdot) ds\\
& = \int_{-1}^t (w'-\Pizpbftmo w')(s,\cdot) ds + (w'-\Pizpbftmo w')(1,\cdot) \int_{-1}^t L_{\pnt-1}(s) ds
  =: T_1 + T_2.
\end{split}
\]
We estimate the two terms on the right-hand side separately.
The properties of the $L^2$ and $H^1$ projectors
imply $\Pizpbftmo w' = (\Pionepbft w)'$.
Therefore, recalling from~\eqref{Pionetpmo} that~$w(-1,\cdot)=\Pionepbft w(-1,\cdot)$,
we write
\[
T_1
= \int_{-1}^t (w'-(\Pionepbft w)')(s,\cdot) ds
= (w-\Pionepbft w)(t,\cdot)
\qquad\qquad \text{in } \mathcal X.
\]
This identity leads to the estimate
on the first term on the right-hand side of~\eqref{Pcaltildept-L2-approx}.

As for the term~$T_2$, we expand~$w'$ with respect to Legendre polynomials:
\[
w'(t,\cdot) = \sum_{j=0}^{+\infty} w_j'(\cdot) L_j(t)
\qquad \text{in } \mathcal X,
\qquad\qquad\qquad w_j' \in \mathcal X.
\]
Standard properties of the $L^2$ projector imply
\[
T_2
=  \left( \sum_{j=\pnt}^{+\infty} w_j'(\cdot) \right) \int_{-1}^t L_{\pnt-1}(s) ds
\qquad\qquad\qquad \text{in } \mathcal X.
\]
From \cite[eq. (3.5)]{Schoetzau-Schwab:2000}, we have
\[
\SemiNorm{\sum_{j=\pnt}^{+\infty} w_j'(\cdot)}
\lesssim (\pnt)^{-\frac12} \Norm{w''(\cdot)}_{0,\Ihat}
\qquad\qquad\qquad \text{in } \mathcal X.
\]
Using that $\int_{-1}^t L_{\pnt-1}(s) ds = (2\pnt-1)^{-1} (L_{\pnt} - L_{\pnt-2})$
and~$\Norm{L_{\pnt}}_{0,\Ihat} \approx ({\pnt})^{-\frac12}$ yields
\[
\Norm{\int_{-1}^{\cdot} L_{\pnt-1}(s)ds}_{0,\Ihat}
\lesssim (\pnt)^{-1} \left( \Norm{L_{\pnt}}_{0,\Ihat} + \Norm{L_{\pnt-2}}_{0,\Ihat} \right)
\lesssim (\pnt)^{-\frac32}.
\]
Collecting the two displays above implies
\[
T_2 \lesssim (\pnt)^{-2} \Norm{w''(\cdot)}_{0,\Ihat}
\qquad\qquad\qquad \text{in } \mathcal X.
\]
This concludes the proof of~\eqref{Pcaltildept-L2-approx}.
As for~\eqref{Pcaltildept-L2-approx-bis},
we add and subtract~$\qpnt$,
\lm{employ~\eqref{Pcaltildept-L2-approx},}
note that~$\Pcalpbftt$ piecewise preserves piecewise polynomials of degree~$\pnt$ in time,
and use the properties of the projector~$\Pionepbft$ as in \cite[Corollary~$3.15$]{Schwab:1998}.
\end{proof}

\subsection{Abstract error analysis} \label{subsection:abstract-analysis}
Let~$\PihEcal: \HozOmega \to \Vh$ denote the elliptic projector
defined as
\begin{equation} \label{elliptic-projector}
a(u-\PihEcal u, \vh) = 0
\qquad\qquad \forall \vh \in  \Vh .
\end{equation}
We have the following stability estimate for the elliptic projector~$\PihEcal$
in~\eqref{elliptic-projector},
which is a consequence of the Aubin-Nitsche technique:
for any nonnegative~$s$ larger than or equal to~$1$,
and possibly equal to~$\infty$,
there exists~$\alpha$ in $(1/2,1]$,
see, e.g., \cite[Corollary~8]{Dauge:1992},
depending on~$\Omega$ such that
\begin{equation} \label{elliptic-projector-L2-stab}
\Norm{\PihEcal v}_{L^s(\In;L^2(\Omega))}
\lesssim \Norm{v}_{L^s(\In;L^2(\Omega))}
        + \frac{\h^\alpha}{\px^\alpha}
        \SemiNorm{v}_{L^s(\In;H^1(\Omega))}
\qquad\qquad \forall n=1,\dots,N.
\end{equation}
Recall that~$\Pcalpbftt$ is given in~\eqref{equivalent-definition-Pcalpt} and let
\[
\utildeh = \Pcalpbftt \PihEcal u .
\]
We analyze a priori estimates of the following error quantity,
which we spit into two contributions:
\begin{equation} \label{error:splitting}
\eh = u-\uh = (u-\utildeh) + (\utildeh - \uh)
    =: \rho + \xih.
\end{equation}
We show an upper bound on a suitable norm of the two terms
on the right-hand side of~\eqref{error:splitting}.
We begin discussing the term~$\xih$.

\begin{lemma}\label{lemma:estimates-xih}
Consider~$m=m(\xih)$, $\mum$, and~$\xih$
as in~\eqref{selecting-m},
\eqref{mun}, and~\eqref{error:splitting}.
Let~$u$ and~$\uh$ be the solutions to~\eqref{weak-problem} and~\eqref{Walkington:method}.
Let~$\uzeroh$ be the elliptic projection of~$\uzero$
and $\uoneh$ be the $L^2$-orthogonal projection of~$\uone$ onto~$\Vh$.
Assume that $\Deltax u$ and  $u''$
belong to $H^{2}(\In;L^2(\Omega))$
and $L^2(\In;H^1_0(\Omega))$ for all~$n=1,\dots,m$.
Recall that their initial conditions are~$\uzero$ and~$\uone$,
and~$\uzeroh$ and~$\uoneh$, respectively.
Then, the following estimate holds true:
\[
\begin{split}
& \mum \left(
    \Norm{\xih'}_{L^\infty(\Im ; L^2(\Omega))}^2
    + \SemiNorm{\xih}_{L^\infty(\Im ; H^1(\Omega))}^2 \right)
 + \frac14 \sum_{n=1}^m \Norm{\jump{\xih'}(\tnmo,\cdot)}_{0,\Omega}^2 \\
& \quad \le \frac12 \SemiNorm{(\uzero-\uzeroh)(\cdot)}_{1,\Omega}^2
           + \Norm{(\uone-\uoneh)(\cdot)}_{0,\Omega}^2
           + 2 \Norm{\uone(\cdot)-(\Pcalpbftt u)'(0,\cdot)}_{0,\Omega}^2 \\
&\quad\quad
        + 2 \Norm{ (\Pcalpbftt (I-\PihEcal )u)'(0,\cdot)}_{0,\Omega}^2 \\
&  \quad\quad
    + 2 \frac{\tm}{\mum} \Big(  \sum_{n=1}^m\Norm{(I-\Pcalpbftt)\Deltax u}_{L^2(\In;L^2(\Omega))}^{2}
    +  \sum_{n=1}^m\Norm{(I-\PihEcal) u''}_{L^2(\In;L^2(\Omega))}^2 \Big).
\end{split}
\]
The first term on the right-hand side above
vanishes if we discretize the initial condition~$\uzero$
with the elliptic projector in~\eqref{elliptic-projector}.
\end{lemma}
\begin{proof}
Let~$\eh$ be as in~\eqref{error:splitting}.
Subtracting~\eqref{weak-problem} and~\eqref{Walkington:method},
for all $n=1,\dots,N$, we arrive at
\[
\int_{\In} [(\eh'',\Wh)_{0,\Omega} + a(\eh, \Wh)]dt
   + (\jump{\eh'}(\tnmo,\cdot), \Wh(\tnmo,\cdot))_{0,\Omega} =0
   \quad \forall \Wh \in \Pbb_{\pnt-1}(\In; \Vh).
\]
Splitting~\eqref{error:splitting} yields
\footnotesize{\begin{equation} \label{error-equation:xih}
\begin{split}
& \int_{\In} [(\xih'',\Wh)_{0,\Omega} + a(\xih,\Wh)]dt
   + \Big(\jump{\xih'}(\tnmo,\cdot), \Wh(\tnmo,\cdot) \Big)_{0,\Omega} \\
& = - \int_{\In} (\rho'',\Wh)_{0,\Omega}dt - \int_{\In} a(\rho, \Wh) dt
   - (\jump{\rho'}(\tnmo,\cdot), \Wh(\tnmo,\cdot))_{0,\Omega}
   \qquad \forall \Wh \in \Pbb_{\pnt-1}(\In; \Vh).
\end{split}
\end{equation}}\normalsize
We rewrite the right-hand side of~\eqref{error-equation:xih}
as $(\fxih,\V)_{L^2(0,T;L^2(\Omega))}$,
i.e., as the right-hand side of method~\eqref{Walkington:method},
for a suitable $\fxih$,
which we fix in~\eqref{fxih},
so as to exploit the stability estimates~\eqref{discrete-stability-estimate}
and deduce the assertion.

We focus on the second term on the right-hand side of~\eqref{error-equation:xih}.
Using the fact that $\Pcalpbftt$ and~$\PihEcal$ commute,
definition~\eqref{elliptic-projector} of~$\PihEcal$,
the fact that~$\Pcalpbftt$ preserves polynomials in time,
and an integration by parts,
we arrive at
\begin{equation} \label{estimating-1st-term-EE}
\begin{split}
& \int_{\In} a(\rho, \Wh) dt
  = \int_{\In} a((I-\Pcalpbftt\PihEcal) u, \Wh)dt \\
& = \int_{\In} a((I-\Pcalpbftt) u,\Wh)dt
  = -(\Deltax (I-\Pcalpbftt) u, \Wh)_{L^2(\In;L^2(\Omega))}.
\end{split}
\end{equation}
As for the first and third
terms on the right-hand side of~\eqref{error-equation:xih},
we use the continuity in time of~$u'$ and~\eqref{useful-property-Pcalpt}, and get
\small{\[
\begin{split}
& (\rho'',\Wh)_{L^2(\In;L^2(\Omega))}
    + \big(\jump{\rho'}(\tnmo,\cdot), \Wh(\tnmo,\cdot)\big)_{0,\Omega} \\
& = (u'',\Wh)_{L^2(\In;L^2(\Omega))}
    + \big(\jump{u'}(\tnmo,\cdot), \Wh(\tnmo,\cdot)\big)_{0,\Omega}\\
&\quad  - ((\Pcalpbftt\PihEcal u)'',\Wh)_{L^2(\In;L^2(\Omega))}
    - \Big(\jump{(\Pcalpbftt\PihEcal u)'}(\tnmo,\cdot), \Wh(\tnmo,\cdot)\Big)_{0,\Omega} \\
& = (u'',\Wh)_{L^2(\In;L^2(\Omega))}
    - ((\Pcalpbftt\PihEcal u)'',\Wh)_{L^2(\In;L^2(\Omega))}
    - \Big(\jump{(\Pcalpbftt\PihEcal u)'}(\tnmo,\cdot), \Wh(\tnmo,\cdot)\Big)_{0,\Omega} \\
& \overset{\eqref{useful-property-Pcalpt}}{=}
    (u'',\Wh)_{L^2(\In;L^2(\Omega))}
    - ((\PihEcal u)'',\Wh)_{L^2(\In;L^2(\Omega))}.
\end{split}
\]}\normalsize
For all $n=1,\dots N$, we deduce
\begin{equation} \label{estimating-2nd-term-EE}
 [(\rho'',\Wh)_{L^2(\In;L^2(\Omega))}
+ (\jump{\rho'}(\tnmo,\cdot), \Wh(\tnmo,\cdot))_{0,\Omega}
 = ( (I-\PihEcal)  u'', \Wh)_{L^2(\In;L^2(\Omega))} .
\end{equation}
We insert~\eqref{estimating-1st-term-EE} and~\eqref{estimating-2nd-term-EE} in~\eqref{error-equation:xih} and arrive at
\begin{equation*}
\begin{split}
& \int_{\In} [(\xih'',\Wh)_{0,\Omega} + a(\xih, \Wh)]dt
   + (\jump{\xih'}(\tnmo,\cdot), \Wh(\tnmo,\cdot))_{0,\Omega} \\
& = ( (I-\Pcalpbftt)\Deltax u  , \Wh)_{L^2(\In;L^2(\Omega))}
    - ( (I-\PihEcal) u'', \Wh)_{L^2(\In;L^2(\Omega))} \\
& = ((I-\Pcalpbftt)\Deltax u
        - (I-\PihEcal) u'', \Wh )_{L^2(\In;L^2(\Omega))}
   \qquad \forall \Wh \in \Pbb_{\pnt-1}(\In; \Vh),
   \quad \forall n=1,\dots N.
\end{split}
\end{equation*}
In words, $\xih$ is the solution to method~\eqref{Walkington:method}
with right-hand side given by
$(\fxih,\V)_{L^2(0,T;L^2(\Omega))}$ with
\begin{equation} \label{fxih}
\fxih := (I-\Pcalpbftt)\Deltax u - (I-\PihEcal)  u'' .
\end{equation}
We are now in a position to apply
the discrete stability estimates~\eqref{discrete-stability-estimate}.
Let~$m=m(\xih)$ be defined in~\eqref{selecting-m}
and get
\[
\begin{split}
& \mum \left(
    \Norm{\xih'}_{L^\infty(\Im; L^2(\Omega))}^2
    + \SemiNorm{\xih}_{L^\infty(\Im; H^1(\Omega))}^2 \right)
 + \frac14  \sum_{n=1}^m
        \Norm{\jump{\xih}(\tn,\cdot)}_{0,\Omega}^2\\
& \le  \frac12 \Big( \SemiNorm{\xih(0,\cdot)}_{1,\Omega}^2
        + \Norm{(\xih^-)'(0,\cdot)}_{0,\Omega}^2 \Big)
        + \frac{\tm}{\mum}
        \Norm{\fxih}_{L^2(0,\tm;L^2(\Omega))}^2
   =: T_1 + T_2. 
\end{split}
\]
We estimate the two terms on the right-hand side.
The triangle inequality implies
\[
T_1
\le \frac12 \left( \SemiNorm{\xih(0,\cdot)}_{1,\Omega}^2
                + \Norm{(\xih^-)'(0,\cdot)}_{0,\Omega}^2 \right)
= \frac12 [T_{1,1}+T_{1,2}].
\]
Using~\eqref{equivalent-definition-Pcalpt}
and the stability of the elliptic projection~$\PihEcal$,
we deduce
\[
T_{1,1}
= \SemiNorm{\xih(0,\cdot)}_{1,\Omega}^2
= \SemiNorm{\uh(0,\cdot)- \PihEcal u(0,\cdot)}_{1,\Omega}^2
\le \SemiNorm{(\uzero - \uzeroh)(\cdot)}_{1,\Omega}^2.
\]
The term~$T_{1,1}$ vanishes if the initial condition
$\uzero$ is discretized with the elliptic projector in~\eqref{elliptic-projector}.
We further have
\small{\[
\begin{split}
T_{1,2}
& = \Norm{(\xih^-)'(0,\cdot)}_{0,\Omega}^2
  = \Norm{(\uh^-)'(0,\cdot) - (\Pcalpbftt \PihEcal u)'(0,\cdot)}_{0,\Omega}^2
  = \Norm{\uoneh(\cdot) - (\Pcalpbftt \PihEcal u)'(0,\cdot)}_{0,\Omega}^2  \\
& \le 2\Norm{(\uone-\uoneh)(\cdot)}_{0,\Omega}^2
      + 2\Norm{\uone(\cdot) - (\Pcalpbftt \PihEcal u)'(0,\cdot)}_{0,\Omega}^2 \\
& \le 2\Norm{(\uone-\uoneh)(\cdot)}_{0,\Omega}^2
      + 4\Norm{\uone(\cdot) - (\Pcalpbftt  u)'(0,\cdot)}_{0,\Omega}^2
  + 4\Norm{ (\Pcalpbftt (I-\PihEcal )u)'(0,\cdot)}_{0,\Omega}^2.
\end{split}
\]}\normalsize
Next, we deal with the term~$T_2$:
\small{\[
T_2
= \frac{\tm}{\mum}
        \Norm{\fxih}_{L^2(0,\tm;L^2(\Omega))}^2
=: \frac{\tm}{\mum}
    \sum_{n=1}^m
    \Norm{(I-\Pcalpbftt)\Deltax u - (I-\PihEcal)  u''}
    _{L^2(\In;L^2(\Omega))}^2
:= \frac{\tm}{\mum}
    \sum_{n=1}^m T_{2,n}.
\]}\normalsize
We have
\[
T_{2,n}
\le 2 \Norm{ (I-\Pcalpbftt)\Deltax u }_{L^2(\In; L^2(\Omega))}^2
     + 2 \Norm{(I-\PihEcal) u''}_{L^2(\In; L^2(\Omega))}^2.
\]
Collecting the bounds on~$T_1$ and~$T_2$ yields the assertion.
\end{proof}

Introduce~$k$ in $1,\dots,N$ such that
\begin{equation} \label{selecting-k}
\left( \Norm{\eh'}_{L^\infty(\Ik; L^2(\Omega))}^2
    + \SemiNorm{\eh}_{L^\infty(\Ik; H^1(\Omega))}^2 \right)
:=\max_{n=1}^N \left( \Norm{\eh'}_{L^\infty(\In; L^2(\Omega))}^2
    + \SemiNorm{\eh}_{L^\infty(\In; H^1(\Omega))}^2 \right) .
\end{equation}
The index~$k$ is defined similarly to
the index~$m$ in~\eqref{selecting-m}.
However, on the one hand, with~$k$ we maximize
positive functionals in Bochner spaces
and not in space--time finite element spaces;
on the other hand, using a different nomenclature of
the indices allows us to improve the readability
of the estimates in Theorems~\ref{theorem:abstract-convergence}
and~\ref{theorem:h-p-convergence} below.

\begin{theorem} \label{theorem:abstract-convergence}
Consider~$m=m(\xih)$, $\mum$, $k$, and~$\xih$ and~$\rho$
as in~\eqref{selecting-m}, \eqref{mun}, \eqref{selecting-k},
and~\eqref{error:splitting}.
Let~$u$ and~$\uh$ be the solutions to~\eqref{weak-problem} and~\eqref{Walkington:method}.
Assume that $\Deltax u$ and  $u''$
belong to $H^{2}(\In;L^2(\Omega))$
and $L^2(\In;H^1_0(\Omega))$ for all~$n=1,\dots,m$, respectively.
Recall that their initial conditions are~$\uzero$ and~$\uone$,
and~$\uzeroh$ and~$\uoneh$, respectively.
Then, the following estimate holds true:
\begin{equation}\label{abstract-estimates:maximum:norm}
\begin{split}
& \max_{n=1}^N \left(
    \Norm{\eh'}_{L^\infty(\In; L^2(\Omega))}^2
+ \SemiNorm{\eh}_{L^\infty(\In; H^1(\Omega))}^2 \right)  
    \le \frac{1}{\mum} \Big(\SemiNorm{(\uzero-\uzeroh)(\cdot)}_{1,\Omega}^2 \\
&   \qquad + 2 \Norm{(\uone-\uoneh)(\cdot)}_{0,\Omega}^2
    + 4 \Norm{\uone(\cdot)-(\Pcalpbftt u)'(0,\cdot)}_{0,\Omega}^2
    + 4 \Norm{ (\Pcalpbftt (I-\PihEcal )u)'(0,\cdot)}_{0,\Omega}^2 \Big) \\
& \quad\quad
    + \frac{4 \tm}{\mum^2} \sum_{n=1}^m
     \Norm{(I-\Pcalpbftt)\Deltax u}_{L^2(\In;L^2(\Omega))}^2
    + \frac{4\tm}{\mum^2} \sum_{n=1}^m
     \Norm{(I-\PihEcal) u''}_{L^2(\In;L^2(\Omega))}^2\\
& \quad \quad
    + 2 \Norm{\rho'}_{L^\infty(\Ik; L^2(\Omega))}^2
    + 2 \SemiNorm{\rho}_{L^\infty(\Ik; H^1(\Omega))}^2 =: \sum_{j=1}^{8} \gimel_j^2.
\end{split}
\end{equation}
\end{theorem}
\begin{proof}
Using the triangle inequality and the definition
of~$m=m(\xih)$ in~\eqref{selecting-m}, we deduce
\[
\begin{split}
&
    \max_{n=1}^N \left(
    \Norm{\eh'}_{L^\infty(\In; L^2(\Omega))}^2
    + \SemiNorm{\eh}_{L^\infty(\In; H^1(\Omega))}^2 \right)\\
 &   \leq 2 \left(
    \Norm{\rho'}_{L^\infty(\Ik; L^2(\Omega))}^2
    + \SemiNorm{\rho}_{L^\infty(\Ik; H^1(\Omega))}^2 \right)
    +2 \left(
    \Norm{\xih'}_{L^\infty(\Ik; L^2(\Omega))}^2
    + \SemiNorm{\xih}_{L^\infty(\Ik; H^1(\Omega))}^2 \right) \\
    &   \leq 2 \left(
    \Norm{\rho'}_{L^\infty(\Ik; L^2(\Omega))}^2
    + \SemiNorm{\rho}_{L^\infty(\Ik; H^1(\Omega))}^2 \right)
    +2 \left(
    \Norm{\xih'}_{L^\infty(\Im; L^2(\Omega))}^2
    + \SemiNorm{\xih}_{L^\infty(\Im; H^1(\Omega))}^2 \right) . 
\end{split}
\]
The assertion follows using Lemma~\ref{lemma:estimates-xih}.
\end{proof}

\subsection{Error estimates} \label{subsection:h-p-convergence}
For~$m=m(\xih)$ and~$\mum$ as in~\eqref{selecting-m} and~\eqref{mun},
we derive error estimates for method~\eqref{Walkington:method},
which are explicit in the spatial mesh size, the time steps,
and the polynomial degrees in space and time,
with respect to the norm
\begin{equation} \label{full-error-a-priori}
\max_{n=1}^N\left( \Norm{\cdot'}_{L^\infty(\In; L^2(\Omega))}^2
+ \SemiNorm{\cdot}_{L^\infty(\In; H^1(\Omega))}^2 \right).
\end{equation}
To this aim, we give explicit bounds on the terms~$\gimel_j$,
$j=1,\dots,8$, appearing on the right-hand side
of~\eqref{abstract-estimates:maximum:norm},
and collect the resulting estimates in the following result.

\begin{theorem} \label{theorem:h-p-convergence}
Let~$u$ and~$\uh$ be the solutions to~\eqref{weak-problem} and~\eqref{Walkington:method},
and~$\eh$ be as in~\eqref{error:splitting};
$\uzero$, $\uone$, and~$u$ be sufficiently smooth;
$m=m(\xih)$, $\mum$, and~$k$ be
as in~\eqref{selecting-m}, \eqref{mun}, and~\eqref{selecting-k};
$\alpha$ be the elliptic regularity parameter as in~\eqref{elliptic-projector-L2-stab}.
Then, the following a priori estimate holds true:
\[
\begin{split}
&  \max_{n=1}^N \left(
    \Norm{\eh'}_{L^\infty(\In; L^2(\Omega))}^2
    + \SemiNorm{\eh}_{L^\infty(\In; H^1(\Omega))}^2 \right)
    \lesssim 
   \daleth_1+\daleth_2+\daleth_3 ,
\end{split}
\]
where we have set
\small\[
\begin{split}
\daleth_1
& := \frac{1}{\mum} \Big(\frac{\h^{2\min{(\px,s-\frac12)}}}{\px^{2s-1}}
    \Norm{\uzero(\cdot)}_{s+\frac12,\Omega}^2\\
& \quad + \frac{\h^{2\min{(\px+1,s-\frac12)}}}{\px^{2s-1}}
        \Norm{\uone(\cdot)}_{s-\frac12,\Omega}^2
        + \Big( \frac{\tauo}{\pot}\Big)^{2\min(\pot,s)-1}
        \Norm{u^{(s+1)}}_{L^2(I_1, L^2(\Omega))}^2 \\
& \quad + \Big(\frac{\tauo}{\pot}\Big)
  \frac{\h^{2\min(\px+1,s-1)-2(1-\alpha)}}{\px^{2s-2-2(1-\alpha)}}
    \Norm{u''}_{L^2(I_1;H^{s-1}(\Omega))} ^2
+ \frac{\h^{2\min(\px+1,s-\frac12)-2(1-\alpha)}}{\px^{2s-1-2(1-\alpha)}}
    \Norm{u_1(\cdot)}_{s-\frac12,\Omega}^2 \Big), \\
\daleth_2 
&:= \frac{\tm}{\mum^2} \sum_{n=1}^m
    \Big( \frac{\taun}{\pnt} \Big)^{2\min(\pnt+1, s-1)}
    \Norm{\Deltax u^{(s-1)}}_{L^{2}(\In; L^2(\Omega))}^2 \\
& \quad + \frac{\tm}{\mum^2} \sum_{n=1}^m
    \frac{\h^{2\min(\px+1,s-1)-2(1-\alpha)}}{\px^{2s-2-2(1-\alpha)}}
    \Norm{u''}_{L^2(\In;H^{s-1}(\Omega))}^2,\\
\end{split}
\]\normalsize
and
\small\[
\begin{split}
\daleth_3
&:= \Big(  \frac{\h^{2\min(\pkt,s)+2}\tauk}{(\pkt)^{2s+2}}
    \Big(1 + \frac{\h^{-2+2\alpha}}{\px^{-2+2\alpha}} \Big)
    \big(\tauk^{-1} \Norm{u'}_{L^{\infty}(\Ik;H^{s+1}(\Omega))}^2 \big)    \Big. \\
& \qquad \Big. + \frac{\tauk^{2\min(\pkt,s)-1}}{(\pkt)^{2\min(\pkt,s)-2}}
        \Big( \Norm{u^{(s+1)}}_{L^{2}(\Ik;L^2(\Omega))}^2
        +  \frac{\h^{2\alpha-2}}{\px^{2\alpha}} \h^{2}
        \Norm{u^{(s+1)}}_{L^{2}(\Ik;H^1(\Omega))}^2 \Big) \Big)  \\
& \quad + \Big( \frac{\h^{2\min(\px,s)}\tauk^{-1}}{\px^{2s}}
    \big(\tauk \Norm{u}_{L^\infty(\Ik;H^{s+1}(\Omega))^2}\big)
    + \Big( \frac{\tauk}{\pkt}\Big)^{2\min(\pkt+1,s)-1}
    \Norm{u^{(s)}}_{L^{2}(\Ik,H^1(\Omega))}^2 \Big).
\end{split}
\]\normalsize
In particular,
focusing on the $\h-$ and~$\tau-$versions of the scheme,
quasi-uniform time-steps comparable to $\tau$,
and sufficiently smooth~$u$, there holds
\footnotesize\[
\begin{split}
&  \max_{n=1}^N \left(
    \Norm{\eh'}_{L^\infty(\In; L^2(\Omega))}^2
    + \SemiNorm{\eh}_{L^\infty(\In; H^1(\Omega))}^2 \right)
    \lesssim 
    \big(\h^{2\p} \Norm{\uzero(\cdot)}_{\p+1,\Omega}^2
    + \h^{2(\p+1)} \Norm{\uone(\cdot)}_{\p+1,\Omega}^2 \\
&   \qquad\qquad    + \tau^{2\p-1} \Norm{u^{(\p+1)}}_{L^2(I_1, L^2(\Omega))}^2
    + \tau \h^{2(\p+1)} \Norm{u''}_{L^2(I_1;H^{\p+1}(\Omega))}^2 \big)  \\
&   \quad + \big( \tm \tau^{2(\p+1)}
    \Norm{\Deltax u^{(\p+1)}}_{L^{2}(0,\tm; L^2(\Omega))}^2 
    + \tm \h^{2(\p+1)} \Norm{u''}_{L^2(0,\tm;H^{\p+1}(\Omega))}^2 \big) \\
&   \quad + \big(  \h^{2(\p+1)}
        \Norm{u'}_{L^{\infty}(\Ik;H^{\p+1}(\Omega))}^2
    + \tau^{2\p-1} \Norm{u^{(\p+1)}}_{L^{2}(\Ik;L^2(\Omega))}^2
        +  \h^{2} \tau^{2\p-1} \Norm{u^{(\p+1)}}_{L^{2}(\Ik;H^1(\Omega))}^2 \\
&   \quad\quad + \h^{2\p} \Norm{u}_{L^\infty(\Ik;H^{\p+1}(\Omega))}^2
    + \tau^{2\p+1}  \Norm{u^{(\p+1)}}_{L^{2}(\Ik,H^1(\Omega))}^2 \big).
\end{split}
\]\normalsize

\end{theorem}
\begin{proof}
\noindent \textbf{Estimates on~$\gimel_1$.}
We recall that~$\uzeroh$ is the elliptic projection of~$\uzero$
as in~\eqref{elliptic-projector} onto~$\Vh$.
Standard polynomial approximation estimates imply 
\begin{equation} \label{estimate-gimel1}
\mum\gimel_1^2
= {\SemiNorm{(\uzero-\uzeroh)(\cdot)}_{1,\Omega}^2}
\lesssim {\frac{\h^{2\min{(\px,s-\frac12)}}}{\px^{2s-1}}
    \Norm{\uzero(\cdot)}_{s+\frac12,\Omega}^2}.
\end{equation}

\noindent \textbf{Estimates on~$\gimel_2$.}
We recall that~$\uoneh$ is the $L^2$ projection of~$\uone$ onto~$\Vh$.
Standard polynomial approximation estimates imply
\begin{equation} \label{estimate-gimel2}
\frac{\mum}{2} \gimel_2^2
= \Norm{(\uone-\uoneh)(\cdot)}_{0,\Omega}^2
\lesssim {\frac{\h^{2\min{(\px+1,s-\frac12)}}}{\px^{2s-1}}
        \Norm{\uone(\cdot)}_{s-\frac12,\Omega}^2}.
\end{equation}

\noindent \textbf{Estimates on~$\gimel_3$.}
Using identity~\eqref{relation:SS-Ptilde}
and the approximation estimate~\eqref{Schoetzau-Schwab-d},
we infer 
\footnotesize{
\begin{equation} \label{estimate-gimel3}
\frac{\mum}{4} \gimel_3^2
=  \Norm{(u - \Pcalpbftt u)'(0,\cdot)}_{0,\Omega}^2
=  \Norm{(u' - \Pcaltildepbftmo u')(0,\cdot)}_{0,\Omega}^2
\lesssim \left( \frac{\tauo}{\pot}\right)
        ^{2\min(\pot,s)-1}
        \Norm{u^{(s+1)}}_{L^2(I_1, L^2(\Omega))}^2.
\end{equation} 
}\normalsize

\medskip
\noindent \textbf{Estimates on~$\gimel_4$.}
Using the triangle inequality, we have
\footnotesize{\[
\frac{\mum}{4} \gimel_4^2
=   \Norm{ (\Pcalpbftt (I-\PihEcal )u)'(0,\cdot)}_{0,\Omega}^2
\le 2 \Norm{ ((I - \Pcalpbftt) (I-\PihEcal )u)'(0,\cdot)}_{0,\Omega}^2
    + 2 \Norm{ (I-\PihEcal )u'(0,\cdot)}_{0,\Omega}^2 =: \gimel_{4,1}^2 + \gimel_{4,2}^2.
\]}\normalsize
Following the proof of the bound on~$\gimel_{3}^2$
we get the following bound on~$\gimel_{4,1}^2$:
\[
\begin{split}
 \gimel_{4,1}^2
& \overset{\eqref{Schoetzau-Schwab-d}}{\lesssim}
\left(\frac{\tauo}{\pot}\right)
    \Norm{ (I-\PihEcal )u''}_{L^2(I_1;L^2(\Omega))}^2\\
&\lesssim \left(\frac{\tauo}{\pot}\right)
\inf_{\qpx \in L^2(I_1;\Vh)}
    \left( \Norm{u''-\qpx}_{L^2(I_1;L^2(\Omega))}^2
    + \Norm{\PihEcal (u''-\qpx)}_{L^2(I_1;L^2(\Omega))}^2 \right) \\
& \overset{\eqref{elliptic-projector-L2-stab}}{\lesssim}
       \left(\frac{\tauo}{\pot}\right)
\inf_{\qpx \in L^2(I_1;\Vh)}
    \left( \Norm{u''-\qpx}_{L^2(I_1;L^2(\Omega))}^2
    + \frac{\h^{2\alpha}}{\px^{2\alpha}}
    {\SemiNorm{u''-\qpx}_{L^2(I_1;H^1(\Omega))}^2} \right) \\
&  \lesssim \left(\frac{\tauo}{\pot}\right)
\left( \frac{\h^{2\min(\px+1,s-1)}}{\px^{2(s-1)}}
    \Norm{u''}_{L^2(I_1;H^{s-1}(\Omega))}^2
    + \frac{\h^{2\alpha}}{\px^{2\alpha}}
    \frac{\h^{2\min(\px+1,s-1)-2}}{\px^{2(s-1)-2}}
    \Norm{u''}_{L^2(I_1;H^{s-1}(\Omega))}^2 \right) \\
&  \lesssim \left(\frac{\tauo}{\pot}\right)
  \frac{\h^{2\min(\px+1,s-1)-2(1-\alpha)}}{\px^{2s-2-2(1-\alpha)}}
    \Norm{u''}_{L^2(I_1;H^{s-1}(\Omega))} ^2.
\end{split}
\]
As for the term~$\gimel_{4,2}$, we proceed similarly:
\small{\[
\begin{split}
 \gimel_{4,2}^2
& =  \Norm{ (I-\PihEcal )u_1(\cdot) }_{0,\Omega}^2
 \lesssim \inf_{\qpx \in L^2(I_1;\Vh)}
    \left( \Norm{(u_1-\qpx)(\cdot)}_{0,\Omega}^2
    + \Norm{\PihEcal (u_1-\qpx)(\cdot)}_{0,\Omega}^2 \right) \\
& \overset{\eqref{elliptic-projector-L2-stab}}{\lesssim}
\left( \frac{\h^{2\min(\px+1,s-\frac12)}}{\px^{2s-1}}
    \Norm{u_1(\cdot)}_{s-\frac12,\Omega}^2
    + \frac{\h^{2\alpha}}{\px^{2\alpha}}
    \frac{\h^{2\min(\px+1,s-\frac12)-2}}{\px^{2s-3}}
    \Norm{u_1(\cdot)}_{s-\frac12,\Omega}^2 \right) \\
& \lesssim \frac{\h^{2\min(\px+1,s-\frac12)-2(1-\alpha)}}{\px^{2s-1-2(1-\alpha)}}
    \Norm{u_1(\cdot)}_{s-\frac12,\Omega}^2.
\end{split}
\]}
Collecting the two displays above gives
\small{\begin{equation} \label{estimate-gimel4}
\mum \gimel_4^2
\lesssim  \left(\frac{\tauo}{\pot}\right)
  \frac{\h^{2\min(\px+1,s-1)-2(1-\alpha)}}{\px^{2s-2-2(1-\alpha)}}
    \Norm{u''}_{L^2(I_1;H^{s-1}(\Omega))} ^2
+   {\frac{\h^{2\min(\px+1,s-\frac12)-2(1-\alpha)}}{\px^{2s-1-2(1-\alpha)}}
    \Norm{u_1(\cdot)}_{s-\frac12,\Omega}^2}.
\end{equation}}\normalsize

\medskip
\noindent \textbf{Estimates on~$\gimel_5$.}
Using Proposition~\ref{proposition:Pcaltildept-L2-approx},
we get
\small{\begin{equation} \label{estimate-gimel5}
\begin{split}
\frac{\gimel_{5}^2}{4}
& = \frac{\tm}{\mum^2} \sum_{n=1}^m
    \Norm{(I-\Pcalpbftt)\Deltax u}_{L^2(\In;L^2(\Omega))}^2
\lesssim \frac{\tm}{\mum^2} \sum_{n=1}^m
    \left( \frac{\taun}{\pnt}\right)^{2\min(s-1,\pnt+1)}
    \Norm{\Deltax u^{(s-1)}}_{L^{2}(\In; L^2(\Omega))}^2.
\end{split}
\end{equation}}\normalsize

\noindent \textbf{Estimates on~$\gimel_6$.}
Using~\eqref{elliptic-projector-L2-stab}
and the fact that~$\PihEcal$ preserves spatial polynomials
of degree~$\px$, we get
\small{\begin{equation} \label{estimate-gimel6}
\begin{split}
\frac{\gimel_{6}^2}{4}
& = \frac{\tm}{\mum^2} \sum_{n=1}^m
    \Norm{(I-\PihEcal)u''}_{L^2(\In;L^2(\Omega))}^2\\
& \lesssim \frac{\tm}{\mum^2} \sum_{n=1}^m
    \inf_{\qpx \in L^2(\In;\Vh)}
    \left( \Norm{u''-\qpx}_{L^2(\In;L^2(\Omega))}^2
    + \Norm{\PihEcal (u''-\qpx)}_{L^2(\In;L^2(\Omega))}^2 \right)    \\
& \lesssim \frac{\tm}{\mum^2} \sum_{n=1}^m
\inf_{\qpx \in L^2(\In;\Vh)}
    \left( \Norm{u''-\qpx}_{L^2(\In;L^2(\Omega))}^2
    + \frac{\h^{2\alpha}}{\px^{2\alpha}}
    {\SemiNorm{u''-\qpx}_{L^2(\In;H^1(\Omega))}^2} \right) \\
& \lesssim \frac{\tm}{\mum^2} \sum_{n=1}^m
    \left(
    \frac{\h^{2\min(\px+1,s-1)}}{\px^{2(s-1)}}
    \Norm{u''}_{L^2(\In;H^{s-1}(\Omega))}^2
    + \frac{\h^{2\alpha}}{\px^{2\alpha}}
    \frac{\h^{2\min(\px+1,s-1)-2}}{\px^{2(s-1)-2}}
    \Norm{u''}_{L^2(\In;H^{s-1}(\Omega))}^2 \right) \\
& \lesssim \frac{\tm}{\mum^2} \sum_{n=1}^m
    \frac{\h^{2\min(\px+1,s-1)-2(1-\alpha)}}{\px^{2s-2-2(1-\alpha)}}
    \Norm{u''}_{L^2(\In;H^{s-1}(\Omega))}^2.
\end{split}
\end{equation}}\normalsize

\noindent \textbf{Estimates on~$\gimel_7$.}
Using~\eqref{elliptic-projector-L2-stab}
and the definition of~$\rho$ in~\eqref{error:splitting} gives
\footnotesize{\begin{equation} \label{starting:gimel7}
\begin{split}
\frac{\gimel_7^2}{2}
&   =  \Norm{\rho'}^2_{L^\infty(\Ik;L^2(\Omega))}
    =  \Norm{(u-\Pcalpbftt \PihEcal u)'}
        ^2_{L^\infty(\Ik;L^2(\Omega))} \\
&   \lesssim
    \Big( \Norm{(u-\PihEcal u)'}^2_{L^{\infty}(\Ik;L^2(\Omega))}
    + \Norm{(u-\Pcalpbftt u)'}^2_{L^\infty(\Ik; L^2(\Omega))}
    + \left( \frac{\h}{\px} \right)^{2\alpha}
      \SemiNorm{(u-\Pcalpbftt u)'}^2_{L^\infty(\Ik;H^1(\Omega))} \Big).
\end{split}
\end{equation}}\normalsize
First, we focus on the first term on the right-hand side:
for any~$\qpx$ in $W^{1,\infty}(\Ik; \Vh)$,
\[
\begin{split}
& \Norm{(u-\PihEcal u)'}^2_{L^{\infty}(\Ik;L^2(\Omega))}
 \lesssim \Norm{(u-\qpx)'}^2_{L^{\infty}(\Ik;L^2(\Omega))}
        + \Norm{\PihEcal(u-\qpx)'}^2_{L^{\infty}(\Ik;L^2(\Omega))}\\
& \overset{\eqref{elliptic-projector-L2-stab}}{\lesssim}
     \Norm{(u-\qpx)'}^2_{L^{\infty}(\Ik;L^2(\Omega))}
     + \frac{\h^{2\alpha}}{\p^{2\alpha}}
        \SemiNorm{(u-\qpx)'}^2_{L^{\infty}(\Ik;H^1(\Omega))} \\
& \lesssim \frac{\h^{2\min(s,\pkt)+2}\tauk}{(\pkt)^{2s+2}}
         \big(\tauk^{-1} \Norm{u'}_{L^{\infty}(\Ik;H^{s+1}(\Omega))}^2\big)
         + \frac{\h^{2\min(s,\pkt)+2\alpha}\tauk}{(\pkt)^{2s+2\alpha}}
         \big(\tauk^{-1} \Norm{u'}_{L^{\infty}(\Ik;H^{s+1}(\Omega))}^2\big) \\
& = \frac{\h^{2\min(s,\pkt)+2}\tauk}{(\pkt)^{2s+2}}
         \left(1 + \frac{\h^{-2+2\alpha}}{\px^{-2+2\alpha}} \right)
         \big(\tauk^{-1} \Norm{u'}_{L^{\infty}(\Ik;H^{s+1}(\Omega))}^2\big).
\end{split}
\]
Next, we deal with the second term in the parenthesis on the right-hand side:
using a polynomial inverse inequality as in~\cite[eq. (3.6.4)]{Schwab:1998},
for any~$\qpkt$ in~$\Pbb_{\pkt}(\Ik;L^2(\Omega))$,
\[
\begin{split}
& \Norm{(u-\Pcalpbftt u)'}^2_{L^\infty(\Ik;L^2(\Omega))}
  \lesssim \Norm{(u-\qpnt)'}^2_{L^\infty(\Ik;L^2(\Omega))}
            + \Norm{\Pcalpbftt'(u-\qpkt)}^2_{L^\infty(\Ik;L^2(\Omega))} \\
& \lesssim \Norm{(u-\qpkt)'}^2_{L^\infty(\Ik;L^2(\Omega))}
            + \frac{(\pkt)^2}{\tauk}
            \Norm{\Pcalpbftt'(u-\qpkt)}_{L^2(\Ik;L^2(\Omega))}^2 \\
& \overset{\eqref{Schoetzau-Schwab-derivatives-c}}{\lesssim}
 \Norm{(u-\qpkt)'}^2_{L^\infty(\Ik;L^2(\Omega))}
 + \frac{(\pkt)^2}{\tauk}  \Norm{(u-\qpkt)'}^2_{L^2(\Ik;L^2(\Omega))}
 + \tauk \Norm{(u-\qpkt)''}_{L^2(\Ik;L^2(\Omega))}^2.
\end{split}
\]
Standard polynomial approximation properties in 1D give
\[
\Norm{(u-\qpkt)'}^2_{L^\infty(\Ik;L^2(\Omega))}
\lesssim \frac{\tauk^{2\min(s,\pkt)-1}}{(\pkt)^{2\min(s,\pkt)-1}}
        \Norm{u^{(s+1)}}_{L^{2}(\Ik;L^2(\Omega))}^2,
\]
\[
\frac{(\pkt)^2}{\tauk} \Norm{(u-\qpkt)'}^2_{L^2(\Ik;L^2(\Omega))}
\lesssim  \frac{\tauk^{2\min(s,\pkt)-1}}{(\pkt)^{2\min(s,\pkt)-2}}
        \Norm{u^{(s+1)}}_{L^{2}(\Ik;L^2(\Omega))}^2,
\]
and
\[
\tauk \Norm{(u-\qpkt)''}^2_{L^2(\Ik;L^2(\Omega))}
\lesssim  \frac{\tauk^{2\min(s,\pkt)-1}}{(\pkt)^{2\min(s,\pkt)-2}}
        \Norm{u^{(s+1)}}_{L^{2}(\Ik;L^2(\Omega))}^2.
\]
We collect the four displays above and get
\[
\Norm{(u-\Pcalpbftt u)'}^2_{L^\infty(\Ik;L^2(\Omega))}
  \lesssim \frac{\tauk^{2\min(s,\pkt)-1}}{(\pkt)^{2\min(s,\pkt)-2}}
        \Norm{u^{(s+1)}}_{L^{2}(\Ik;L^2(\Omega))}^2.
\]
We proceed similarly for the third term on the right-hand side of~\eqref{starting:gimel7}:
\[
\begin{split}
\left(\frac{\h}{\px} \right)^{2\alpha}
               \SemiNorm{(u-\Pcalpbftt u)'}^2_{L^\infty(\Ik;H^1(\Omega))}
& \lesssim \left(\frac{\h}{\px} \right)^{2\alpha}
        \frac{\tauk^{2\min(s,\pkt)-1}}{(\pkt)^{2\min(s,\pkt)-2}}
        \Norm{u^{(s+1)}}_{L^{2}(\Ik;H^1(\Omega))}^2 \\
& = \frac{\h^{2\alpha-2}}{\px^{2\alpha}}
        \frac{\tauk^{2\min(s,\pkt)-1}}{(\pkt)^{2\min(s,\pkt)-2}}
        \h^{2}\Norm{u^{(s+1)}}_{L^{2}(\Ik;H^1(\Omega))}^2.
\end{split}
\]
Recalling~\eqref{mun}, we arrive at
\begin{equation}\label{estimate-gimel7}
\begin{split}
\gimel_7^2
& \lesssim
    \left[  \frac{\h^{2\min(s,\pkt)+2}\tauk}{(\pkt)^{2s+2}}
         \left(1 + \frac{\h^{-2+2\alpha}}{\px^{-2+2\alpha}} \right)
         \tauk^{-1}\Norm{u}_{W^{1,\infty}(\Ik;H^{s+1}(\Omega))}^2
         \right. \\
& \qquad\qquad \left. + \frac{\tauk^{2\min(s,\pkt)-1}}{(\pkt)^{2\min(s,\pkt)-2}}
        \left( \Norm{u^{(s+1)}}_{L^{2}(\Ik;L^2(\Omega))}^2
        +  \frac{\h^{2\alpha-2}}{\px^{2\alpha}} \h^{2}
        \Norm{u^{(s+1)}}_{L^{2}(\Ik;H^1(\Omega))}^2 \right) \right].
\end{split}
\end{equation}

\noindent \textbf{Estimates on~$\gimel_8$.}
Using the stability of the elliptic projector~$\PihEcal$ in $H^1(\Omega)$
and the 1D Sobolev embedding in \cite[eq. (1.3)]{Ilyin-Laptev-Loss-Zelik:2016},
we write
\begin{equation} \label{starting:gimel8}
\begin{split}
\frac{\gimel_8^2}{2}
&   =  {\SemiNorm{\rho}_{L^\infty(\Ik;H^1(\Omega))}^2}
    =  {\SemiNorm{u-\Pcalpbftt\PihEcal u}_{L^\infty(\Ik;H^1(\Omega))}^2} \\
& \lesssim
    \Big( {\SemiNorm{u- \PihEcal u}_{L^\infty(\Ik;H^1(\Omega))}^2}
    + {\SemiNorm{\PihEcal(\text{Id}-\Pcalpbftt) u}_{L^\infty(\Ik;H^1(\Omega))}^2} \Big) \\
&   \le \Big( {\SemiNorm{u- \PihEcal u}_{L^\infty(\Ik;H^1(\Omega))}^2}
        + {\SemiNorm{u- \Pcalpbftt u}_{L^\infty(\Ik;H^1(\Omega))}^2} \Big) \\
&   \le \Big( \SemiNorm{u- \PihEcal u}_{L^\infty(\Ik;H^1(\Omega))}^2
        + \SemiNorm{u- \Pcalpbftt u}_{L^2(\Ik;H^1(\Omega))}
        \SemiNorm{(u- \Pcalpbftt u)'}_{L^2(\Ik;H^1(\Omega))} \Big).
\end{split}
\end{equation}
As for the first term on the right-hand side,
polynomial approximation properties (in space) give
\[
\SemiNorm{u- \PihEcal u}_{L^\infty(\Ik;H^1(\Omega))}^2
\le \SemiNorm{u- \qpx}_{L^\infty(\Ik;H^1(\Omega))}^2
\lesssim \frac{\h^{2\min(\px,s)}\tauk^{-1}}{\px^{2s}}
    \big(\tauk \Norm{u}_{L^\infty(\Ik;H^{s+1}(\Omega))}^2 \big).
\]
As for the second term on the right-hand side of~\eqref{starting:gimel8},
we use~\eqref{Pcaltildept-L2-approx-bis}
and~\eqref{Schoetzau-Schwab-derivatives-b}, and obtain
\[
\begin{split}
\SemiNorm{u- \Pcalpbftt u}_{L^2(\Ik;H^1(\Omega))}
    \SemiNorm{(u- \Pcalpbftt u)'}_{L^2(\Ik;H^1(\Omega))}
\lesssim \left( \frac{\tauk}{\pkt}\right)^{2\min(s,\pkt+1)-1}
            \Norm{u^{(s)}}_{L^{2}(\Ik,H^1(\Omega))}^2.
\end{split}
\]
Combining the above displays
and recalling~\eqref{mun} entail
\small{\begin{equation} \label{estimate-gimel8}
\gimel_8^2
\lesssim
\left( \frac{\h^{2\min(\px,s)}\tauk^{-1}}{\px^{2s}}
    \big(\tauk\Norm{u}_{L^\infty(\Ik;H^{s+1}(\Omega))}^2 \big)
    + \left( \frac{\tauk}{\pkt}\right)^{2\min(s,\pkt+1)-1}
    \Norm{u^{(s)}}_{L^{2}(\Ik,H^1(\Omega))}^2 \right).
\end{equation}}\normalsize

\noindent \textbf{error estimates.}
The assertion follows combining \eqref{estimate-gimel1}, \eqref{estimate-gimel2},
\eqref{estimate-gimel3}, \eqref{estimate-gimel4}, \eqref{estimate-gimel5},
\eqref{estimate-gimel6}, \eqref{estimate-gimel7},  and \eqref{estimate-gimel8}.
\end{proof}

The a priori estimates in Theorem~\ref{theorem:h-p-convergence} can be simplified
under the assumption of elliptic regularity,
i.e., assuming the parameter~$\alpha$
in~\eqref{elliptic-projector-L2-stab} to be~$1$
(this happens for instance if~$\Omega$ is convex);
under more regularity on the solution;
requiring the isotropy of the spatial and time meshes;
fixing the polynomial degrees in space and time
(in what follows, $\p$ denotes the polynomial degree of the scheme).

Notably, the next result contains two estimates:
the first one holds true for smooth exact solutions and
the rate in the time step has optimal order~$\p$;
the second one holds true for exact solutions
with finite total Sobolev regularity indices~$s$
smaller than or equal to $\p$
and is optimal in terms of the Sobolev scaling.

\begin{corollary} \label{corollary:h-p-convergence-alpha=1}
Let~$u$ and~$\uh$ be the solutions to~\eqref{weak-problem} and~\eqref{Walkington:method},
and~$\eh$ be as in~\eqref{error:splitting}.
We assume that the initial conditions $\uzero$ and~$\uone$ in~\eqref{weak-problem}
and $u$ are smooth
in the sense $s$ is larger than or equal to
$\pnt +2$ on each $I_n$ for all $n=1,\dots,N$.
Assume that the parameter~$\alpha$ in~\eqref{elliptic-projector-L2-stab}
is equal to~1, i.e., elliptic regularity holds for the domain~$\Omega$.
Let~$\taun=\tau$ for all~$n=1,\dots,N$, and choose $h=O(\tau)$.
For given~$\p$ in~$\Nbb$,
we further demand that~$\px=\p$ and~$\pnt=\p$
for all $n=1,\dots,N$.
Recall that~$m=m(\xih)$ is defined in~\eqref{selecting-m}.
Then, the following a priori error estimate is valid:
\[
\max_{n=1}^N \left(
    \Norm{\eh'}_{L^\infty(\In; L^2(\Omega))}^2
    + \SemiNorm{\eh}_{L^\infty(\In; H^1(\Omega))}^2 \right)^\frac12
\lesssim  \tau^{\p}.
\]
On the other hand, if $s$ is smaller than or equal to~$\p$,
then the following a priori error estimate is valid:
\begin{equation} \label{h-p-convergence-alpha=1:singular}
\max_{n=1}^N\left(
    \Norm{\eh'}_{L^\infty(\In; L^2(\Omega))}^2
    + \SemiNorm{\eh}_{L^\infty(\In; H^1(\Omega))}^2 \right)^\frac12
\lesssim  \tau^{s-\frac12}.
\end{equation}
\end{corollary}
\begin{proof}
The proof boils down to using Theorem~\ref{theorem:h-p-convergence}
and note that, for the estimates
of the terms~$\gimel_3$ and~$\gimel_7$
in~\eqref{abstract-estimates:maximum:norm},
we can use
\[
\Norm{u^{(p+1)}}_{L^{2}(\Ik;L^2(\Omega))}
\leq  \tauk^\frac12 \Norm{u^{(p+1)}}_{L^{\infty}(\Ik;L^2(\Omega))}.
\]
All other~$\gimel_j$ terms, $j=1,2,4,5,6,8$,
already gave $\mathcal O(\tau^\p)$ rates.
\end{proof}

\begin{remark} \label{remark:jumps-in-error}
The error measure in~\eqref{full-error-a-priori} does not involve the jumps
at the time nodes of the first time derivative.
Some comments about this feature are in order.
\begin{itemize}
    \item Lemma~\ref{lemma:estimates-xih} involves an error measure containing
    the sum of jump terms up to~$m$,
    $m=m(\xih)$ as in~\eqref{selecting-m};
    a modification of the proofs of
    Theorems~\ref{theorem:abstract-convergence} and~\ref{theorem:h-p-convergence}
    would lead to estimates for norms of the error
    containing jumps, but only up to the time node~$\tm$
    and leading to an optimal convergence rate $\mathcal{O}(\tau^{p-\frac12})$,
    which is however suboptimal for the first
    term in the error measure~\eqref{full-error-a-priori}.
    \item Walkington's strategy~\cite{Walkington:2014}
    involves the use of a special test function,
    leading to errors measured in $L^\infty$-type norms in time.
    Such norms are different from $L^2$-type norms in time,
    which are typically obtained by testing with ``more standard'' functions
    and typically come together with jumps at the time nodes.
    \item Optimal convergence rates for a norm involving jumps in our setting
    may be derived by using stability estimates
    obtained by using the test function $\Wh=\uh'$.
    This is shown, for instance, in \cite[eq. (4.1)]{Walkington:2014}
    and would give an optimal convergence rate for the jump terms.
    \item Even though we did not prove convergence rates for the jumps
    on the theoretical level, in Section~\ref{section:numerical-experiments} below,
    we shall investigate their practical behaviour.
\end{itemize}
\end{remark}
\section{A posteriori error estimates for the semi-discrete in time method} \label{section:apos}
This section is concerned with introducing an error estimator
for the semi-discrete in time method~\eqref{Walkington:method-time-semi-discrete},
and prove fully explicit, reliable
a posteriori estimates
for the error measured in the $L^\infty(L^2)$ norm
under extra assumptions discussed in Section~\ref{subsection:apos-assumption}.
In Section~\ref{subsection:reconstruction-operator},
we introduce a novel reconstruction operator
and exhibit its approximation properties,
which are instrumental in the a posteriori error estimates
given in Section~\ref{subsection:error-estimator-estimates}.

\subsection{Data assumptions} \label{subsection:apos-assumption}
Throughout, we make the following assumption,
which is instrumental in deriving the a posteriori
error estimates:
for~$u$ and~$\U$ solutions to~\eqref{weak-problem}
to~\eqref{Walkington:method-time-semi-discrete},
\begin{equation} \label{consequence:convexity}
\text{$\Deltax u \in L^1(0,T;L^2(\Omega))$}
\qquad(\text{which implies }
\text{$\Deltax \U$ belongs to $L^1(0,T;L^2(\Omega))$).}
\end{equation}
Assumption~\eqref{consequence:convexity}
can be proven for instance under certain conditions
on the data, as detailed in the next result.

\begin{proposition} \label{proposition:regularity-Laplacian}
Assume that
\begin{equation} \label{assumption:convexity}
\begin{split}
& \text{the spatial domain~$\Omega$ is convex;}\\
& \text{$\uzero$ and~$\uone$ belong to $H^2(\Omega)\cap H^1_0(\Omega)$ and $H^1_0(\Omega)$};\\
& \text{the right-hand side $f$ belongs to $H^1(0,T; L^2(\Omega))$}.
\end{split}
\end{equation}
Then, the solution~$u$ to~\eqref{weak-problem} is such that
\[
u \in   L^\infty(0,T; H^2(\Omega))
        \cap W^{1,\infty}(0,T; H^1_0(\Omega))
        \cap W^{2,\infty}(0,T; L^2(\Omega))
        \cap W^{3,\infty}(0,T; H^{-1}(\Omega)).
\]
In particular, property~\eqref{consequence:convexity} holds true. 
\end{proposition}
\begin{proof}
A proof for domains with sufficiently smooth boundary
can be found in \cite[Theorem 5, Chapter 7.2]{Evans:2022}
and is based on the Faedo-Galerkin technique,
based on taking the limits of expansions into eigenfunctions.
In turns, the regularity in space only depends on the regularity
of the eigenfunctions of~$\Deltax$,
which on convex domains~\eqref{assumption:convexity} is~$H^2(\Omega)$;
see, e.g., \cite{Grisvard:2011}.
\end{proof}

\subsection{A reconstruction operator} \label{subsection:reconstruction-operator}
In view of deriving a posteriori error estimates
in Section~\ref{subsection:error-estimator-estimates} below,
we discuss here the properties of a generalization
of the reconstruction operator introduced in \cite{Makridakis-Nochetto:2006}
for parabolic problems; see also \cite{Schoetzau-Wihler:2010, Holm-Wihler:2018}
for a proof of $\p$-approximation properties of the operator in~\cite{Makridakis-Nochetto:2006}.

Given a Hilbert space~$\mathcal X$
with inner product $(\cdot,\cdot)_{\mathcal X}$
and induced norm~$\Norm{\cdot}_{\mathcal X}$,
and $\V$ in $\mathcal C^0(0,T; \mathcal X)$ with
$\V{}_{|\In}$ in $\Pbb_{\pnt}(\In;\mathcal X)$
and $(\V^-)'(0,\cdot)$ is an element in~$\mathcal X$,
let~$\Vhat$ be piecewise defined
for all $n=1,\dots,N$ as
\begin{equation} \label{definition-Rcalppo}
\begin{cases}
(\Vhat'', \qpntmo)_{L^2(\In;\mathcal X)}
= (\V'',\qpntmo)_{L^2(\In;\mathcal X)}
        + (\jump{\V'}(\tnmo,\cdot), \qpntmo(\tnmo,\cdot))_{\mathcal X} \\
\Vhat(\tnmo,\cdot) = \V(\tnmo,\cdot) ,
\;\;
\Vhat'(\tnmo,\cdot) = (\V^-)'(\tnmo,\cdot)
\text{ in~$\mathcal X$ }
\quad \forall \qpntmo \in \Pbb_{\pnt-1}(\In;\mathcal X).
\end{cases}
\end{equation}
The corresponding operator for parabolic problems
also allows for spatial mesh changes,
see, e.g., \cite{Georgoulis-Lakkis-Wihler:2021},
a topic that is still open for the wave equation
in second order formulation
and that we shall investigate in the future.

We begin by proving the following property of~$\Vhat$.
\begin{proposition} \label{proposition:smoothness-reconstruction}
The function~$\Vhat$ is a $\mathcal C^1$
piecewise polynomial in time reconstruction
of a~$\mathcal C^0$ piecewise polynomial in time.
\end{proposition}
\begin{proof}
The assertion follows using the last condition
in~\eqref{definition-Rcalppo}
and proving that
\begin{equation} \label{what-to-prove-smoothness-time}
\Vhat' (\tn,\cdot) = (\V^-)'(\tn,\cdot),
\qquad\qquad
\Vhat (\tn,\cdot) = \V(\tn,\cdot)
\quad \text{ in } \mathcal X.
\end{equation}
\noindent \textbf{Proving the first identity in~\eqref{what-to-prove-smoothness-time}.}
Taking~$\qpntmo=\cX$ to be fixed in $\mathcal X$ below
and independent of time
in the first line of~\eqref{definition-Rcalppo}
for all $n=1,\dots,N$,
and integrating by parts on both sides lead to
\[
\begin{split}
& (\Vhat'(\tn,\cdot)-\Vhat'(\tnmo,\cdot),
    \cX)_{\mathcal X} \\
& = ( (\V^-)'(\tn,\cdot) -
\underbrace{(\V^+)'(\tnmo,\cdot)
  + (\V^+)'(\tnmo,\cdot)}_{=0}
  - (\V^-)'(\tnmo,\cdot), \cX)_{\mathcal X} .
\end{split}
\]
Using that
$\Vhat'(\tnmo,\cdot) = (\V^-)'(\tnmo,\cdot)$ in~$\mathcal X$
and taking~$\cX$ equal to
$\Vhat'(\tn,\cdot) - (\V^-)'(\tn,\cdot)$,
we deduce the first identity in~\eqref{what-to-prove-smoothness-time}.
\medskip

\noindent \textbf{Proving the second identity in~\eqref{what-to-prove-smoothness-time}.}
We take~$\qpntmo=(t-\tnmo)\ctildeX$
with~$\ctildeX$ to be fixed in~$\mathcal X$ below
and independent of time
in the first condition of~\eqref{definition-Rcalppo}
for all $n=1,\dots,N$,
integrate by parts, and get
\small{\[
\begin{split}
& (\Vhat'(\tn,\cdot), \taun \ctildeX)_{\mathcal X}
  - (\Vhat',\ctildeX)_{L^2(\In;\mathcal X)}
  = ((\V^-)'(\tn,\cdot),\taun\ctildeX)_{\mathcal X}
  - (\V',\ctildeX)_{L^2(\In;\mathcal X)}.
\end{split}
\]}\normalsize
Using $\Vhat' (\tn,\cdot) = (\V^-)'(\tn,\cdot)$ entails
\[
(\Vhat',\ctildeX)_{L^2(\In;\mathcal X)}
= (\V',\ctildeX)_{L^2(\In;\mathcal X)}.
\]
Integrating by parts again gives
\[
(\Vhat(\tn,\cdot)
- \Vhat(\tnmo,\cdot), \ctildeX)_{\mathcal X}
= (\V(\tn,\cdot) - \V(\tnmo,\cdot),\ctildeX)_{\mathcal X}.
\]
Using that~$\Vhat(\tnmo,\cdot) = \V(\tnmo,\cdot)$ in~$\mathcal X$,
which is the first initial condition in~\eqref{definition-Rcalppo},
and choosing~$\ctildeX$ equal to
$\Vhat(\tn,\cdot) - \V(\tn,\cdot)$
yield the second identity in~\eqref{what-to-prove-smoothness-time}.
\end{proof}

We have additional properties on the operator~$\Vhat$.
\begin{lemma} \label{lemma:properties-a-la-Wihler}
Consider~$\V$ in $\mathcal C^0(0,T; \mathcal X)$ with
$\V{}_{|\In}$ in $\Pbb_{\pnt}(\In;\mathcal X)$
and $\Vhat$ as in~\eqref{definition-Rcalppo}.
For all~$n=1,\dots,N$, the following identities hold true:
\begin{subequations} \label{properties-a-la-Wihler}
\begin{align}
\Norm{(\V-\Vhat)'}_{L^2(\In;\mathcal X)}^2
&  = \taun c_1(\pnt)^2
        \Norm{\jump{\V'}(\tnmo,\cdot)}^2_{\mathcal X} ,\label{Wihler-1}  \\
 \Norm{(\V-\Vhat)'}_{L^\infty(\In;\mathcal X)}^2
&  = \Norm{\jump{\V'}(\tnmo,\cdot)}^2_{\mathcal X},  \label{Wihler-2} \\
 \Norm{\V-\Vhat}_{L^2(\In;\mathcal X)}^2
& \le \taun^3 c_2(\pnt)^2 \Norm{\jump{\V'}(\tnmo,\cdot)}^2_{\mathcal X} , \label{Wihler-3}
\end{align}
\end{subequations}
where
\begin{subequations}
\begin{align}
c_1(\pnt)^2
& : = \frac{\pnt}{(2\pnt-1)(2\pnt+1)} ,
& c_1(\pnt) \approx (\pnt)^{-\frac12} , \label{copnt}\\
c_2(\pnt)^2 
& :=
\begin{cases}
\frac14 \frac{\pnt}{(\pnt-2)(\pnt-1)(2\pnt-1)(2\pnt+1)}
                        & \text{if } \pnt\ge 3,\\
\frac{2}{15\pi^2}       & \text{if } \pnt=2,
\end{cases}
& c_2(\pnt) \approx (\pnt)^{-\frac32} \label{ctwpnt} .
\end{align}
\end{subequations}
\end{lemma}
\begin{proof}
Identity~\eqref{Wihler-1} is proven in \cite[Theorem~2]{Schoetzau-Wihler:2010}.
Identity~\eqref{Wihler-2} is essentially proven in \cite[Lemma~1]{Holm-Wihler:2018}.
As for inequality~\eqref{Wihler-3}, we first observe that definition~\eqref{definition-Rcalppo},
an integration by parts,
the first identity in~\eqref{what-to-prove-smoothness-time},
and the smoothness of~$\Vhat$
imply, for all $n=1,\dots,n$,
\small{\[
\begin{split}
&-(\jump{\Vpr}(\tnmo,\cdot),\qpntmo(\tnmo,\cdot))_{\mathcal X}
 \overset{\eqref{definition-Rcalppo}}{=}
 ((\V-\Vhat)'',\qpntmo)_{L^2(\In;\mathcal X)} \\
& \overset{(\text{IBP})}{=}
    -((\V-\Vhat)',\qpntmo')_{L^2(\In;\mathcal X)}\\
& \qquad    + ((\V - \Vhat)'(\tn,\cdot),\qpntmo(\tn,\cdot))_{\mathcal X}
           - ((\V - \Vhat)'(\tnmo,\cdot),\qpntmo(\tnmo,\cdot))_{\mathcal X}\\
& \overset{\eqref{what-to-prove-smoothness-time}}{=}
    -((\V-\Vhat)',\qpntmo')_{L^2(\In;\mathcal X)}
    -(\jump{\Vpr}(\tnmo,\cdot),\qpntmo(\tnmo,\cdot))_{\mathcal X}
    \qquad \forall \qpntmo \in \Pbb_{\pnt-1}(\In;\mathcal X).
\end{split}
\]}\normalsize
\noindent \textbf{Proving~\eqref{Wihler-3} for~$\pnt$ larger than~$2$.}
A further integration by parts
and the smoothness of~$\V$ and~$\Vhat$
at the endpoints of each $\In$ entail
\begin{equation}\label{eq: L2 orthogonality}
 (\V-\Vhat,\qpntmo'')_{L^2(\In;\mathcal X)}
 = 0 \qquad \forall \qpntmo \in \Pbb_{\pnt-1}(\In;\mathcal X).
\end{equation}
For all $t$ in $\In$, we pick
\[
\qpntmo(t,\cdot) := \int_{\tnmo}^t \int_{\tnmo}^s
            \Pizpbftmth(\V-\Vhat)(r,\cdot) dr \ ds
        \qquad \text{in } \mathcal X
\]
and deduce
\begin{equation} \label{Paris-1}
(\V-\Vhat,
   \Pizpbftmth(\V-\Vhat))_{L^2(\In;\mathcal X)} = 0 .
\end{equation}
Using \cite[Theorem~3.11]{Schwab:1998} entails
\begin{equation}\label{Paris-2}
\Norm{(I-\Pizpbftmth)(\V-\Vhat)}_{L^2(\In;\mathcal X)}
\le \frac12\frac{\taun}{\sqrt{(\pnt-2)(\pnt-1)}}
    \Norm{(\V-\Vhat)'}_{L^2(\In;\mathcal X)}.
\end{equation}
Approximation properties as in \cite[Theorem~3.11]{Schwab:1998} imply
\[
\begin{split}
&\Norm{\V-\Vhat}_{L^2(\In;\mathcal X)}^2
 = (\V-\Vhat, \V-\Vhat)_{L^2(\In;\mathcal X)} \\
& \overset{\eqref{Paris-1}}{=}
    (\V-\Vhat,
    (I-\Pizpbftmth)(\V-\Vhat))_{L^2(\In;\mathcal X)}  \\
& \overset{\eqref{Paris-2}}{\le}
    \frac12\frac{\taun}{\sqrt{(\pnt-2)(\pnt-1)}}
    \Norm{\V-\Vhat}_{L^2(\In;\mathcal X)}
        \Norm{(\V-\Vhat)'}_{L^2(\In;\mathcal X)}.
\end{split}
\]
The first line in inequality~\eqref{Wihler-3}
follows recalling identity~\eqref{Wihler-1}.

\noindent \textbf{Proving~\eqref{Wihler-3} for~$\pnt$ equal to~$2$.}
The last display above modifies as follows:
the 1D Poincar\'e inequality in time implies
\[
\Norm{\V-\Vhat}_{L^2(\In;\mathcal X)}^2
\le \frac{\taun^2}{\pi^2}\Norm{(\V-\Vhat)'}_{L^2(\In;\mathcal X)}^2.
\]
The second line in inequality~\eqref{Wihler-3}
follows recalling identity~\eqref{Wihler-1}
with $\pnt=2$.

\end{proof}

\subsection{An error estimator for the semi-discrete in time method} \label{subsection:error-estimator-estimates}
We construct a computable error estimator
and prove fully explicit, reliable a posteriori error estimates.
We define~$\xi$ and~$m$ such that
\begin{equation} \label{choice-xi}
\Norm{(u-\Uhat)(\xi,\cdot)}_{0,\Omega}
= \Norm{u-\Uhat}_{L^{\infty}(0,T;L^2(\Omega))}
\qquad\qquad  \text{with } \xi \in \Im.
\end{equation}
We introduce
\begin{equation} \label{test-function-Baker}
\vB(t,\cdot)
= \int_t^\xi (u-\Uhat) (s,\cdot) ds
\qquad\qquad \text{in } \mathcal X.
\end{equation}
The function~$\vB$ belongs
to~$\mathcal C^1(0,T; L^2(\Omega))$.
In definition~\eqref{definition-Rcalppo}, we set
\begin{equation} \label{set-initial-condition}
    \Uhat'(0,\cdot) = \uone(\cdot)
    \qquad\qquad \text{in } L^2(\Omega).
\end{equation}
We state an auxiliary technical result.
\begin{lemma} \label{lemma:Trefethen}
Given~$\vB$ and~$\xi$ as in~\eqref{test-function-Baker} and~\eqref{choice-xi},
the following estimates hold true:
\begin{subequations} \label{L-infinity-norm-approximation}
\begin{align}
\inf_{\qpnt \in \Pbb_{\pnt}(\In;L^2(\Omega))}
\Norm{\vB -\qpnt}_{L^{\infty}(I_n;L^2(\Omega))}
&\leq \frac{\taun}{\pi} c_3(\pnt) \Norm{\vB^\prime}_{L^{\infty}(I_n;L^2(\Omega))}
\label{infty-1},\\
\Norm{\vB}_{L^{\infty}(t_{m-1},\xi;L^2(\Omega))}
&\leq \taum \Norm{\vB^\prime}_{L^{\infty}(t_{m-1},\xi;L^2(\Omega))}
     \label{infty-2},
\end{align}
\end{subequations}
where we have set,
also for future convenience,
\begin{equation} \label{cthpnt}
c_3(\pnt):=
\begin{cases}
\sqrt{\pi}          & \text{if } \pnt= 0,1,2\\
\frac{1}{\pnt-2}    & \text{if } \pnt\ge3,
\end{cases}
\qquad\qquad\qquad
c_3(\pnt) \approx (\pnt)^{-1}.
\end{equation}
\end{lemma}
\begin{proof}
Inequality~\eqref{infty-1} in the case $\pnt$ larger than $2$
is the Bochner version of \cite[Theorem 7.2]{Trefethen-2019}
up to a scaling argument.
Inequality~\eqref{infty-1} in the case~$\pnt$
equal to~$0$, $1$, and~$2$
follows from \cite[eq. (1.3)]{Ilyin-Laptev-Loss-Zelik:2016}
and the 1D Poincar\'{e} inequality in time.
Inequality~\eqref{infty-2} follows from
the following computations:
\[
\begin{split}
&   \Norm{\vB}_{L^{\infty}(t_{m-1},\xi;L^2(\Omega))}
    = \sup_{t\in(\tmmo,\xi)}
    \Norm{\vB(t,\cdot)}_{L^2(\Omega)}
    = \sup_{t\in(\tmmo,\xi)}
    \Big( \int_\Omega \vert \vB(t,\xbf) \vert^2 d\xbf \Big)^\frac12 \\
&   = \sup_{t\in(\tmmo,\xi)}
    \Big( \int_\Omega \Big\vert \int_{t}^\xi 
            (u-\Uhat) (s,\xbf) ds \Big\vert^2 d\xbf \Big)^\frac12
    \le \sup_{t\in(\tmmo,\xi)}
    \Big( \int_\Omega \vert \xi-t \vert 
        \int_{t}^\xi 
            \vert (u-\Uhat) (s,\xbf) \vert^2 ds \ d\xbf \Big)^\frac12 \\
&   \le \taum^\frac12 \sup_{t\in(\tmmo,\xi)}
    \Big( \int_\Omega
        \int_{t}^\xi 
            \vert (u-\Uhat) (s,\xbf) \vert^2 ds \ d\xbf \Big)^\frac12
    = \taum^\frac12 \sup_{t\in(\tmmo,\xi)}
    \Big( \int_\Omega
        \int_{t}^\xi 
            \vert \vB' (s,\xbf) \vert^2 ds \ d\xbf \Big)^\frac12 \\
&   \le \taum \Norm{\vB^\prime}_{L^{\infty}(t_{m-1},\xi;L^2(\Omega))}.
\end{split}
\]
\end{proof}

\begin{proposition} \label{proposition:initial-reliability-estimate}
Let assumption~\eqref{consequence:convexity} hold true.
Consider~$u$ and $\U$ the solutions to~\eqref{weak-problem}
and~\eqref{Walkington:method-time-semi-discrete},
and the operator~$\Uhat$ in~\eqref{equivalent-definition-Pcalpt}.
Given~$\vB$, $\xi$ and~$m$, $c_2(\pnt)$, and~$c_3(\pnt)$
as in~\eqref{test-function-Baker}, \eqref{choice-xi},
\eqref{ctwpnt}, and~\eqref{cthpnt}, we have
\small{\begin{equation} \label{initial-reliability-estimate}
\begin{split}
&\Norm{u-\Uhat}_{L^{\infty}(0,T;L^2(\Omega))}
\\
 & \le 2\Bigg( \Big( \sum_{n=1}^{m-1}
\frac{\taun}{\pi} c_3(\pnt-1) \Norm{f-\Pizpbftmo f}_{L^1(\In; L^2(\Omega))}
+ \taum \Norm{f-\Pizpbftmo f}_{L^1(\Im; L^2(\Omega))}
\Big)\\
& \quad + \Big( \sum_{n=1}^{m-1}
    \frac{\taun}{\pi}  c_3(\pnt-1)
    \Norm{\Deltax( \U-  \Pizpbftmo  \U)}_{L^1(\In; L^2(\Omega))}
    +\taum \Norm{\Deltax( \U-  \Pizpbftmo  \U)}_{L^1(\Im; L^2(\Omega))} \Big) \\
& \quad +  \Big( \sum_{n=1}^{m-1}
        \frac{\taun^3}{\pi} c_2(\pnt) c_4(\pnt-3)
        \Norm{\jump{ \Deltax \U'}(\tnmo,\cdot)}_{0,\Omega}
    + \taum^3 c_2(\pmt)
        \Norm{\jump{\Deltax \U'}(\tmmo,\cdot)}_{0,\Omega} \Big)  \Bigg),
\end{split}
\end{equation}}\normalsize
with~$c_4(\pnt-3)$ defined as
\begin{equation} \label{c4}
c_4(\pnt-3):=
\begin{cases}
 \pi \frac{|t_m - t_{n-1}|}{\taun}       & \text{if } \pnt= 2\\
c_3(\pnt-3)    & \text{if } \pnt\ge3.
\end{cases}
\end{equation}
\end{proposition}
\begin{proof}
Let~$\Uhat$ be the reconstruction operator in~\eqref{definition-Rcalppo}.
Using~\eqref{Walkington:method-time-semi-discrete}, we have the identity
\[
\begin{split}
& (\Uhat'',\V)_{L^2(\In; L^2(\Omega))}
- (\Deltax \U, \V)_{L^2(\In; L^2(\Omega))} \\
& = (f, \V )_{L^2(\In; L^2(\Omega))} \qquad \forall \V  \in \Pbb_{\pnt-1}(\In; \HozOmega)
\quad \forall n=1,\dots,N.
\end{split}
\]
Recall the time semi-discrete right-hand side in~\eqref{rhs:semidiscrete}.
Due to~\eqref{consequence:convexity}, for all $n=1,\dots,N$,
we have the following identity in $L^2(\Omega)$
inside each time interval:
\begin{equation} \label{discrete-star-identity}
\Uhat'' - \Pizpbftmo \Deltax \U = \Pizpbftmo f.
\end{equation}
Using $\Deltax \Pizpbftmo \U= \Pizpbftmo \Deltax\U$,
and subtracting~\eqref{discrete-star-identity} to~\eqref{strong-problem},
we get the following identity in $L^2(\In;L^2(\Omega))$:
\begin{equation} \label{initial-identity-for-error-equation}
(u-\Uhat)''
- \Deltax (u- \Uhat)
= (f- \Pizpbftmo f)
+\Deltax( \Uhat-  \Pizpbftmo  \U) .
\end{equation}
Next, we derive an error equation
testing~\eqref{initial-identity-for-error-equation}
with a particular function mimicking that
proposed in~\cite[eq. (3.7)]{Baker:1976}, i.e.,
the function in~\eqref{test-function-Baker}.
It is immediate to check
\begin{equation} \label{derivative-test-Baker}
-\vB'(t,\cdot) = (u-\Uhat)(t,\cdot)
\end{equation}
and
\begin{equation} \label{v-xi=0}
    \vB(\xi,\cdot) = 0.
\end{equation}
Multiplying~\eqref{initial-identity-for-error-equation}
with the function in~\eqref{test-function-Baker},
integrating in space,
and integrating by parts,
we arrive at
the following identity in $L^2(0,T)$:
\[
\begin{split}
&((u-\Uhat)'',\vB)_{0,\Omega}
+ a(u-\Uhat,\vB)\\
& \qquad \qquad
= (f-\Pizpbftmo f, \vB)_{0,\Omega}
+(\Deltax( \Uhat-  \Pizpbftmo  \U), \vB)_{0,\Omega} .
\end{split}
\]
Using~\eqref{derivative-test-Baker}, we readily deduce
the following identity in $L^2(0,T)$:
\[
((u-\Uhat)'',\vB)_{0,\Omega} - a(\vB',\vB) = (f-\Pizpbftmo f, \vB)_{0,\Omega}
    +(\Deltax( \Uhat-  \Pizpbftmo  \U), \vB)_{0,\Omega}.
\]
The following trivial but crucial identity holds true:
given sufficiently smooth in time functions~$\aleph$ and~$\beth$,
we have
\[
(\aleph'', \beth)_{0,\Omega}
= (\aleph',\beth)'_{0,\Omega} - (\aleph', \beth')_{0,\Omega}
\qquad\qquad \text{ in } L^2(0,T).
\]
This identity and~\eqref{derivative-test-Baker} imply
\[
\begin{split}
&((u-\Uhat)',\vB)'_{0,\Omega}
+ ((u-\Uhat)',u-\Uhat)_{0,\Omega}
- a(\vB',\vB)\\
&\qquad \qquad
= (f-\Pizpbftmo f, \vB)_{0,\Omega}
    +(\Deltax( \Uhat-  \Pizpbftmo  \U), \vB)_{0,\Omega}
\qquad\qquad \text{ in } L^1(0,T).
\end{split}
\]\normalsize
Equivalently, we write
\[
\begin{split}
& ((u-\Uhat)',\vB)'_{0,\Omega}
+\frac12 (\Norm{u-\Uhat}_{0,\Omega}^2)'
- \frac12 a(\vB,\vB)' \\
& \qquad \qquad
= (f-\Pizpbftmo f, \vB)_{0,\Omega}
    +(\Deltax( \Uhat-  \Pizpbftmo  \U), \vB)_{0,\Omega}
    \qquad\qquad \text{ in } L^1(0,T).
\end{split}
\]
We integrate in time over $(0,\xi)$,
$\xi$ as in the choice of the test function in~\eqref{test-function-Baker}.
We arrive at
\[
\begin{split}
& \frac12 \Norm{(u-\Uhat)(\xi,\cdot)}_{0,\Omega}^2
 - \frac12 a(\vB(\xi,\cdot),\vB(\xi,\cdot))
= \frac12 \Norm{(u-\Uhat)(0,\cdot)}_{0,\Omega}^2
- \frac12 a(\vB(0,\cdot),\vB(0,\cdot))\\
& \qquad
    +  \int_{0}^\xi  \left((f-\Pizpbftmo f, \vB)_{0,\Omega}
    +(\Deltax( \Uhat-  \Pizpbftmo  \U), \vB)_{0,\Omega}
    - ((u-\Uhat)',\vB)'_{0,\Omega} \right) dt.
\end{split}
\]
Using~\eqref{definition-Rcalppo} and~\eqref{set-initial-condition},
we have $(u-\Uhat)'(0,\cdot)=0$, whence we write
\[
\begin{split}
& - \int_{0}^\xi  ((u-\Uhat)',\vB)'_{0,\Omega} dt \\
& = - ((u-\Uhat)'(\xi,\cdot),\vB(\xi,\cdot))_{0,\Omega}
    + ((u-\Uhat)'(0,\cdot),\vB(0,\cdot))_{0,\Omega}\\
& \overset{\eqref{v-xi=0}}{=}  ((u-\Uhat)'(0,\cdot),\vB(0,\cdot))_{0,\Omega}
    = 0.
\end{split}
\]
Combining the two above displays again with~\eqref{v-xi=0},
adding and subtracting $\U$ in the last term on the right-hand side,
and using that $(u-\Uhat)(0,\cdot)=0$ yield
\[
\begin{split}
& \frac12 \Norm{(u-\Uhat)(\xi,\cdot)}_{0,\Omega}^2
    +  \frac12 a(\vB(0,\cdot),\vB(0,\cdot))
     \\
& = \int_{0}^\xi  \left((f-\Pizpbftmo f, \vB)_{0,\Omega}
    + (\Deltax( \U-  \Pizpbftmo  \U), \vB)_{0,\Omega}
    + (\Deltax( \Uhat-  \U), \vB)_{0,\Omega} \right) dt.
\end{split}
\]
Let~$\qpntmo$ and~$\qpntmth$ realize~\eqref{infty-1}
of degree~$\pnt-1$ and~~$\pnt-3$ (with $\qpntmth=0$ for $\pnt=2$), respectively.
Using the properties of~$\Pizpbftmo$,
the orthogonal property~\eqref{eq: L2 orthogonality},
and the choice of~$\xi$ in~\eqref{choice-xi}
(including the fact that~$\xi$ lies in the time interval~$\Im$),
we end up with
\small{\begin{equation} \label{T1-T2-T3-T4-T5-T6}
\begin{split}
& \frac12 \Norm{u-\Uhat}^2_{L^{\infty}(0,T;L^2(\Omega))}
+  \frac12 a(\vB(0,\cdot),\vB(0,\cdot))
    \\
& \quad \leq  \sum_{n=1}^{m-1} \int_{I_n}  (f-\Pizpbftmo f, \vB - \qpntmo)_{0,\Omega} dt
        + \int_{t_{m-1}}^\xi  (f-\Pizpbftmo f, \vB )_{0,\Omega} dt\\
& \quad + \sum_{n=1}^{m-1} \int_{I_n} (\Deltax( \U- \Pizpbftmo  \U), \vB - \qpntmo)_{0,\Omega} dt
+ \int_{t_{m-1}}^\xi (\Deltax( \U - \Pizpbftmo \U), \vB)_{0,\Omega} dt \\
& \quad+ \sum_{n=1}^{m-1} \int_{I_n} (\Deltax( \Uhat-  \U),
        \vB - \qpntmth)_{0,\Omega}dt
+ \int_{t_{m-1}}^\xi (\Deltax( \Uhat-  \U), \vB )_{0,\Omega}dt = \sum_{i=1}^6 T_i.
\end{split}
\end{equation}}\normalsize
As for the the term~$T_1$,
we use H\"{o}lder's inequality,
inequality~\eqref{infty-1} (with constant~$c_3(\pnt-1)$),
and the fact that
$\Norm{\vB'}_{L^{\infty}(\In; L^2(\Omega))}$ is smaller than
$\Norm{\vB'}_{L^{\infty}(0,T; L^2(\Omega))}$,
and end up with
\begin{equation} \label{bound-T1}
\begin{split}
T_1
&  \le \sum_{n=1}^{m-1} \Norm{f-\Pizpbftmo f}_{L^1(\In; L^2(\Omega))}
       \Norm{\vB-\qpntmo}_{L^\infty(\In; L^2(\Omega))} \\
& \le \sum_{n=1}^{m-1} \frac{\taun }{\pi}c_3(\pnt-1)
            \Norm{f-\Pizpbftmo f}_{L^1(\In; L^2(\Omega))}
            \Norm{\vB'}_{L^{\infty}(\In; L^2(\Omega))} \\
& \le \Big(\sum_{n=1}^{m-1} \frac{\taun }{\pi}c_3(\pnt-1)
            \Norm{f-\Pizpbftmo f}_{L^1(\In; L^2(\Omega))}\Big)
            \Norm{\vB'}_{L^{\infty}(0,T; L^2(\Omega))}.
\end{split}
\end{equation}
Similarly, using \eqref{infty-2}
and the fact that $\Norm{\vB'}_{L^{\infty}(t_{m-1},\xi; L^2(\Omega))}$
is smaller than $\Norm{\vB'}_{L^{\infty}(0,T; L^2(\Omega))}$,
we also infer
the following bound on the term~$T_2$:
\begin{equation} \label{bound-T2}
\begin{split}
T_2
&  \le\Norm{f-\Pizpbftmo f}_{L^1(t_{m-1},\xi; L^2(\Omega))}
       \Norm{\vB}_{L^\infty(t_{m-1},\xi; L^2(\Omega))} \\
& \le \taum \Norm{f-\Pizpbftmo f}_{L^1(\Im; L^2(\Omega))}
                \Norm{\vB'}_{L^{\infty}(t_{m-1},\xi; L^2(\Omega))} \\
& \le \taum \Norm{f-\Pizpbftmo f}_{L^1(\Im; L^2(\Omega))}
             \Norm{\vB'}_{L^{\infty}(0,T; L^2(\Omega))}.
\end{split}
\end{equation}
Next, we bound the term~$T_3$
using again~\eqref{infty-1}
with constant $c_3(\pnt-1)$:
\begin{equation} \label{bound-T3}
\begin{split}
 T_3
 &\leq \sum_{n=1}^{m-1}  \Norm{\Deltax( \U-  \Pizpbftmo \U)}_{L^1(\In; L^2(\Omega))}   \Norm{\vB-\qpntmo}_{L^\infty(\In; L^2(\Omega))}  \\
&\leq \sum_{n=1}^{m-1}   \frac{\taun }{\pi} c_3(\pnt-1)
        \Norm{\Deltax( \U-  \Pizpbftmo \U)}_{L^1(\In; L^2(\Omega))}
        \Norm{\vB^{\prime}}_{L^\infty(\In; L^2(\Omega))}  \\
&\leq \Big(\sum_{n=1}^{m-1}   \frac{\taun}{\pi} c_3(\pnt-1)
        \Norm{\Deltax( \U-  \Pizpbftmo \U)}_{L^1(\In; L^2(\Omega))}\Big)
        \Norm{\vB^{\prime}}_{L^\infty(0,T; L^2(\Omega))}.
\end{split}
\end{equation}
We bound the term~$T_4$ based on~\eqref{infty-2}
and proceeding as in the bound of~$T_3$:
\begin{equation} \label{bound-T4}
\begin{split}
 T_4
 & \leq \Norm{\Deltax( \U- \Pizpbftmo \U)}_{L^1(t_{m-1},\xi; L^2(\Omega))}
    \Norm{\vB}_{L^\infty(t_{m-1},\xi; L^2(\Omega))} \\
& \leq \taum \Norm{\Deltax( \U- \Pizpbftmo \U)}_{L^1(t_{m-1},\xi; L^2(\Omega))}
            \Norm{\vB^{\prime}}_{L^\infty(t_{m-1},\xi; L^2(\Omega))}  \\
& = \taum \Norm{\Deltax( \U- \Pizpbftmo \U)}_{L^1(\Im; L^2(\Omega))}   \Norm{\vB^{\prime}}_{L^\infty(0,T; L^2(\Omega))}.
\end{split}
\end{equation}
Using estimates~\eqref{infty-1} with constant $c_3(\pnt-3)$
and~\eqref{Wihler-3} with constant $c_2(\pnt)$,
we show an upper bound on the term~$T_5$ for $\pnt\geq 3$:
\begin{equation} \label{bound-T5}
\begin{split}
 T_5 &\leq  \sum_{n=1}^{m-1}
      \tau_n^{\frac12}  \Norm{\Deltax( \Uhat-  \U)}_{L^2(I_n;L^2(\Omega))}
  \Norm{\vB-\qpntmth}_{L^\infty(\In; L^2(\Omega))}  \\
& \leq  \sum_{n=1}^{m-1}
  \frac{\taun^{\frac32}}{\pi} c_3(\pnt-3)  \Norm{\Deltax( \Uhat-  \U)}_{L^2(I_n;L^2(\Omega))}
  \Norm{\vB^{\prime}}_{L^\infty(\In; L^2(\Omega))}  \\
&  \le \left( \sum_{n=1}^{m-1}
    \frac{\taun^3}{\pi} c_2(\pnt)c_3(\pnt-3)
        \Norm{\jump{ \Deltax \U'}(\tnmo,\cdot)}_{0,\Omega} \right)
         \Norm{\vB^{\prime}}_{L^\infty(0,T; L^2(\Omega))}.
\end{split}
\end{equation}
For $\pnt=2$, using the definition of~$\vB$
and~$\xi$ in~\eqref{test-function-Baker} and~\eqref{choice-xi},
and the fact that $\vert \xi-\tnmo \vert$
is larger than $\vert \xi-\tn \vert$, we have
\[
\Norm{\vB}_{L^\infty(\In; L^2(\Omega))} \leq |\xi -t_{n-1}|\Norm{\vB'}_{L^\infty(\In; L^2(\Omega))} 
\leq |t_{m} -t_{n-1}|\Norm{\vB'}_{L^\infty(\In; L^2(\Omega))}.
\]
Using the definition of $c_4(\pnt-3)$ in~\eqref{c4},
we have the following bound on the term~$T_5$
for all~$\pnt$ larger than or equal to~$2$:
\begin{align*}
    T_5 & \leq  \sum_{n=1}^{m-1}
      \tau_n^{\frac12}  \Norm{\Deltax( \Uhat-  \U)}_{L^2(I_n;L^2(\Omega))}
  \Norm{\vB}_{L^\infty(\In; L^2(\Omega))}  \\
  &\leq \sum_{n=1}^{m-1}
 |t_m -t_{n-1}| \taun^{\frac12} \Norm{\Deltax( \Uhat-  \U)}_{L^2(I_n;L^2(\Omega))}
  \Norm{\vB^{\prime}}_{L^\infty(\In; L^2(\Omega))}  \\
  &  \le \left( \sum_{n=1}^{m-1}
    |t_m -t_{n-1}| \taun^2 c_2(\pnt)
        \Norm{\jump{ \Deltax \U'}(\tnmo,\cdot)}_{0,\Omega} \right)
         \Norm{\vB^{\prime}}_{L^\infty(0,T; L^2(\Omega))}.
\end{align*}
We bound the term~$T_6$
using estimates~\eqref{Wihler-3}
and~\eqref{infty-2}:
\begin{equation} \label{bound-T6}
\begin{split}
 T_6 &\leq
      \taum^{\frac12}  \Norm{\Deltax(\Uhat-U)}_{L^2(t_{m-1},\xi;L^2(\Omega))}
  \Norm{\vB}_{L^\infty(t_{m-1},\xi; L^2(\Omega))}  \\
& \leq
  \taum^{\frac32} \Norm{\Deltax(\Uhat-U)}_{L^2(\Im;L^2(\Omega))}
   \Norm{\vB^{\prime}}_{L^\infty(0,T; L^2(\Omega))}\\
&  \le   \taum^3 c_2(\pmt)
        \Norm{\jump{ \Deltax \U'}(\tmmo,\cdot)}_{0,\Omega}
         \Norm{\vB^{\prime}}_{L^\infty(0,T; L^2(\Omega))}.
\end{split}
\end{equation}
Finally, using the fact that $\Norm{\vB^{\prime}}_{L^\infty(0,T; L^2(\Omega))}$
is equal to $\Norm{u-\Uhat}_{L^{\infty}(0,T;L^2(\Omega))}$,
the assertion follows inserting~\eqref{bound-T1}, \eqref{bound-T2},
\eqref{bound-T3}, \eqref{bound-T4}, \eqref{bound-T5},
and~\eqref{bound-T6} in~\eqref{T1-T2-T3-T4-T5-T6}.
\end{proof}

\begin{remark}
With the notation as in the proof of Proposition~\ref{proposition:initial-reliability-estimate},
for the case $\pnt=2$, the parameter $c_{4}(\pnt-3)$ in~\eqref{c4}
scales as $\taun^{-1}$, which reduces by one order
the convergence rate of the term~$T_5$ compared to the case $\pnt\geq3$.
On the other hand, the convergence rate
for the term~$T_5$
for~$\pnt\geq3$ is one order higher
than that of the error $\Norm{u-\U}_{L^{\infty}(0,T;L^2(\Omega))}$.
Therefore, for~$\pnt=2$, the term~$T_5$ converges
with the same rate as that of the error
$\Norm{u-\U}_{L^{\infty}(0,T;L^2(\Omega))}$.
\end{remark}

The error measure on the left-hand side of~\eqref{initial-reliability-estimate}
contains the reconstruction operator~$\Uhat$,
which we do not want to compute in practice.
For this reason, we elaborate further.
Recall that~$m$ is as in~\eqref{choice-xi}.
Consider the time steps and time polynomial degree
distributions~$\taubold$ and~$\pbft$ as discussed
in Section~\ref{section:introduction};
$c_1(\pnt)$, $c_2(\pnt)$, $c_3(\pnt)$, and~$c_4(\pnt)$
as defined in~\eqref{copnt}, \eqref{ctwpnt}, \eqref{cthpnt}, and~\eqref{c4}.
Define
\[
\eta_{1}: =  \max_{n=1,\dots, N} \taun (c_1(\pnt)c_2(\pnt))^{\frac12}
                \Norm{\jump{\U'}(\tnmo,\cdot)}_{0,\Omega}.
\]
Given~$m$ as in~\eqref{choice-xi}, for all $ n=1, \dots,m-1$,
we also define
\footnotesize{\[
\eta_{2,n} : =
\begin{cases}
\frac{2}{\pi}  \Big( \taun c_3(\pnt-1)
        \Norm{\Deltax(\U-\Pizpbftmo  \U)}_{L^1(\In;L^2(\Omega))} \\
\qquad\qquad  + \taun^3 c_2(\pnt) c_4(\pnt-3)
        \Norm{\jump{ \Deltax \U'}(\tnmo,\cdot)}_{0,\Omega} \Big)
        & \text{if } n=1,\dots,m-1 \\
2 \Big( \taum \Norm{\Deltax( \U-  \Pizpbftmo  \U)}_{L^1(\Im; L^2(\Omega))}
     + c_2(\pnt)
       \taum^3 \Norm{\jump{ \Deltax \U'}(\tmmo,\cdot)}_{0,\Omega} \Big)     & \text{if } n=m,
\end{cases}
\]}\normalsize
and the data oscillation terms
\[
\oscnf :=
\begin{cases}
\frac{2\taun}{\pi} c_3(\pnt-1)
        \Norm{f-\Pizpbftmo f}_{L^1(\In; L^2(\Omega))}
        & \text{if } n=1,\dots,m-1 \\
2\taum \Norm{f-\Pizpbftmo f}_{L^1(\Im; L^2(\Omega))}
        & \text{if } n=m.
\end{cases}
\]
With this at hand, for $m$ as in~\eqref{choice-xi}, we introduce
\begin{equation} \label{eta2-oscf}
    \eta_{2}: =  \sum_{n=1}^m \eta_{2,n},
     \qquad\qquad
    \oscf: =  \sum_{n=1}^m \oscnf.
\end{equation}
We are now in a position to derive
a fully explicit, reliable a posteriori upper bound
for the $L^\infty(L^2)$ error
with respect to the error estimator
\begin{equation} \label{error-estimator}
    \eta:= \etao + \etatw.
\end{equation}

\begin{proposition} \label{proposition:reliability-estimate}
Let~$u$ and~$\U$ be the solutions to~\eqref{weak-problem}
and~\eqref{Walkington:method-time-semi-discrete};
$\eta$ as in~\eqref{error-estimator},
and $\oscf$ as in~\eqref{eta2-oscf}.
Then, the following bound holds true:
\begin{equation} \label{reliability:estimate}
\begin{split}
\Norm{u-\U}_{L^{\infty}(0,T;L^2(\Omega))} \le \eta  + \oscf.
\end{split}
\end{equation}
\end{proposition}
\begin{proof}
The triangle inequality implies
\[
\Norm{u-\U}_{L^{\infty}(0,T; L^2(\Omega))}
\le  \Norm{u-\Uhat}_{L^{\infty}(0,T; L^2(\Omega))}
 +  \Norm{\U-\Uhat}_{L^{\infty}(0,T; L^2(\Omega))}.
\]
Next, we pick~$\mtilde=1,\dots,N$ and~$\widetilde{\xi}$ such that
\[
\Norm{(\U-\Uhat)(\widetilde{\xi},\cdot)}_{0,\Omega}
= \Norm{\U-\Uhat}_{L^{\infty}(0,T;L^2(\Omega))}
\qquad\qquad  \text{with } \widetilde{\xi} \in \Imtilde.
\]
An upper bound for the first term in the maximum
is a consequence of~\eqref{initial-reliability-estimate}.
As for the second term, we use
the (1D in time) Sobolev embedding $H^1 \hookrightarrow L^{\infty}$,
see, e.g., \cite[eq. (1.3)]{Ilyin-Laptev-Loss-Zelik:2016},
and observe
\[
\begin{split}
\Norm{\U-\Uhat}_{L^{\infty}(0,T; L^2(\Omega))}^2
&\le \Norm{\U-\Uhat}_{L^2( \Imtilde; L^2(\Omega))}
        \Norm{(\U-\Uhat)'}_{L^2( \Imtilde; L^2(\Omega))}\\
& \overset{\eqref{Wihler-1}, \eqref{Wihler-3}}{\le}
 \tau_{\mtilde}^2 c_1(\pmtildet)c_2(\pmtildet)
     \Norm{\jump{\U'}(\tmmotilde,\cdot)}^2_{0,\Omega}.
\end{split}
\]
As~$\mtilde$ and~$\widetilde{\xi}$ cannot be determined a priori,
we further elaborate the above bound as follows:
\[
\begin{split}
\Norm{\U-\Uhat}_{L^{\infty}(0,T; L^2(\Omega))}^2
&\le
   \max_{n=1,\dots, N} \taun^2 c_1(\pnt)c_2(\pnt)
     \Norm{\jump{\U'}(\tnmo,\cdot)}^2_{0,\Omega}.
\end{split}
\]
The assertion follows combining the above displays.
\end{proof}

\begin{remark} \label{remark:negative-norms}
Property~\eqref{consequence:convexity} is crucial in deriving bounds
\eqref{reliability:estimate}. 
In principle, we may also assume the weaker requirement
that $\Deltax \U$ belongs to $L^2(0,T; H^{-1}(\Omega))$
but end up with an error estimator involving a negative norm,
which is harder to realize in practice than the $L^2(0,T;L^2(\Omega))$ norm.
\end{remark}

\section{Numerical experiments}\label{section:numerical-experiments}
We assess numerically the a priori and a posteriori
error estimates proven in
Theorem~\ref{theorem:h-p-convergence}
and Corollary~\ref{corollary:h-p-convergence-alpha=1},
and Proposition~\ref{proposition:reliability-estimate}.
The numerical experiments are conducted with the
\texttt{Gridap.jl} library~\cite{VerBa:22} in the
\texttt{Julia} programming language.

\paragraph*{Spatial and time discretization.}
In what follows, we always consider the spatial domain
$\Omega=(-1,1)^2$ partitioned into uniform
tensor-product meshes;
we use tensor Lagrangian (equidistributed) nodal basis functions.
As for the time discretization,
we take the Lagrangian basis functions
in time for simplicity.
Unless otherwise specified, $T=1$;
other values will be considered and detailed at the appropriate points.

\paragraph*{Error measures.}
We are interested in investigating
different error measures for the a priori and a posteriori error estimates.

Given~$u$ and~$\uh$ the solutions
to~\eqref{weak-problem} and~\eqref{Walkington:method},
and~$\eh:=u-\uh$,
we consider the two error measures appearing
on the left-hand side of~\eqref{full-error-a-priori} separately:
\begin{equation} \label{error-measure-apriori-1}
\max_{n=1}^N \Norm{\eh'}_{L^{\infty}(\In; L^2(\Omega))},
\qquad\qquad\qquad\qquad
\max_{n=1}^N \SemiNorm{\eh}_{L^\infty(\In; H^1(\Omega))} .
\end{equation}
We also consider the errors
\begin{equation} \label{error-measure-apriori-2}
\SemiNorm{\eh}_{L^2(0,T;H^1(\Omega))},
\qquad\qquad
\Norm{\eh'}_{L^2(0,T;L^2(\Omega))},
\qquad\qquad
\Norm{\eh}_{L^\infty(0,T;L^2(\Omega))},
\end{equation}
and the jump error
\begin{equation} \label{error-measure-a-priori-3}
    \Big(\sum_{n=1}^N \Norm{\jump{\eh'}(\tnmo,\cdot)}_{0,\Omega}^2 \Big)^{\frac12}.
\end{equation}
The $L^2$-type norms in time are computed
with Gau\ss-Legendre quadrature formulas of order~$2\pnt+3$;
the $L^\infty$-type norms in time are evaluated
at~$2\pnt+3$ equally distributed nodes in time for each time interval;
the spatial norms are computed by using
tensor product Gau\ss-Legendre
quadrature formulas of order~$2\px+3$.
\medskip

\paragraph*{Imposing Dirichlet boundary conditions.}
We consider homogeneous Dirichlet boundary conditions,
which are imposed strongly at the boundary degrees of freedom.
The case of inhomogeneous boundary conditions
is discussed in \cite[Sections 3 and 5.3]{Walkington:2014}.

\subsection{Uniform refinements}
We investigate the convergence rate
of the fully-discrete method~\eqref{Walkington:method}
under uniform time step
and polynomial degree in time refinements
in various norms for three tests.

The initial and boundary conditions, and the right-hand side
of~\eqref{strong-problem}
are computed accordingly to the explicit formula of the
different solutions we pick below.

\subsubsection{Uniform refinements: test case 1}
We consider the analytic solution
\begin{equation} \label{test-case:1}
    u(x,y,t) := (1-x^2)(1-y^2) \cos(4t).
\end{equation}
As for the spatial discretization, we fix~$\px=2$ and~$h=0.4$.
Since the exact solution is a quadratic polynomial in space
and we use a quadratic nodal tensor product basis,
up to machine precision, the spatial error is zero.

\paragraph*{Uniform time step refinements.}
For $\pnt$ in $\{2,3\}$,
we pick $\taubold$ in $\{2\times 10^{-1}, 10^{-1}, 5\times 10^{-2}, 2.5\times 10^{-2}, 1.25 \times 10^{-2} \}$;
for $\pnt$ in $\{4,5\}$, we pick
$\taubold$ in $\{2\times 10^{-1}, 1.25\times 10^{-1}, 9.09\times 10^{-2}, 7.15\times 10^{-2}, 5.88 \times 10^{-2} \}$.
We display the errors in~\eqref{error-measure-apriori-1},
\eqref{error-measure-apriori-2}, and~\eqref{error-measure-a-priori-3}
in Figure~\ref{fig:error_plot_tau_refine_case_1}.

\begin{figure}[htb]
    \centering
    \includegraphics[width=0.48\linewidth]{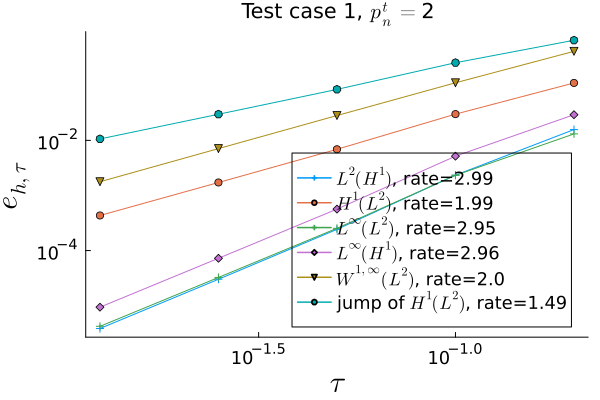}
    \includegraphics[width=0.48\linewidth]{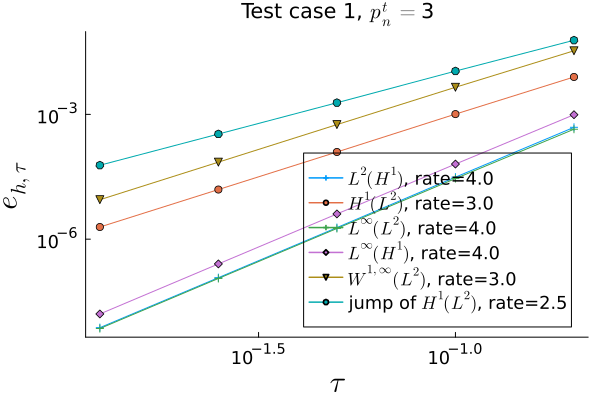}
    \includegraphics[width=0.48\linewidth]{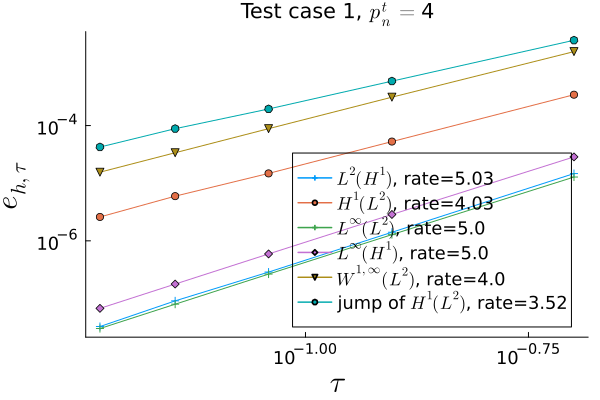}
    \includegraphics[width=0.48\linewidth]{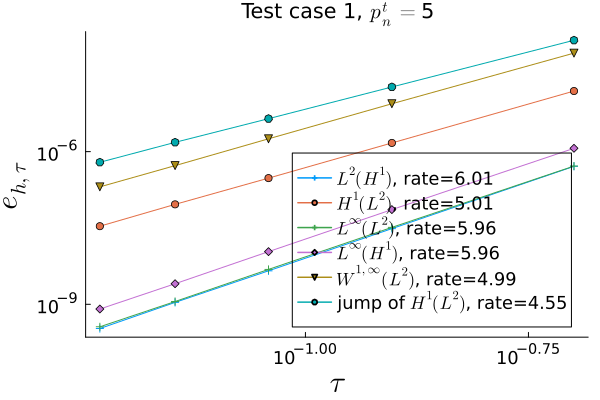}
    \caption{Exact solution as in~\eqref{test-case:1},
    uniform $\taubold$-refinement.} \label{fig:error_plot_tau_refine_case_1}
\end{figure}

The expected optimal convergence rates
as discussed in Corollary \ref{corollary:h-p-convergence-alpha=1}
are observed for the errors in~\eqref{error-measure-apriori-1};
the same convergence rate is achieved by the other error measures,
which is not covered by the theoretical results
from Section~\ref{section:a-priori}.
In particular, the error in~\eqref{error-measure-a-priori-3}
has order $\mathcal O(\tau^{\pnt-\frac12})$;
see also Remark~\ref{remark:jumps-in-error} for further comments on this point.

\paragraph*{Uniform polynomial degree in time refinements.}
We pick $\pnt$ in $\{2,3,4,5,6\}$
and $\taubold$ in $\{2\times 10^{-1},10^{-1}\}$.
We display the errors in~\eqref{error-measure-apriori-1}
\eqref{error-measure-apriori-2}, and~\eqref{error-measure-a-priori-3}
in Figure~\ref{fig:error_plot_p_refine_case_1}.

\begin{figure}[htb]
    \centering
    \includegraphics[width=0.48\linewidth]{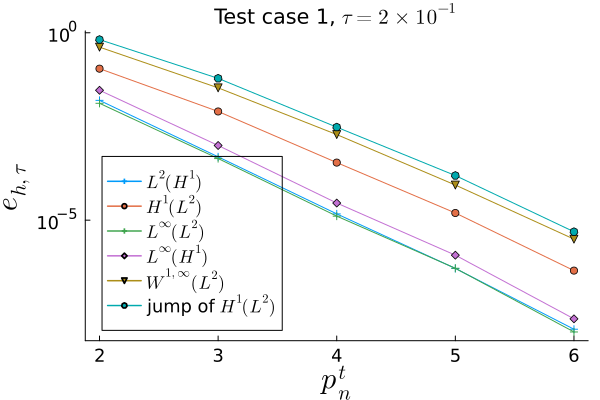}
    \includegraphics[width=0.48\linewidth]{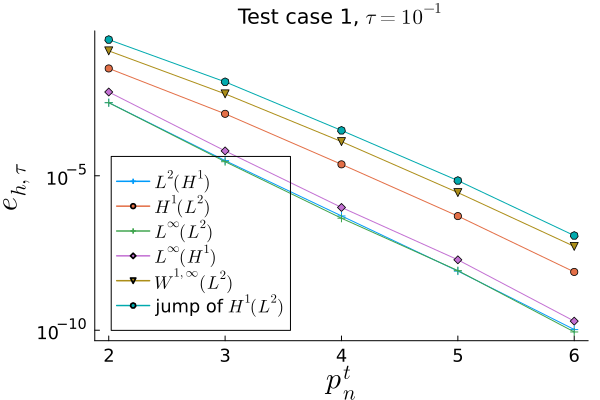}
    \caption{Exact solution as in~\eqref{test-case:1},
    uniform $\pnt$-refinement.} \label{fig:error_plot_p_refine_case_1}
\end{figure}

We observe exponential convergence rate for all the errors.
Even though this is not covered by the results in Section~\ref{section:a-priori},
we can expect this behaviour from the smoothness of the function in~\eqref{test-case:1}
and standard $\p$-FEM techniques~\cite{Schwab:1998}.

\subsubsection{Uniform refinements: test case 2} \label{Numerics singular solution}
We consider analytic in space solutions
\begin{equation} \label{test-case:2}
    u(x,y,t) := (1-x^2)(1-y^2)t^\alpha, \qquad \alpha>1.5,
\end{equation}
which belong to $H^{\alpha+\frac12}(0,T;\mathcal C^\infty(\Omega))$.
We fix~$\px=2$ and~$h=0.4$.

\paragraph*{Uniform time step refinements.}
We pick~$\pnt=2$ and
$\taubold$ in $\{2\times 10^{-1}, 10^{-1}, 5\times 10^{-2}, 2.5\times 10^{-2}, 1.25 \times 10^{-2},6.13 \times 10^{-3}, 3.06\times 10^{-3}, 1.53\times 10^{-3} \}$
for $\alpha=1.75$ and
$\taubold$ in $\{2\times 10^{-1}, 10^{-1}, 5\times 10^{-2}, 2.5\times 10^{-2}, 1.25 \times 10^{-2} \}$
for $\alpha=2.5$.
We do not consider higher polynomial degrees in time,
since they deliver the same convergence rates.
The results are exhibited in Figure \ref{fig:error_plot_h_refine_case_2}.

\begin{figure}[htb]
    \centering
    \includegraphics[width=0.48\linewidth]{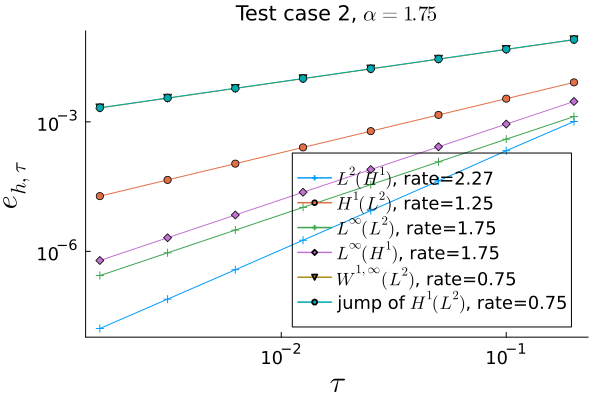}
    \includegraphics[width=0.48\linewidth]{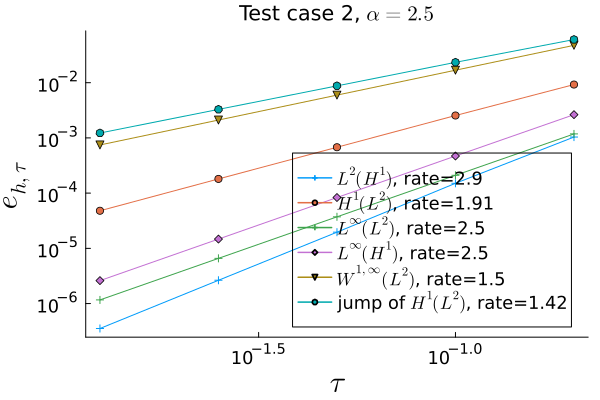}
    \caption{Exact solution as in~\eqref{test-case:2},
    uniform $\taubold$-refinement.} \label{fig:error_plot_h_refine_case_2}
\end{figure}

We observe optimal convergence rates for
the errors in~\eqref{error-measure-apriori-1}
as dictated by Corollary~\ref{corollary:h-p-convergence-alpha=1};
similar rates are achieved by the error measures
in~\eqref{error-measure-apriori-2}.
The error measured in the $W^{1,\infty}(L^2)$-seminorm
confirms estimate~\eqref{h-p-convergence-alpha=1:singular};
that seminorm converges with the same rate
of the error in~\eqref{error-measure-a-priori-3}.

\paragraph*{Uniform polynomial degree in time refinements.}
We pick $\pnt$ in $\{2,...,10\}$ and $\taubold = 0.2$.
The results are displayed in Figure~\ref{fig:error_plot_p_refine_case_2}.

\begin{figure}[htb]
    \centering
    \includegraphics[width=0.48\linewidth]{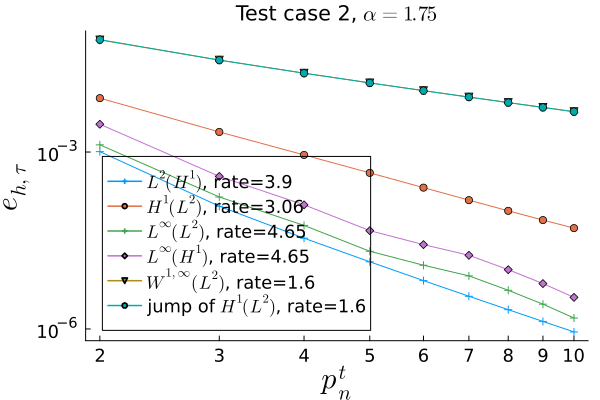}
    \includegraphics[width=0.48\linewidth]{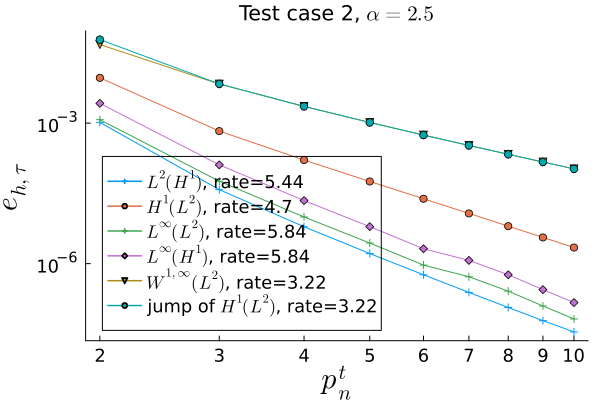}
    \caption{Exact solution as in~\eqref{test-case:2}, $\pnt$-refinement.} \label{fig:error_plot_p_refine_case_2}
\end{figure}

We observe doubling order convergence rate in $\pnt$,
which is standard in $\p$-FEM while approximating
functions with growth of $t^\alpha$ type \cite[Section 3.3.5]{Schwab:1998},
for the jump of the $H^1(L^2)$-seminorm,
whereas other quantities display a super-convergence phenomenon. 

\subsubsection{Simultaneous space--time uniform refinements: test case 3}
We consider the analytic solution 
\begin{equation} \label{test-case:3}
    u(x,y,t) := \sin(\pi n x)\sin(\pi m y)\cos(\omega \pi t),
\end{equation}
with $m=n=1$, $\omega = \sqrt{2}$,
and pick $\px=\pnt+1$ and $h=\taubold$, i.e.,
we are interested in simultaneous space and time refinements.
On the other hand, for $\pnt\in \{2,3\}$,
we choose $\taubold$ in $\{2\times 10^{-1}, 10^{-1}, 5\times 10^{-2}, 2.5\times 10^{-2}, 1.25 \times 10^{-2} \}$;
for $\pnt=4$, we choose $\taubold$ in
$\{2\times 10^{-1},10^{-1},6.67\times 10^{-2},5\times 10^{-2},4\times 10^{-2}\}$;
for $\pnt=5$, we choose
$\taubold$ in $\{5\times 10^{-1},3.33\times 10^{-1},2.5\times 10^{-1},2\times 10^{-1},1.67\times 10^{-1}\}$.
The results are displayed in Figure~\ref{fig:error_plot_tau_refine_case_3}.

\begin{figure}[htb]
    \centering
    \includegraphics[width=0.48\linewidth]{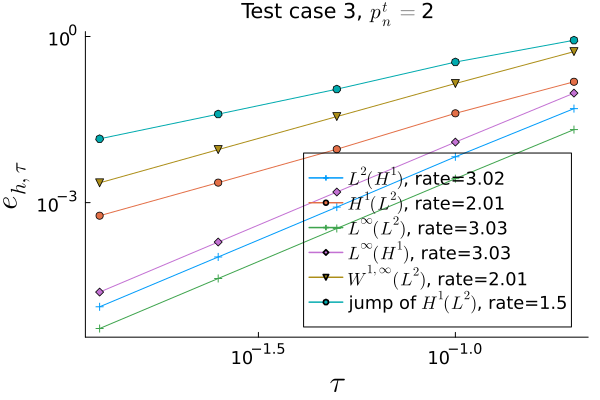}
    \includegraphics[width=0.48\linewidth]{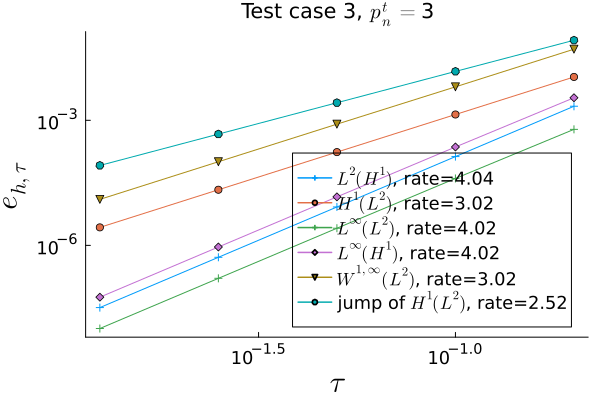}
    \includegraphics[width=0.48\linewidth]{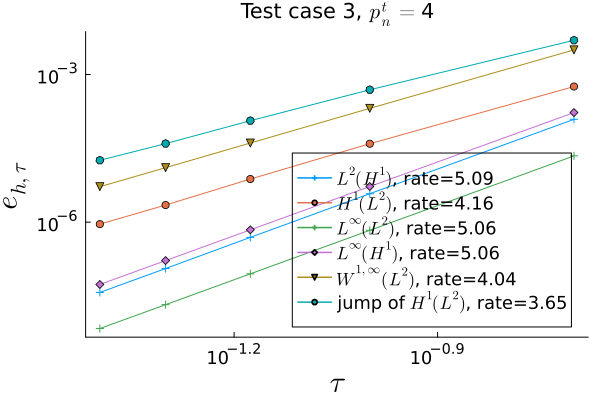}
    \includegraphics[width=0.48\linewidth]{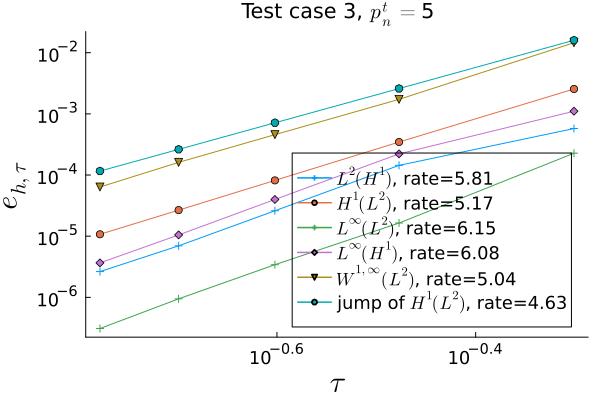}
    \caption{Exact solution as in~\eqref{test-case:3},  $\taubold$-refinement.} \label{fig:error_plot_tau_refine_case_3}
\end{figure}

We observe optimal convergence rates for the errors
\eqref{error-measure-apriori-1} and~\eqref{error-measure-apriori-2}.

\subsubsection{Long time behaviour: test case 3}
Here, we investigate the influence of the final time $T$ on the errors and estimators
for the test case 3 with exact solution in~\eqref{test-case:3}
with $m=n=1$ and $\omega=\sqrt{2}$.
In particular, we take $T=6,8,10$, $\px=\pnt=2,3$, $h=2\taubold=0.4$.
In Table~\ref{tab:error_estimator_long_time}, we report the errors and estimators
in the selected norms. We present only a few norms compared to the previous tests,
since all the $L^\infty$-in-time type norms
have similar behaviour to the $L^\infty(L^2)$-norm,
and all $L^2$-in-time type norms have similar behaviour to the $L^2(H^1)$-norm.

\begin{table}[h]
\caption{Exact solution as in~\eqref{test-case:3}
with $m=n=1$ and $\omega=\sqrt{2}$, long time behaviour.}
    \label{tab:error_estimator_long_time}
\begin{tabular}{|c|cccc|}
\hline
error (rate) & \multicolumn{4}{c|}{$p=q=2$}                                                                                                         \\ \hline
$T$          & \multicolumn{1}{c|}{$L^\infty(L^2)$} & \multicolumn{1}{c|}{$L^2(H^1)$}     & \multicolumn{1}{c|}{jump of $H^1(L^2$)} & $\eta$        \\ \hline
6            & \multicolumn{1}{c|}{1.12e-1}         & \multicolumn{1}{c|}{5.41e-1}        & \multicolumn{1}{c|}{2.06e0}             & 3.16e0        \\ \hline
8            & \multicolumn{1}{c|}{1.46e-1 (0.93)}  & \multicolumn{1}{c|}{8.24e-1 (1.46)} & \multicolumn{1}{c|}{2.32e0 (0.42)}      & 4.18e0 (0.97) \\ \hline
10           & \multicolumn{1}{c|}{1.84e-1 (1.01)}  & \multicolumn{1}{c|}{1.14e0 (1.47)}  & \multicolumn{1}{c|}{2.58e0 (0.47)}      & 5.23e0 (1.00) \\ \hline
\end{tabular}
\begin{tabular}{|c|cccc|}
\hline
error (rate) & \multicolumn{4}{c|}{$p=q=3$}                                                                                                          \\ \hline
$T$          & \multicolumn{1}{c|}{$L^\infty(L^2)$} & \multicolumn{1}{c|}{$L^2(H^1)$}     & \multicolumn{1}{c|}{jump of $H^1(L^2$)} & $\eta$         \\ \hline
6            & \multicolumn{1}{c|}{1.32e-3}         & \multicolumn{1}{c|}{7.96e-3}        & \multicolumn{1}{c|}{1.97e-1}            & 5.79e-2        \\ \hline
8            & \multicolumn{1}{c|}{1.69e-3 (0.85)}  & \multicolumn{1}{c|}{1.03e-2 (0.89)} & \multicolumn{1}{c|}{2.29e-1 (0.52)}     & 7.78e-2 (1.03) \\ \hline
10           & \multicolumn{1}{c|}{2.07e-3(0.89)}   & \multicolumn{1}{c|}{1.30e-2 (1.04)} & \multicolumn{1}{c|}{2.58e-1 (0.52)}     & 9.94e-2 (1.09) \\ \hline
\end{tabular}
\end{table}

From Table~\ref{tab:error_estimator_long_time},
we observe a linear dependence on the final time for the $L^\infty(L^2)$-norm
and the estimator~$\eta$;
the $L^2$-in-time type norms have slightly worse dependence
for $p=q=2$; but the same behaviour for $p=q=3$;
the jump error has half an order dependence.

\subsubsection{Higher oscillating modes: test case 3}
Here we investigate the influence on the performance of the scheme
in presence higher oscillating modes in the exact solution.
We still consider the test case~3 with exact solution in~\eqref{test-case:3}
with parameters $m=n=10$ and $\omega=10\sqrt{2}$.
We fix $T=1$, $\pnt=\px$ in $\{2,3,4,5\}$, with $h=\taubold$ in $\{ 2\times 10^{-1}, 10^{-1}, 6.67\times 10^{-2}, 5\times 10^{-2}, 4\times 10^{-2}, 3.33\times 10^{-2}, 2.86\times 10^{-2}\}$.
In Figure~\ref{fig:robust_test}, we display the errors in the $L^\infty(L^2)$-norm;
other norms mentioned in the above sections have also been tested
and omitted here for brevity since they have a similar behaviour
to the $L^\infty(L^2)$-norm.

\begin{figure}[htb]
    \centering
    \includegraphics[width=0.5\linewidth]{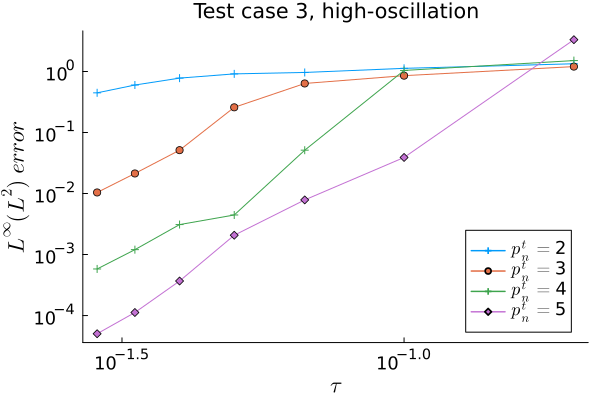}
    \caption{Exact solution as in \eqref{test-case:3}
    with parameters $m=n=10$ and $\omega=10\sqrt{2}$: polynomial degrees VS errors.}
    \label{fig:robust_test}
\end{figure}

From Figure~\ref{fig:robust_test}, we observe that higher polynomial degrees
reduce the pre-asymptotic regime.

\subsection{Efficiency of the error estimator for the semi-discrete in time scheme} \label{subsection:nr-efficiency}
Henceforth, the semi-discrete
formulation~\eqref{Walkington:method-time-semi-discrete} is considered.
We check the efficiency of the error estimator~$\eta$ in~\eqref{error-estimator}
(with $m=N$ and $\xi=t_N$ in~\eqref{eta2-oscf}),
compared to the error measured in the $L^\infty(0,T;L^2(\Omega))$ norm.
In particular, we investigate the behaviour
of the effectivity index
\begin{equation} \label{effectivity-indices}
\kappa:= \frac{\eta}{\Norm{\eh}_{L^\infty(0,T;L^2(\Omega))}}.
\end{equation}
We focus on the test cases with exact solutions
in~\eqref{test-case:1} and~\eqref{test-case:2},
and consider here uniform time steps and polynomial in time refinements.

In Figure~\ref{fig:estimator_tau_refine},
we present the results we obtained under uniform time steps refinements.
For the test case with exact solution as in~\eqref{test-case:1},
we pick $\pnt$ in $\{2,3,4\}$;
for the test case with exact solution as in~\eqref{test-case:2},
we pick $\pnt=2$, $\alpha=1.75$.
For the test case with exact solution as in~\eqref{test-case:3},
we pick $m=n=10$, $\omega=10\sqrt{2}$, $\px=4$, $h=1.33\times 10^{-1}$
and $\pnt$ in $\{2,3,4\}$.

\begin{figure}[htb]
    \centering
    \includegraphics[width=0.48\linewidth]{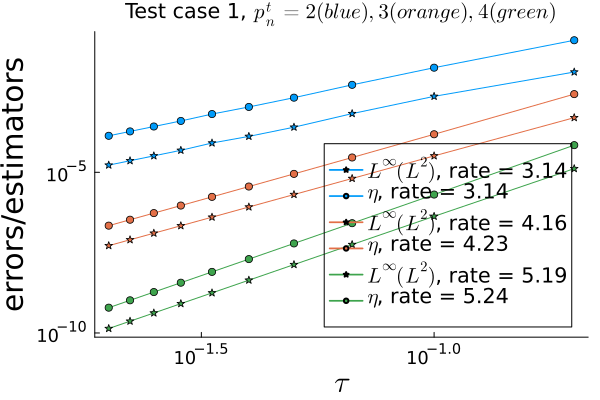}
    \includegraphics[width=0.48\linewidth]{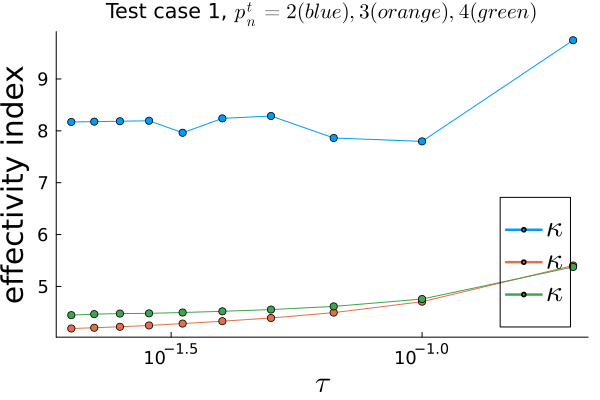}
    \includegraphics[width=0.48\linewidth]{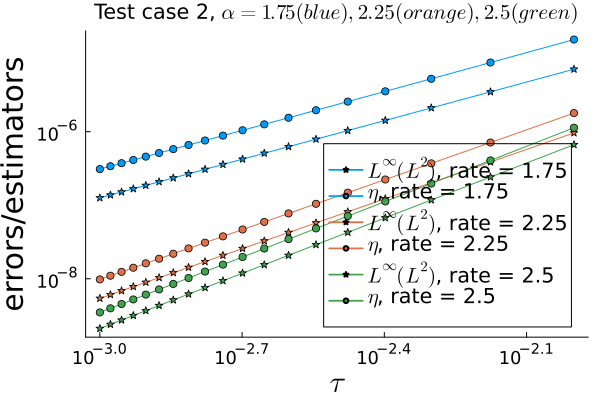}
    \includegraphics[width=0.48\linewidth]{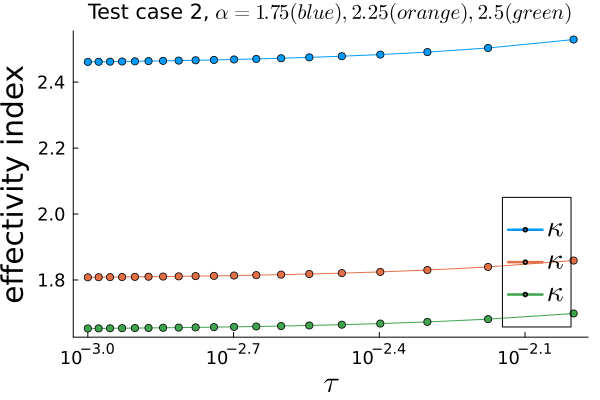}
    \includegraphics[width=0.48\linewidth]{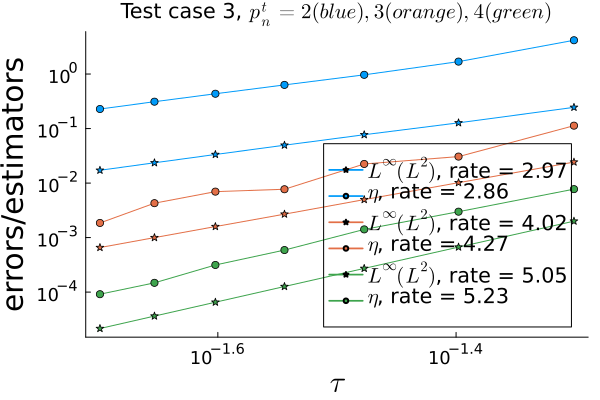}
    \includegraphics[width=0.48\linewidth]{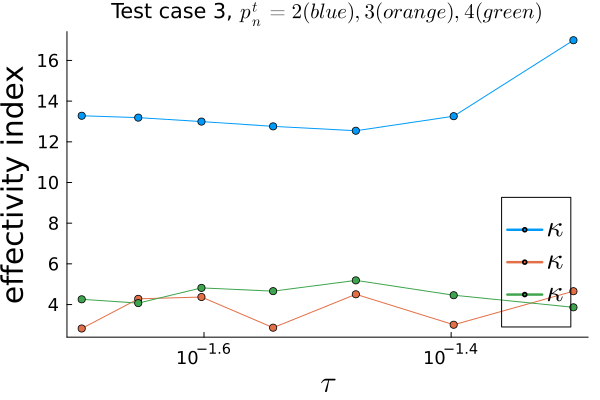}
    \caption{Exact solution as in~\eqref{test-case:1}, \eqref{test-case:2}, and~\eqref{test-case:3}
    uniform $\taubold$-refinement.}
    \label{fig:estimator_tau_refine}
\end{figure}

The estimator has the optimal convergence rate
as the error measured in the $L^\infty(0,T;L^2(\Omega))$ norm.
Notably, the effectivity index in~\eqref{effectivity-indices}
seems stable with respect to~$\taubold$,
i.e., is uniformly bounded by a constant with respect to~$\taubold$.
\medskip

Then, in Figure \ref{fig:estimator_pnt_refine}, the same tests are investigated
with $N=5$, i.e., for a fixed $\tau=0.2$
under $\pnt$-refinements in time;
for the test case with exact solution in~\eqref{test-case:2}
we only consider~$\alpha=1.75$.

\begin{figure}[htb]
    \centering
    \includegraphics[width=0.48\linewidth]{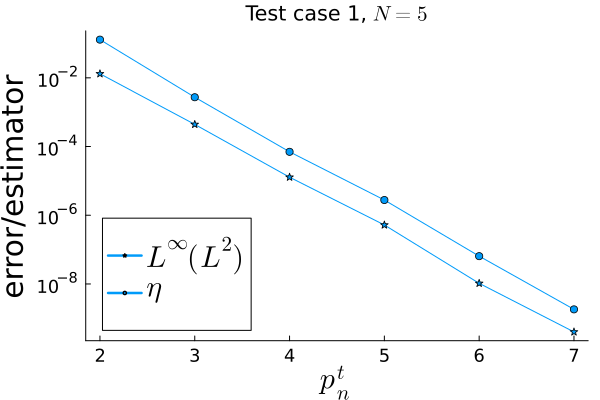}
    \includegraphics[width=0.48\linewidth]{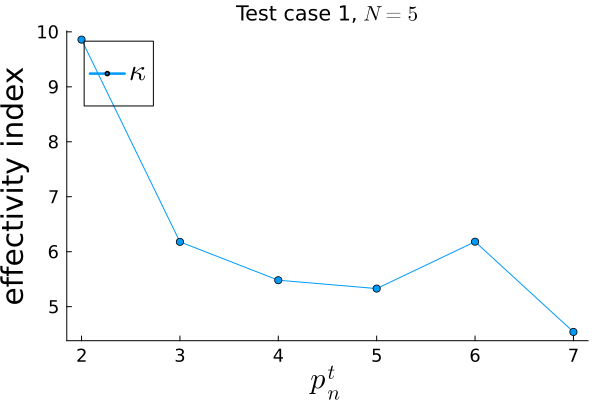}
    \includegraphics[width=0.48\linewidth]{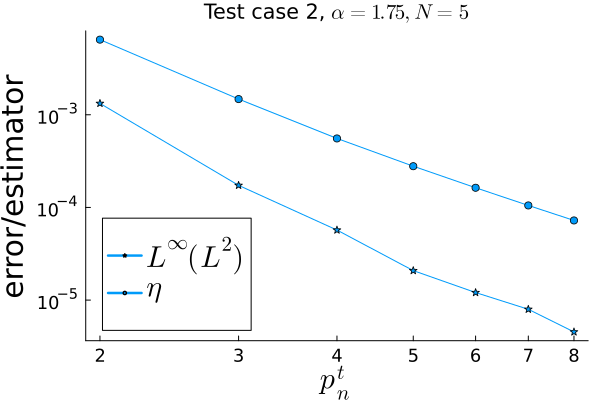}
    \includegraphics[width=0.48\linewidth]{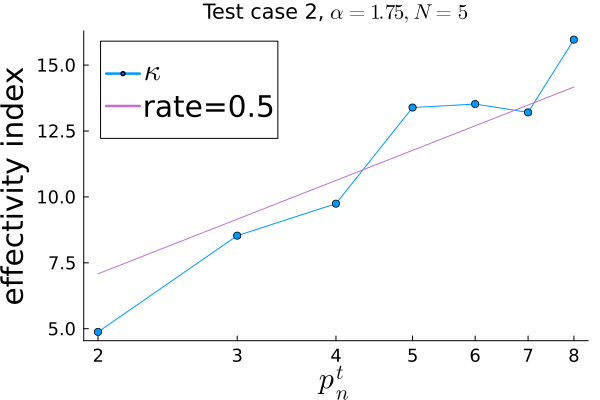}
    \includegraphics[width=0.48\linewidth]{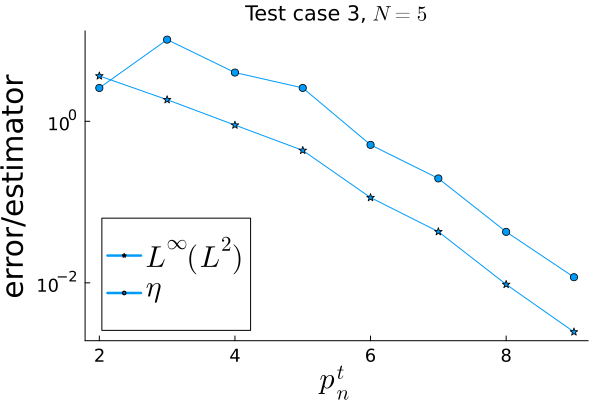}
    \includegraphics[width=0.48\linewidth]{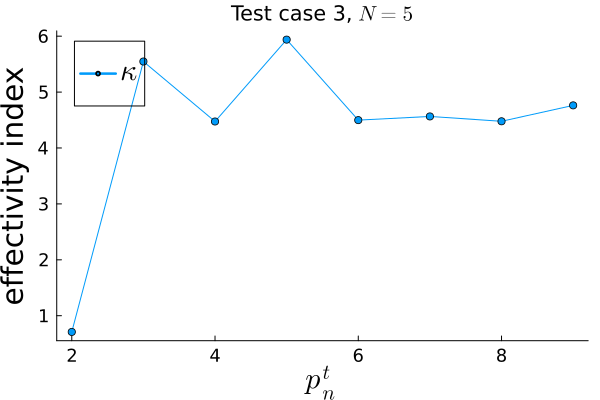}
    \caption{Exact solution as in~\eqref{test-case:1}, \eqref{test-case:2} and~\eqref{test-case:3}, $\pnt$-refinement.}
    \label{fig:estimator_pnt_refine}
\end{figure}

Also in this case, the estimator has the same convergence rate
as that of the $L^\infty(0,T;L^2(\Omega))$ norm of the error.
For the test case with exact solution in~\eqref{test-case:1} and~\eqref{test-case:3}, the effectivity index~$\kappa$ is uniformly bounded
in terms of~$\pnt$;
for the test case with exact solution in~\eqref{test-case:2},
$\kappa$ increases with rate~$1/2$ in terms of~$\pnt$.

\subsection{Adaptive refinements in time}\label{subsection:adaptive_test}
Since the data oscillation terms in the a posteriori error estimates
\eqref{reliability:estimate} are
not dominant, we omit them for simplicity.
We consider here an adaptive algorithm
with the usual structure
\[
\textbf{SOLVE}
\qquad \Longrightarrow \qquad
\textbf{ESTIMATE}
\qquad \Longrightarrow \qquad
\textbf{MARK}
\qquad \Longrightarrow \qquad
\textbf{REFINE}.
\]
The \textbf{ESTIMATE} step is driven by using the error estimator~$\eta$;
to this aim, we propose an algorithm for the localization
of~$\eta$, notably to determine a practical value of~$m$
in~\eqref{eta2-oscf}.
As for the \textbf{MARK} step, we use D\"orfler’s marking
(with $\ell^1$-type summation)
with a given threshold $\theta$ in~$(0,1]$.
The \textbf{REFINE} step is realized by the bisection of the marked time intervals.

In what follows, we denote the Kronecker delta function by~$\delta_{i,j}$.

\paragraph{Adaptive algorithm for the localization of each ESTIMATE step.}
\begin{enumerate}
    \item Find the interval index~$m$ where $\eta_1$ attains the maximum.
    \item Compute 
    \[
    \eta_{2,n} :=
    \begin{cases}
    \frac{2}{\pi}  \Big( \taun c_3(\pnt-1)
        \Norm{\Deltax(\U-\Pizpbftmo  \U)}_{L^1(\In;L^2(\Omega))} \\
\qquad\qquad  + \taun^3 c_2(\pnt) c_4(\pnt-3)
        \Norm{\jump{ \Deltax \U'}(\tnmo,\cdot)}_{0,\Omega} \Big)        
            & \text{for all } n=1,\dots,m \\
    0 & \text{otherwise,}
    \end{cases}
    \]
    and $\eta_2$ by~\eqref{eta2-oscf} with the above given $m$.
    \item Denote $\eta_n = \delta_{m,n}\eta_1 + \eta_{2,n}$ 
    the local error estimator
    on the time interval~$\In$ for $n=1,\dots,N$.
\end{enumerate}

\subsubsection{Numerical results: the adaptive algorithm}
We consider the test case with exact solution as in~\eqref{test-case:2}
and~$\alpha=1.75$; $\pnt$ in $\{2,3,4\}$; $\px=2$;
$5$ nodes in each space direction;
D\"orfler’s marking parameter $\theta=0.5$.
We define $DoFs =N \times \pnt \times \card(\Vh)$.
In Figure~\ref{fig:adaptive}, we display the
$L^\infty(L^2)$ uniform and adaptive errors
and estimators~$\eta$ in the left panels;
the effectivity indices~$\kappa$
are presented in the right panels.

\begin{figure}[htb]
    \centering
    \includegraphics[width=0.48\linewidth]{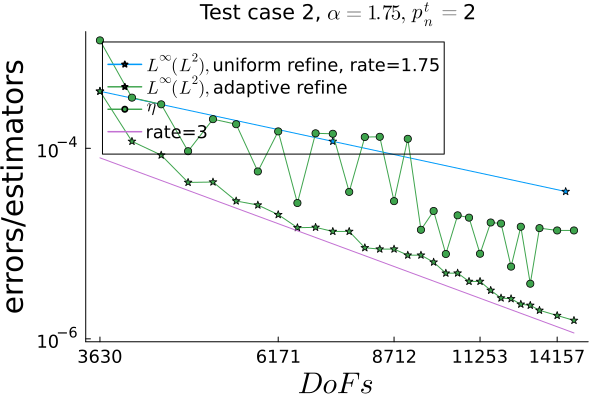}
    \includegraphics[width=0.48\linewidth]{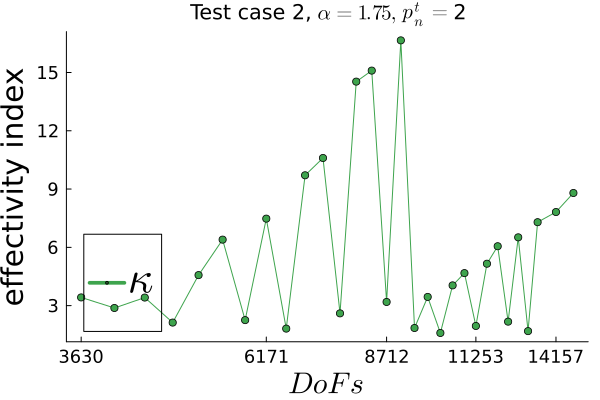}
    \includegraphics[width=0.48\linewidth]{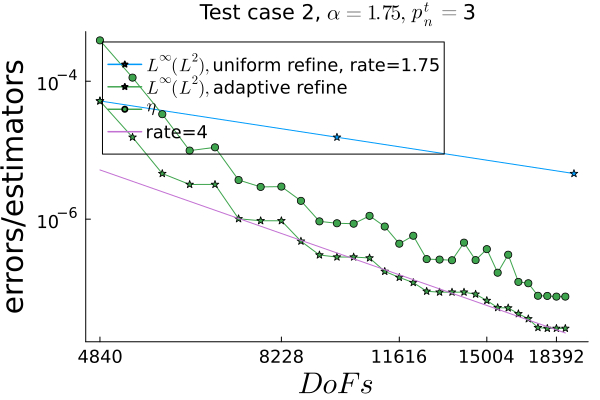}
    \includegraphics[width=0.48\linewidth]{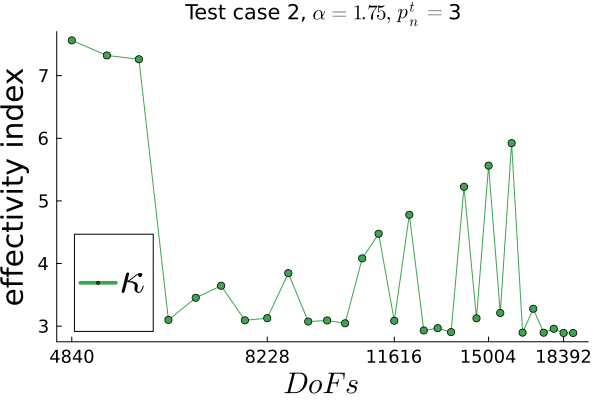}
    \includegraphics[width=0.48\linewidth]{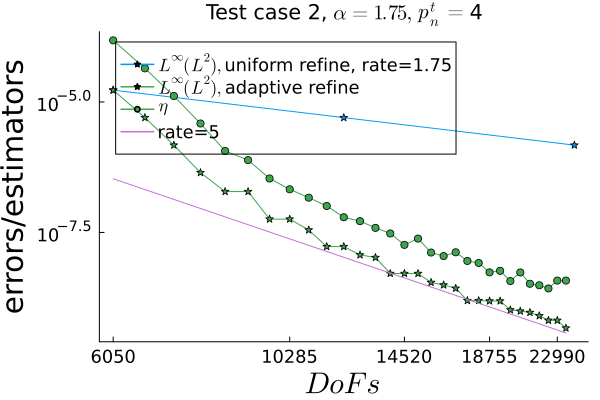}
    \includegraphics[width=0.48\linewidth]{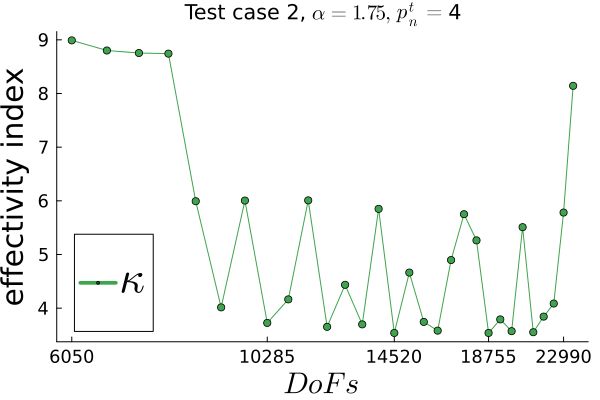}
    \caption{Exact solution as in~\eqref{test-case:2}.}
    \label{fig:adaptive}
\end{figure}

Some remarks for this test case are in order:
\begin{itemize}
    \item the adaptive algorithm delivers optimal convergence rate in terms of the number $DoFs$ of the method;
    \item the effectivity index is uniformly bounded for fixed~$\pnt$,
    and the adaptive algorithm asymptotically returns smaller effectivity indices.
\end{itemize}

In Figure \ref{fig:adaptive_mesh},
we illustrate the final time meshes
produced with~$\pnt=2$ and~$4$ from the
adaptive and uniform refinement algorithms.

\begin{figure}
    \centering
    \includegraphics[width=0.48\linewidth]{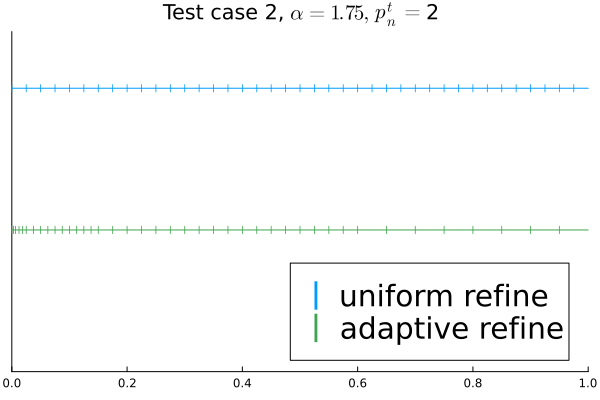}
    \includegraphics[width=0.48\linewidth]{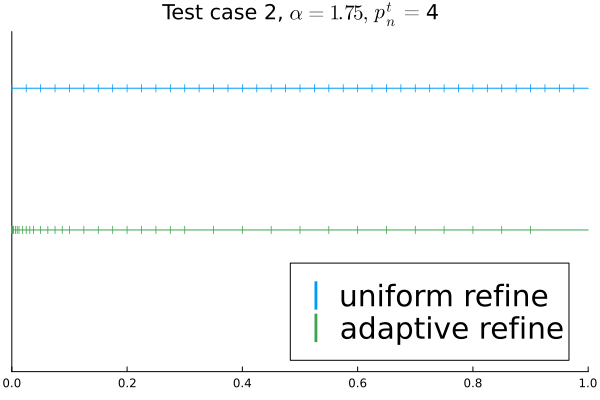}
    \caption{Exact solution as in~\eqref{test-case:2},
    time mesh visualization.}
    \label{fig:adaptive_mesh}
\end{figure}

The adaptive algorithm generates a time mesh
with a strong grading towards the initial time.
Such temporal meshes are more graded
in the case $\pnt=4$ compared to the case $\pnt=2$.

\section{Conclusions} \label{section:conclusions}
For a $\mathcal C^0$-in-time discretization of the wave equation
in second order formulation
\begin{itemize}
    \item we derived a priori estimates
    for the fully-discrete method, which are
    explicit in the spatial mesh size,
    the time steps, and the polynomial degree 
    distributions in space and time
    (the errors were measured in $L^\infty$-type
    norms in time);
    \item we derived fully explicit, reliable
    a posteriori estimates for the $L^\infty(L^2)$
    error in terms of a novel error estimator,
    which involves jumps of the time derivatives
    and the spatial Laplacian at the time nodes. 
\end{itemize}
A campaign of numerical experiments revealed that
\begin{itemize}
    \item for sufficiently smooth solutions, the order of convergence
    under uniform refinements
    is optimal also for $L^2$-type errors
    in time;
    \item the proposed error estimator
    is efficient under time step refinements
    and may be inefficient under $\p$-refinements
    in time.
\end{itemize}
More recently~\cite{Dong-Georgoulis-Mascotto-Wang:2025},
we analyzed and assessed the performance of a fully discrete
scheme for the wave equation in 2nd order formulation
with dynamic mesh change.
A crucial question remains
the proof of a localized lower bound for the error estimator
in~\eqref{error-estimator},
which, to the best of our knowledge,
has been an open problem for at least the last two decades
for the error measures considered in this work;
lower bounds have recently been established
at least for other choices of norm,
cf., e.g., \cite{ChaumontFrelet:2023, ChaumontFrelet-Ern:2025}.
Another challenge is the proof of convergence
and optimality of an adaptive algorithm
driven by that error estimator;
recent results in~\cite{Feischl:2022, Feischl-Niederkofer:2025}
show that this can be accomplished for parabolic problems
without mesh change in the energy norm,
while no similar results are currently available for the wave equation.

\paragraph{Acknowledgements.}
ZD has been partially funded by the French National Research Agency (ANR, STEERS, project number ANR-24-CE56-0127-01). LM has been partially funded by the European Union
(ERC, NEMESIS, project number 101115663);
views and opinions expressed are however those of the author
only and do not necessarily reflect those of the EU or the ERC Executive Agency.
LM has been partially funded by MUR (PRIN2022 research grant n. 202292JW3F).
LM is also member of the Gruppo Nazionale
Calcolo Scientifico-Istituto Nazionale di Alta Matematica (GNCS-INdAM).

{\footnotesize \bibliography{bibliogr}} \bibliographystyle{plain}

\end{document}